\documentclass[a4paper,11pt]{amsart}
\usepackage{amssymb,amscd,amsxtra,xypic}
\usepackage[all]{xy}

\usepackage[pdftex]{hyperref}
\hypersetup{%
bookmarksnumbered=true,%
colorlinks=true,%
setpagesize=false,%
pdftitle={},%
pdfauthor={Takuzo Okada}}

\theoremstyle{plain}
\newtheorem{Thm}{Theorem}[section]
\newtheorem{Lem}[Thm]{Lemma}
\newtheorem{Cor}[Thm]{Corollary}
\newtheorem{Prop}[Thm]{Proposition}

\theoremstyle{definition}
\newtheorem{Def}[Thm]{Definition}
\newtheorem{Def-Lem}[Thm]{Definition-Lemma}
\newtheorem{Cond}[Thm]{Condition}
\newtheorem{Rem}[Thm]{Remark}

\newtheorem*{Ack}{Acknowledgments}
\newtheorem{Ex}[Thm]{Example}

\newcommand{\prt}{\partial}
\newcommand{\Sing}{\operatorname{Sing}}

\newcommand{\rank}{\operatorname{rank}}

\newcommand{\Pic}{\operatorname{Pic}}

\newcommand{\Bs}{\operatorname{Bs}}

\newcommand{\Exc}{\operatorname{Exc}}
\newcommand{\mult}{\operatorname{mult}}

\newcommand{\wt}{\operatorname{wt}}

\newcommand{\Wbl}{\operatorname{Wbl}}

\newcommand{\NQsm}{\operatorname{NQsm}}

\newcommand{\mbA}{\mathbb{A}}

\newcommand{\mbC}{\mathbb{C}}

\newcommand{\mbP}{\mathbb{P}}
\newcommand{\mbQ}{\mathbb{Q}}

\newcommand{\mbZ}{\mathbb{Z}}

\newcommand{\mcG}{\mathcal{G}}
\newcommand{\mcH}{\mathcal{H}}
\newcommand{\mcI}{\mathcal{I}}

\newcommand{\mcL}{\mathcal{L}}
\newcommand{\mcM}{\mathcal{M}}

\newcommand{\mcO}{\mathcal{O}}

\newcommand{\msp}{\mathsf{p}}
\newcommand{\msq}{\mathsf{q}}

\newcommand{\ratmap}{\dashrightarrow}

\makeatletter
\def\imod#1{\allowbreak\mkern10mu({\operator@font mod}\,\,#1)}
\makeatother

\title[$\mathbb{Q}$-Fano weighted complete intersections, III]{Birational Mori fiber structures of  $\mathbb{Q}$-Fano 3-fold weighted complete intersections, III}
\author[Takuzo Okada]{Takuzo Okada}
\address{Department of Mathematics, Faculty of Science and Engineering\endgraf
Saga University, Saga 840-8502 Japan}
\email{okada@cc.saga-u.ac.jp}
\date{}

\begin{document}

\begin{abstract}
This is a continuation of a series of papers studying the birational Mori fiber structures of anticanonically embedded $\mbQ$-Fano $3$-fold weighted complete intersections of codimension $2$.
We have proved that $19$ families consists of birationally rigid varieties and $14$ families consists of birationally birigid varieties.
The aim of this paper is to continue the work systematically and prove that, among the remaining $52$ families, $21$ families consist of birationally birigid varieties.
\end{abstract}

\maketitle


\section{Introduction} \label{sec:intro}

This is a continuation of a series of papers studying birational Mori fiber structures of anticanonically embedded $\mathbb{Q}$-Fano $3$-fold weighted complete intersections of codimension $2$ ($\mbQ$-Fano WCIs of codimension $2$, for short).
Here, a Mori fiber space that is birational to a given algebraic variety is called a {\it birational Mori fiber structure} of the variety.
There are $85$ families of $\mbQ$-Fano $3$-fold WCIs of codimension $2$ (see \cite{IF}).
In \cite{Okada1}, we divide $85$ families into the disjoint union of $3$ pieaces $I := \{1,2,\dots,85\} = I_{br} \cup I_{dP} \cup I_F$, where $|I_{br}| = 19$, $|I_{dP}| = 6$ and $|I_F| = 60$, and studied their birational (non-)rigidity.
We proved birational rigidity of general members of family No.~$i$ with $i \in I_{br}$, and among other things, we also proved that a member of family No.~$i$ with $i \in I_{dP}$ and $i \in I_F$ is birational to a del Pezzo fiber space over $\mbP^1$ and to a $\mbQ$-Fano $3$-fold, respectively.
Recently, Ahmadinezhad and Zucconi \cite{AZ} succeeded in removing generality conditions for family No.~$i$ with $i \in I_{br}$ and proved birational rigidity of every quasismooth member.
From now on, we focus on families No.~$i$ with $i \in I_F$ and recall the following result for such families.

\begin{Thm}[{\cite[Theorem 1.3]{Okada1}}] \label{thmG}
Let $X$ be a general member of family No.~$i$ with $i \in I_F$.
Then there exists a $\mbQ$-Fano $3$-fold weighted hypersurface $X'$ that is birational but not isomorphic to $X$.
Moreover, each maximal singularity on $X$ is untwisted by a Sarkisov link that is either a birational involution or a link to $X'$.
\end{Thm}

This is an important step toward the determination of the birational Mori fiber structures of $X$.
In fact, by continuing the work, we proved in \cite{Okada2} birational birigidity of members of family No.~$i$ with $i \in I_{cA/x} \cup I_{cAx/4}$ for suitable subsets $I_{cAx/2}$ and $I_{cAx/4}$ of $I_F$.
Here, {\it birational birigidity} means that there exist exactly two Mori fiber spaces in the birational equivalence class.

Note that, in Theorem \ref{thmG}, the weighted hypersurface $X'$ is uniquely determined by $X$ and we call it the {\it birational counterpart} of $X$.
It has a unique non-quotient terminal singular point together with some other terminal quotient singular points.
We define subsets $I_{cA/n}$ and $I_{cD/3}$ of $I_F$ by the following rule: we have $i \in I_{cA/n}$ (resp.\ $i \in I_{cD/3}$) if and only if the non-quotient terminal singular point of the birational counterpart of a general member $X$ of family No.~$i$ is of type $cA/n$ (resp. $cD/3$).
Here, we regard $cA/1$ as $cA$.
Specifically, those subsets are give as
\[
\begin{split}
I_{cA/n} &= \{3,6,7,9,10,16,18,21,22,26,28,33,36,38,44,48,52,57,63\}, \\
I_{cD/3} &= \{61,62\}.
\end{split}
\]
Family No.~$3$ has been studied in this direction: Corti and Mella \cite{CM} proved birational birigidity of a general $\mbQ$-Fano WCI $X = X_{3,4} \subset \mbP (1,1,1,1,2,2)$ of cubic and quartic.
Note that $3 \in I_{cA/n} \subset I_F$.
The aim of this paper is to complete the determination of Mori fiber structures of families No.~$i$ with $i \in I^*_{cA/n} \cup I_{cD/3}$, where $I^*_{cA/n} := I_{cA/n} \setminus \{3\}$.
We state the main theorem and its direct consequence.

\begin{Thm} \label{mainthm}
Let $X$ be a general member of family No.~$i$ with $i \in I_{cA/n} \cup I_{cD/3}$.
Then $X$ is birationally birigid.
More precisely, $X$ is birational to a $\mbQ$-Fano $3$-fold weighted hypersurface $X'$ and is not birational to any other Mori fiber space.
\end{Thm}

\begin{Cor}
A general member of family No.~$i$ is not rational for $i \in I_{cA/n} \cup I_{cD/3}$.
\end{Cor}

\begin{Ack}
The author is partially supported by JSPS KAKENHI Grant Number 26800019.
\end{Ack}

\section{Preliminaries} \label{sec:prelim}

\subsection{Notation and convention}

Throughout the paper, we work over the field $\mbC$ of complex numbers.
A normal projective variety $X$ is said to be a $\mbQ$-{\it Fano variety} if $-K_X$ is ample, it is $\mbQ$-factorial, has only terminal singularities and its Picard number is $1$.
We say that an algebraic fiber space $X \to S$ is a {\it Mori fiber space} if $X$ is a normal projective $\mbQ$-factorial variety with at most terminal singularities, $\dim S < X$, $-K_X$ is relatively ample over $S$ and the relative Picard number is $1$.

Let $X$ be a normal projective $\mbQ$-factorial variety, $\mcH$ a linear system on $X$, $D \subset X$ a Weil divisor and $C \subset X$ a curve.
We say that $\mcH$ is $\mbQ$-{\it linearly equivalent} to $D$, denoted by $\mcH \sim_{\mbQ} D$, if a member of $\mcH$ is $\mbQ$-linearly equivalent to $D$.
We define $(\mcH \cdot C) := (H \cdot C)$ for $H \in \mcH$.
In this paper, a {\it divisorial extraction} (or {\it divisorial contraction}) is always an extremal divisorial extraction (or contraction) in the Mori category.

A weighted projective space (WPS) $\mbP (a,b,c,d,e)$ with homogeneous coordinates $x,y,z,s,t$ is sometimes denoted by $\mbP (a_x,b_y,c_z,d_s,e_t)$.
A closed subscheme $Z$ in $\mbP (a_0,\dots,a_n)$ is {\it quasismooth} (resp.\ {\it quasismooth outside} a point $\msp$) if the affine cone $C_Z \subset \mbA^{n+1}$ is smooth outside the origin (resp. outside the closure of the inverse image of $\msp$ via the morphism $\mbA^{n+1} \setminus \{o\} \to \mbP (a_0,\dots,a_n)$).
For $i = 0,1,\dots,n$, we denote by $\msp_i$ the vertex $(0 \!:\! \cdots \!:\! 1 \!:\! \cdots \!:\! 0)$ of $\mbP (a_0,\dots,a_n)$, where the $1$ is in the $(i+1)$-th position.  

The main object of this paper is a weighted hypersurface $X'$ birational to a member $X$ of family No.~$i$ with $i \in I^*_{cA/n} \cup I_{cD/3}$.
Let $\mbP := \mbP (a_0,\dots,a_4)$ be the ambient WPS of $X'$.
We write $x_0,\dots,x_3,w$ (resp.\ $x,y,z,t, \dots,w$) for the homogeneous coordinates when we treat several families at a time (resp.\ a specific family).
For example, we write $x_0,x_1,y,z,w$ (resp.\ $x,y,z,t,w$) for the coordinates of $\mbP (1,1,2,3,2)$ (resp.\ $\mbP (1,2,3,5,4)$).
The coordinate $w$ is distinguished so that the vertex at which only the coordinate $w$ is non-zero is the unique $cA/n$ or $cD/3$ point of $X'$.
For homogeneous polynomials $f_1,\dots,f_m$ in the variables $x_0,\dots,x_3,w$ or $x,y,z,\dots,w$, we denote by $(f_1 = \cdots = f_m = 0)$ the closed subscheme of $\mbP$ defined by the homogeneous ideal $(f_1,\dots,f_m)$ and denote by $(f_1 = \cdots = f_m = 0)_{X'}$ the scheme-theoretic intersection $(f_1 = \cdots = f_m = 0) \cap X'$.
For a polynomial $f = f (x_0,\dots,x_3,w)$ and a monomial $x_0^{x_0} \cdots x_3^{c_3} w^{d}$, we write $x_0^{c_0} \cdots x_3^{c_3} w^c \in f$ (resp.\ $x_0^{c_0} \cdots x_3^{c_3} w^c \notin f$) if the coefficient of the monomial in $f$ is non-zero (resp. zero).

A {\it weighted complete intersection curve} ({\it WCI curve}) of type $(c_1,c_2,c_3)$ (resp.\ of type $(c_1,c_2,c_3,c_4)$) in $\mbP (a_0,\dots,a_4)$ (resp.\ $\mbP (a_0,\dots,a_5)$) is an irreducible and reducible curve defined by three (resp.\ four) homogeneous polynomials of degree $c_1,c_2$ and $c_3$ (resp.\ $c_1,c_2,c_3$ and $c_4$).

\subsection{Weighted blowup}

We fix notation on weighted blowups of cyclic quotients of affine spaces.

Let $\mbA := \mbA^n$ be the affine $n$-space with affine coordinates $x_1,\dots,x_n$ and let $b_1,\dots,b_b$ be positive integers.
Let $\Phi \colon \mbA \ratmap \mbP := \mbP (b_1,\dots,b_n)$ the rational map defined by $(\alpha_1,\dots,\alpha_n) \mapsto (\alpha_1 \!:\! \cdots \!:\! \alpha_n)$ and $\Wbl (\mbA) \subset \mbA \times \mbP$ the graph of $\Phi$.
Suppose that we are given a $\mbZ_r$-action on $\mbA$ of type $\frac{1}{r} (a_1,\dots,a_n)$. 
We assume that $b_i \equiv a_i \imod{r}$ for every $i$.
Let $V$ be the quotient of $\mbA$ by the $\mbZ_r$-action and $X$ a subvariety of $V$ through the origin.
The rational map $\Phi$ descends to a rational map $\Phi_V \colon V \ratmap \mbP$.
We define $\Wbl (V)$ (resp.\ $\Wbl (X)$) to be the graph of $\Phi_V$ (resp.\ $\Phi_V|_X$) and call the projection $\varphi_V \colon \Wbl (V) \to V$ (resp.\ $\varphi_X \colon \Wbl (X) \to X$) the {\it weighted blowup} of $V$ (resp.\ $X$) at the origin with $\wt (x_1,\dots,x_n) = \frac{1}{r} (b_1,\dots,b_n)$.
This weight is referred to as the $\varphi_V$-{\it weight} or the $\varphi_X$-{\it weight}. 
We consider the $\mbZ_r$-action on $\mbA \times \mbP$ which acts on $\mbA$ as above and on $\mbP$ trivially.
We have a natural morphism $\Wbl (\mbA) \to \Wbl (V)$ and from this we can see $\Wbl (V)$ as the quotient of $\Wbl (\mbA)$ by the $\mbZ_r$-action.

We explain orbifold charts of $\Wbl (V)$ and $\Wbl (X)$.
Let $X_1,\dots,X_n$ be the homogeneous coordinates of $\mbP$.
For each $i$, we define $\Wbl_i (\mbA)$ (resp.\ $\Wbl_i (V)$) to be the open subset of $\Wbl (\mbA)$ (resp.\ $\Wbl (V)$) which is the intersection of $\Wbl (\mbA)$ (resp.\ $\Wbl (V)$) and the open subset $\mbA \times (X_i \ne 0)$ (resp.\ $V \times (X_i \ne 0)$).
Let $U_i (\mbA)$ be an affine $n$-space with affine coordinates $\tilde{x}_1,\dots,\tilde{x}_{i-1}, x'_i, \tilde{x}_{i+1}, \dots,\tilde{x}_n$ and define a morphism $U_i (\mbA) \to \mbA$ by the identification $x_i = {x_i'}^{b_i}$ and $x_j = \tilde{x}_j {x_i'}^{b_j}$ for $j \ne i$.
We consider the $\mbZ_{b_i}$-action on $U_i (\mbA)$ of type $\frac{1}{b_i} (b_1,\dots,b_{i-1},-1,b_{i+1},\dots,b_n)$.
We see that the quotient $U_i (\mbA)/\mbZ_{b_i}$ is naturally isomorphic to $\Wbl_i (\mbA)$ and the section $x'_i$ cuts out the open subset $(X_i \ne 0)$ of the exceptional divisor $\mbP$.
Note that $\Wbl_i (V)$ is the quotient of $\Wbl_i (\mbA)$ by the $\mbZ_r$-action.
The $\mbZ_r$-action on $\Wbl_i (\mbA)$ is given by $x'_i \mapsto \zeta_r x'_i$ and $\tilde{x}_j \mapsto \tilde{x}_j$ for $j \ne i$.
Let $U_i (V)$ be the affine $n$-space with affine coordinates $\tilde{x}_1,\dots,\tilde{x}_n$.
The identification $\tilde{x}_i = {x_i'}^r$ defines a morphism $U_i (\mbA) \to U_i (V)$ and we see that $\Wbl_i (V)$ is the quotient of $U_i (V)$ by the $\mbZ_{b_1}$-action on $U_i (V)$ of type $\frac{1}{b_i} (b_1,\dots,b_{i-1}, -r,b_{i+1},\dots,b_n)$.
We call $U_i (V)$ the {\it orbifold chart} of $\Wbl_i (V)$ and call $\tilde{x}_1,\dots,\tilde{x}_n$ {\it orbifold coordinates} of $\Wbl_i (V)$.

We keep the above setting and let $X$ be a closed subvariety of $V$ which is a complete intersection defined by $\mbZ_r$-semi-invariant polynomials $f_1,\dots,f_k \in \mbC [x_1,\dots,x_n]$.
Let $m_i/r$ be the order of $f_i$ with respect to the $\varphi_X$-weight and we define $g_i$ (resp.\ $h_i$) to be the weight $= m_i/r$ part (resp.\ the weight $= m_i/r + 1$ part) of $f_i$.
We define
\[
f_j^{(i)} := f (\tilde{x}_1,\dots, \tilde{x}_i^{1/r},\dots,\tilde{x}_n)/\tilde{x}_i^{m_i} \in \mbC [\tilde{x}_i,\dots,\tilde{x}_n]
\]
and
\[
U_i (X) := \left( f_1^{(i)} = \cdots = f_k^{(i)} = 0 \right) \subset U_i (V).
\]
Then, the quotient $\Wbl_i (X) := U_i (X)/\mbZ_{b_i}$ is an open subset of $\Wbl (X)$.
We call $U_i (X)$ the {\it orbifold chart} of $\Wbl (X)$. 
Assume that $(g_1 = \cdots = g_k = 0) \subset \mbA^n$ is a local complete intersection at the origin.
Then, we have an isomorphism
\[
E \cong (g_1 = g_2 = \cdots = g_k = 0) \subset \mbP (b_1,b_2,\dots,b_n),
\]
where $E$ is the exceptional divisor of the weighted blowup and, by a slight abuse of notation, $\mbP (b_1,\dots,b_n)$ is a weighted projective space with coordinates $x_1,\dots,x_n$.
We denote by $J_{C_E}$ the Jacobian matrix of the affine cone $C_E$ of $E$.
Note that 
\[
J_{C_E} =
\left( \frac{\prt g_i}{\prt x_j} \right)_{1 \le i \le k, 1 \le j \le n}
\] 
is a $k \times n$ matrix.
We define the matrix $J_{\varphi}$ to be the $k \times (n+1)$ matrix 
\[
J_{\varphi} := \left( J_{C_E} \ h \right),
\]
where $h := {}^t \! \left( h_1 \ h_2 \ \cdots \ h_k \right)$.

\begin{Lem} \label{lem:singwbl}
We keep the above setting. 
If $J_{\psi}$ is of rank $k$ at every point of $E$, then $\Wbl (X)$ has at most cyclic quotient singular points and 
\[
\Sing (\Wbl (X)) \cap E = \Sing (E) \cap \Sing (\mbP (b_1,\dots,b_n)).
\]
\end{Lem}

\begin{proof}
Let $J_{U_i (X)}$ be the Jacobian matrix of the orbifold chart $U_i (X) \subset \mbA^n$.
Let $\tilde{E}_i$ be the inverse image of $E \cap \Wbl_i (X)$ by $U_i (X) \to \Wbl_i (X)$.
For each $\msq \in \tilde{E}_i$, we have $\rank J_{U_i (X)} (\msq) =\rank J_{\varphi} (\msq)$, hence $\rank J_{U_i (X)} = k$ along $\tilde{E}_i$.
It follows that $U_i (X)$ is nonsingular along $\tilde{E}_i$ for each $i$.
This shows that $\Wbl (X)$ has at most cyclic quotient singular points.
Since $U_i (X)$ is nonsingular, the singularities of $\Wbl (X)$ come from the actions by cyclic groups.
The rest is immediate from this observation.
\end{proof}

\subsection{Generality conditions and the definition of families}

The following condition is introduced in \cite{Okada1}.

\begin{Cond} \label{cond1}
Let $X$ be a member of family No.~$i$ with $i \in I^*_{cA/n} \cup I_{cD/3}$.
\begin{enumerate}
\item[($\mathrm{C}_0$)] $X$ is quasismooth.
\item[($\mathrm{C}_1$)] The monomial in Table \ref{table:monomial} appears in one of the defining polynomial of $X$ with non-zero coefficient.
\item[($\mathrm{C}_2$)] $X$ does not contain any WCI curve listed in Table \ref{table:WCIcurve}.
\item[($\mathrm{C}_2$)] If $i = 21$, then $(x_0 = x_1 = 0)_X$ is an irreducible curve.
\item[($\mathrm{C}_4$)] If $i = 6$ (resp.\ $9$), then $X$ satisfies \cite[Lemma 7.3]{Okada1} (resp.\ \cite[Lemma 7.11]{Okada1}).
\end{enumerate}
\end{Cond}

\begin{table}[h]
\begin{center}
\caption{Monomials}
\label{table:monomial}
\begin{tabular}{cccc}
\hline
No. & Monomial & No. & Monomial \\
\hline
9 & $y z$ & 33 & $z s$ \\
22 & $z s$ & 48 & $z s$ \\
28 & $z s$ & 57 & $s t$ \\
\end{tabular}
\end{center}
\end{table}

\begin{table}[h]
\begin{center}
\caption{Type of WCI curves}
\label{table:WCIcurve}
\begin{tabular}{cccc}
\hline
No. & Type & No. & Type \\
\hline
6 & (1,1,1,2) & 33 & (1,1,5,6) \\
10 & (1,1,2,3) & 38 & (1,3,4,7) \\
16 & (1,1,2,3), (1,1,3,4), (1,1,3,6) & 44 & (1,2,3,5) \\
18 & (1,1,2,3) & 52 & (1,3,5,8) \\
22 & (1,1,4,5) & 57 & (1,2,7,9) \\
26 & (1,1,3,4) & 63 & (1,3,5,8)
\end{tabular}
\end{center}
\end{table} 

Let $X$ be a member of family No.~$i$ satisfying Condition \ref{cond1}.
Then, by \cite{Okada1}, there is a Sarkisov link to an anticanonically embedded $\mbQ$-Fano $3$-fold weighted hypersurface $X'$.
Note that $X'$ is uniquely determined by $X$ and we call $X'$ the {\it birational counterpart} of $X$.
Precise descriptions of birational counterparts will be given in the next subsection.

If $i \in I^*_{cA} \setminus \{7,10\}$ (resp.\ $i \in I_{cD/3}$), then the birational counterpart $X'$ of any member $X$ of family No.~$i$ satisfying Condition \ref{cond1} has a singularity of type $cA/n$ (resp.\ $cD/3$) at $\msp_4$.
This is not the case for families No.~$7$ and $10$.
We introduce the following additional generality condition which ensures that the unique non-quotient singular points of birational counterparts are of type $cA/n$.

\begin{Cond} \label{cond2}
If $i = 7$ (resp.\ $10$), then $\prt F_1/\prt z_0$ is not proportional to $\prt F_1/\prt z_1$, that is, there is no constant $\alpha$ such that $\prt F_1/\prt z_0 = \alpha \prt F_1/\prt z_1$, where $z_0,z_1$ are the coordinates of degree $3$ and $F_1$ is the defining polynomial of degree $4$ (resp.\ $5$).
\end{Cond}

\begin{Def}
For $i \in I^*_{cA/n} \cup I_{cD/3}$, we define $\tilde{\mcG}_i$ to be the subfamily of family No.~$i$ consisting of members satisfying Conditions \ref{cond1} and \ref{cond2}.
We define $\tilde{\mcG}'_i$ to be the family of birational counterparts of members of $\tilde{\mcG}_i$.
\end{Def}

We introduce further generality conditions for suitable families in the next subsection.

\subsection{Standard defining polynomials and additional conditions}

Let $X = X_{d_1,d_2} \subset \mbP (a_0,\dots,a_5)$ be a member of $\tilde{\mcG}_i$ with $i \in I^*_{c A/n} \cup I_{cD/3}$.
Let $x_0,\dots,x_5$ be the homogeneous coordinates of the ambient space and let $F_1$, $F_2$ be the defining polynomials of degree respectively $d_1$, $d_2$.
We assume $d_1 \le d_2$. 

\begin{Lem} \label{stdefeqXcA}
Suppose $i \in I^*_{cA/n}$.
After re-ordering and replacing coordinates, $X$ is defined in $\mbP (1,n,a_2,a_3,a_2+n,a_3+n)$ by the polynomials of the form
\[
\begin{split}
F_1 &= x_5 x_2 + x_4 x_3 + f (x_0,x_1,x_2,x_3), \\
F_2 &= x_5 x_4 - g (x_0,x_1,x_2,x_3).
\end{split}
\]
\end{Lem}

\begin{proof}
We assume $a_5 \ge a_j$ for every $j$.
If $X \in \tilde{\mcG_i}$ with $i \notin \{7,10\}$, then $X$ passes through $\msp_5$.
If $X \in \mcG_i$ with $i \in \{7,10\}$, we can assume that $X$ passes through $\msp_5$ after replacing coordinates.
We see that there are $0 \le j_1 < j_2 \le 3$ such that $a_5 + a_{j_1} = d_1$ and $a_5 + a_{j_2} = d_2$.
After re-ordering coordinates, we assume that $i_1 = 2$ and $i_2 = 4$.
Then, by quasismoothness of $X$ at $\msp_5$, we have $x_5 x_2 \in F_1$ and $x_5 x_4 \in F_2$.
After replacing $x_2$ and $x_4$, we can write $F_1 = x_5 x_2 + G_1$ and $F_2 = x_5 x_4 + G_2$ for some $G_1, G_2 \in \mbC [x_0,\dots,x_4]$.
Moreover, by filtering off terms divisible by $x_4$ in $F_2$ and then replacing $x_5$, we assume that $G_2$ does not involve $x_4$.

Suppose that $i \in \{7,10,16,18,21,26,36,38,44,52,63\}$.
In this case $d_1 \le 2 a_4$.
We claim that, after replacing coordinates other than $x_2,x_4,x_5$, we may assume $x_4 x_3 \in F_1$ and $x_4^2 \notin F_1$.
Suppose that $i \in \{21,36,38,52,63\}$. 
Then $d_1 < 2 a_4$ and hence there is no monomial divisible by $x_4^2$ in $F_1$.
It follows that $\msp_4 \in X$ and there is a unique $j$ such that $x_4 x_j \in F_1$ since $X$ is quasismooth.
By setting $j = 3$, we have $x_3 x_4 \in F_1$ and $x_4^2 \notin F_1$.
Suppose that $i \in \{7,10\}$.
In this case, we have $d_1 < 2 a_4$ and $a_4 = a_5$.
We can write $F_1 = x_5 x_2 + x_4 f_1 + f_2$ for some $f_1,f_2 \in \mbC [x_0,x_1,x_2,x_3]$.
Note that $\deg f_1 = a_2$ and there is at least one $j \ne 2$ such that $a_j = a_2$.
If $f_j$ does not involve coordinates other than $x_2$, then $\prt F_1/\prt x_5$ and $\prt F_1/\prt x_4$ are proportional.
This is impossible by $(\mathrm{C}_5)$.
It follows that there is $j \ne 2$ such that $x_j \in f_1$.
By setting $j = 3$, we have $x_4 x_3 \in F_1$ and $x_4^2 \notin F_1$. 
If $i \in \{16,18,26,44\}$, then $d_1 = 2 a_4$ and in this case there is $j \ne 4$ such that $a_j = a_4$.
We may assume that $j = 3$.
By quasismoothness of $X$, the polynomial $F_1 (0,0,0,x_3,x_4,0)$ cannot be a square.
In particular, at least one of $x_3 x_4$ and $x_3^2$ appear in $F_1$ with non-zero coefficient.
After replacing $x_3$, we may assume that $x_3 x_4 \in F_1$ and $x_4^2 \notin F_1$.
Therefore, in any case, we can write $F_1 = x_5 x_2 + x_4 x_3 + x_4 f_1 + f_2$ and $F_2 = x_5 x_4 + G_2$, where $f_1,f_2,G_2 \in \mbC [x_0,x_1,x_2,x_3]$.
After replacing $x_3 \mapsto x_3 - f_1$, we obtain the desired defining polynomials.

Suppose that $i \in \{9,22,28,33,48,57\}$.
In this case $2 a_4 < d_1 < 3 a_4$.
Then there is a unique $j \ne 2,4,5$ such that $2 a_j = d_2$.
We assume $j = 3$.
By Condition $(\mathrm{C}_1)$, we have $x_4 x_3 \in F_1$.
Then, since $d_1 < 3 a_4$ and $d_2 = 2 a_3$, we can write $F_1 = x_5 x_2 + x_4 x_3 + x_4^2 f_1 + x_4 f_2 + f_3$ and $F_2 = x_5 x_4 + x_3^3 + x_3 g_1 + g_2$, where $f_1,f_2,f_3 \in \mbC [x_0,x_1,x_2,x_3]$ and $g_1,g_2 \in \mbC [x_0,x_1]$.
Then by the replacement, 
\[
x_3 \mapsto x_3 - x_4 f_1 + x_2 f_1^2 - f_2,
x_5 \mapsto x_5 - x_4 f_1^2 + 2 x_3 f_1 + f_1 g_1,
\]
we can eliminate terms divisible by $x_4$ in $F_1$ and $F_2$ (other than $x_4 x_3$ in $F_1$).
Thus, $F_1$ and $F_2$ are in the desired forms.

Finally, we observe that $\{a_0,a_1\} = \{1,n\}$ and $a_5 - a_3 = a_4-a_2 = n$.
Thus by interchanging $x_0$ and $x_1$ if necessary, we may assume that $a_0 = 1$, $a_1 = n$, $a_4 = a_2 + n$ and $a_5 = a_3 + n$.
This completes the proof.
\end{proof}

\begin{Lem} \label{stdefeqX'cA}
Suppose that $X \in \tilde{\mcG}_i$, $i \in I^*_{cA/n}$, is defined by $F_1, F_2$ in \emph{Lemma \ref{stdefeqXcA}}.
Then the birational counterpart $X' \in \tilde{\mcG}'_i$ is the weighted hypersurface defined by the polynomial
\[
F' = w^2 x_2 x_3 + w f + g 
\]
in $\mbP (1,n,a_2,a_3,n)$, where $w$ is the homogeneous coordinates of degree $n$ other than $x_1$.
\end{Lem}

\begin{proof}
It is proved in \cite[Section 4.2]{Okada1} that if a member $X = X_{d_1,d_2} \subset \mbP (a_0,\dots,a_5)$ of $\mcG_i$, where $a_5 \ge a_i$ for $i = 0,1,2,3,4$ and $d_1 > d_2$, is defined by polynomials $F_1 = x_5 x_2 + G_1$ and $F_2 = x_5 x_4 + G_2$, where $G_1,G_2 \in \mbC [x_0,\dots,x_4]$, then the birational counterpart $X'$ is the weighted hypersurface in $\mbP (a_0,a_1,a_2,a_3,a_4-a_2)$ with homogeneous coordinates $x_0,\dots,x_3,w$ defined by 
\[
w G_1 (x_0,x_1,x_2,x_3,w x_2) - G_2 (x_0,x_1,x_2,x_3,w x_2) = 0.
\]
This proves the lemma.
\end{proof}

\begin{Lem} \label{stdefeqXcD}
\begin{enumerate}
\item After replacing homogeneous coordinates, defining polynomials of $X \in \mcG_{61}$ can be written as
\[
\begin{split}
F_1 &= u x_0 + s^3 + s^2 f_4 + s f_8 + f_{12}, \\
F_2 &= u s - g_{15},
\end{split}
\]
where $x_0,x_1,s,y,z,u$ are the homogeneous coordinates of $\mbP (1,1,4,5,6,11)$ and $f_j, g_{15} \in \mbC [x_0,x_1,y,z]$.
\item After replacing homogeneous coordinates, defining polynomials of $X \in \mcG_{62}$ can be written as
\[
\begin{split}
F_1 &= u y + s^2 + s f_6 + f_{12}, \\
F_2 &= u s - g_{15},
\end{split}
\]
where $x,y,z,t,s,u$ are the homogeneous coordinates of $\mbP (1,3,4,5,6,9)$ and $f_j, g_{15} \in \mbC [x,y,z,t]$.
\end{enumerate}
\end{Lem}

\begin{proof}
This is straightforward and we omit the proof.
\end{proof}

\begin{Lem} \label{stdefeqX'cD}
Suppose that $X \in \mcG_{61}$ $($resp. $\mcG_{62}$$)$ is defined by $F_1, F_2$ in \emph{Lemma \ref{stdefeqXcD}}.
Then the birational counterpart $X' \in \mcG'_{61}$ $($resp.\ $\mcG'_{62}$$)$ is the weighted hypersurface defined by the polynomial
\[
F' := w^4 x_0^3 + w^3 x_0^2 f_4 + w^2 x_0 f_8 + w f_{12} + g_{15},
\]
\[
(\text{resp. } F' := w^3 y^2 + w^2 y f_6 + w f_{12} + g_{15})
\]
in $\mbP (1_{x_0},1_{x_1},5_y,6_z,3_w)$ $($resp.\ $\mbP (1_x,3_y,4_z,5_t,3_w)$$)$.
\end{Lem}

\begin{proof}
This is proved by the same argument as in that of Lemma \ref{stdefeqX'cA}.
\end{proof}

\begin{Def}
Let $X' \in \tilde{\mcG}_i$ with $i \in I^*_{cA/n} \cup I_{cD/3}$.
A defining polynomial given in Lemmas \ref{stdefeqX'cA} and \ref{stdefeqX'cD} is called a {\it standard defining polynomial} of $X'$.
\end{Def}

\begin{Rem} \label{remcond}
In the big table, a standard defining polynomial of each family is given.
For some families, specific monomials are given right after the polynomial.
This is a condition imposed on members of $\tilde{\mcG}'_i$ which is a consequence of conditions $(\mathrm{C}_2)$ and $(\mathrm{C}_3)$ for families other than No.~$7$ (see Example \ref{excondX'}).
For family No.~$7$, the condition $y^2 \in f_4$ will be imposed in Condition \ref{addcond} below. 
\end{Rem}

\begin{Ex} \label{excondX'}
Let $X' = X'_5 \subset \mbP (1,1,1,2,1)$ be a member of $\tilde{\mcG}'_6$ with standard defining polynomial $F' = w^2 x_0 y + w f_4 + g_5$.
The birational counterpart $X \subset \mbP (1,1,1,2,2,3)$ is defined by $F_1 = z x_0 + y_0 y_1 + f_4 (x_0,x_1,x_2,y_1)$ and $F_2 = z y_0 - g_5 (x_0,x_1,x_2,y_1)$.
If $y^2 \notin f_4$, then $X$ contains the WCI curve $(x_0 = x_1 = x_2 = y_0 = 0)$ of type $(1,1,1,2)$.
This is impossible by $(\mathrm{C}_2)$.
Thus $y^2 \in f_4$.

We give an another example.
Let $X' = X'_9 \subset \mbP (1,1,2,3,3)$ be a member of $\tilde{\mcG}'_{21}$ with standard defining polynomial $F' = w^2 x_0 y + w f_6 + g_9$.
The birational counterpart $X \in \tilde{\mcG}_{21}$ is defined by $t x_0 + s y + f_6$ and $F_2 = t s - g_9$ in $\mbP (1_{x_0},1_{x_1},2_y,3_z,4_s,5_t)$.
We can write $F_1 (0,0,y,z,s,t) = s y + \alpha z^2 + \beta y^3$ and $F_2 (0,0,y,z,s,t) = t s - (\gamma z^3 + \delta z y^3)$. 
If $\alpha = 0$ or $\beta = 0$, then $(x_0 = x_1 = 0)_X$ is reducible, which is impossible by $(\mathrm{C}_3)$.
Hence $z^2 \in f_6$ and $y^3 \in f_6$.
By the same reason, the case $\gamma = \delta = 0$ cannot happen.
This implies that if $z^3 \notin g_9$, then $z y^3 \in g_9$.
\end{Ex}

We introduce additional conditions on members of $\tilde{\mcG}'_i$ for $i \in \{ 6,7,9,10,16,18,21\}$.
We explain that a member $X'$ of family $\tilde{\mcG}'_i$ listed in Table \ref{table:eq} is defined by the polynomial $F'$ given in the second column of Table \ref{table:eq}.
For families No.~$10$, $16$ and $21$, $F'$ is a standard defining polynomial.
Let $X' \in \tilde{\mcG}'_6$.
Then a standard defining polynomial of $X'$ is of the form $F' = w^2 x_0 y + w f_4 + g_5$, where $f_4,g_5 \in \mbC [x_0,x_1,x_2,y]$.
Since $y^2 \in f_4$ by generality condition, we can write $F' = w^2 x_0 y + w (y^2 + y a_2 + a_4) + y^2 b_1 + y b_3 + b_5$, where $a_j,b_j \in \mbC [x_0,x_1,x_2]$, after re-scaling $y$.
Then, after replacing $w \mapsto w-b_1$, we may assume $b_1 = 0$.
This is the one given in the second column.
Similarly, defining polynomials of members of families No.~$9$ and $18$ are given as in the second column.
Note that, for family No.~$7$, the assertion $y^2 \in f_4$ does not follow from Condition \ref{cond1} and it is imposed in the following condition.

\begin{Cond} \label{addcond}
For a member $X$ of family $\tilde{\mcG}_i$ listed in Table \ref{table:eq}, the defining polynomial of the birational counterpart $X'$ of $X$ is of the form in the second column and each system of equations in the third column do not have a non-trivial solution.
\end{Cond}

\begin{table}[h]
\begin{center}
\caption{System of equations}
\label{table:eq}
\begin{tabular}{ccc}
\hline
No. & $F'$ & Equations \\
\hline
6 & $w^2 x_0 y + w (y^2 + y a_2 + a_4) + y b_3 + b_5$ & $x_0 = a_2 = a_4 = 0$ \\
& & $x_0 = b_3 = b_5 = 0$ \\
& & $x_0 = a_4 = b_5 - a_2 \frac{\prt a_4}{\prt x_0} = 0$ \\
\hline
7 & $w^2 x_0 x_1 + w (y^2 + y a_2 + a_4) + y b_4 + b_6$ & $x_0 = b_4 = b_6 = 0$ \\
& & $x_1 = b_4 = b_6 = 0$ \\
& & $x_0 = x_1 = 4 a_4 - a_2^2 = 0$ \\
\hline
9 & $w^2 x_0 y + w (y a_2 + a_6) + y^2 + y b_3 + b_6$ & $x_0 = a_2 = a_5 = 0$ \\
\hline
10 & $w^2 y_0 y_1 + w f_5 + g_6$ & $y_0 = y_1 = f_5 = g_6 = 0$ \\
\hline
16 & $w^2 y z + w f_6 + g_7$ & $y = f_6 = g_7 = f_6 (x_0,x_1,0,0) = 0$ \\
\hline
18 & $w^2 x_0 z + w f_6 + z b_5 + b_8$ & $x_0 = f_6 = z b_5 + b_8 = \frac{\prt f_6}{\prt z} = 0$ \\
& & $x_0 = b_5 = b_8 = 0$ \\
\hline
21 & $w^2 x_0 y + w f_6 + g_9 = 0$ & $x_0 = f_6 = g_9 = \frac{\prt f_6}{\prt f_6} = 0$. \\
& & $y = f_6 = g_9 = \frac{\prt f_6}{\prt z} = 0$
\end{tabular}
\end{center}
\end{table}

\begin{Def}
For $i \in I^*_{cA/n} \cup I_{cD/3}$, we define $\mcG_i$ to be the subfamily of $\tilde{\mcG}$ consisting of members satisfying Condition  \ref{addcond}.
We define $\mcG'_i$ to be the family of birational counterparts of members of $\mcG_i$.
\end{Def}

Note that $\mcG_i$ and $\mcG'_i$ are non-empty open subset of $\tilde{\mcG}_i$ and $\tilde{\mcG}'_i$, respectively, and $\mcG_i = \tilde{\mcG}_i$, $\mcG'_i = \tilde{\mcG}_i$ for $i \notin \{6,7,9,10,16,18,21\}$.
Note also that Condition \ref{addcond} is only used to exclude curves of low degree in Section \ref{sec:exclcurve}.

\subsection{Structure of the proof}

We recall definitions of maximal extraction and maximal centers.

\begin{Def}
Let $X$ be a $\mbQ$-Fano variety with Picard number $1$ and let $\varphi \colon Y \to X$ be a divisorial contraction with exceptional divisor $E$.
We say that $\varphi$ is a {\it maximal extraction} if there is a movable linear system $\mcH \sim_{\mbQ} -n K_X$ such that 
\[
\frac{1}{n} > c (X,\mcH) = \frac{a_E (K_X)}{m_E (\mcH)},
\]
where $m_E (\mcH)$ is the multiplicity of $\mcH$ along $E$, $a_E (K_X)$ is the discrepancy of $K_X$ along $E$, and $c (X,\mcH)$ is the canonical threshold of the pair $(X,\mcH)$.
A closed subvariety $\Gamma \subset X$ is called a {\it maximal center} if there is a maximal extraction centered along $\Gamma$.
\end{Def}

In the rest of this paper, we prove the following.

\begin{Thm} \label{mainthm2}
Let $X'$ be a member of $\mcG'_i$ with $i \in I_{cA/n}^* \cup I_{cD/3}$.
Then, no nonsingular point and no curve on $X'$ is a maximal center.
Moreover, for each divisorial extraction $\varphi \colon Y' \to X'$ centered at a singular point, one of the following holds.
\begin{enumerate}
\item $\varphi$ is not a maximal extraction.
\item There is a birational involution $\iota \colon X' \ratmap X'$ that is a Sarkisov link starting with $\varphi$.
\item There is a Sarkisov link $\sigma \colon X' \ratmap X$ starting with $\varphi$.
\end{enumerate}
\end{Thm}

We note that Theorem \ref{mainthm2} follows from Theprems \ref{thm:SlinkcD}, \ref{thm:SlinkcA}, \ref{thm:birivQI}, \ref{thm:birinvG18-2}, \ref{thm:birinvcA1}, \ref{thm:birinvcA2}, \ref{thm:nspt}, \ref{thm:excltqs}, \ref{thm:exclcAneg} and \ref{thm:curve}.
By \cite[Lemma 2.32]{Okada2}, Theorem \ref{mainthm} follows from Theorem \ref{mainthm2} and \cite[Theorem 1.3]{Okada1}.
We explain the outline of this paper.
In Section \ref{sec:divextr}, we explain divisorial extractions centered at $cA/n$ and $cD/3$ points.
In Sections \ref{sec:Slink}, \ref{sec:birinvtqpt} and \ref{sec:birinvcA}, we construct various Sarkisov links that are links from $X'$ to $X$ and birational involutions of $X'$ centered at singular points.
The construction of birational involutions centered at terminal quotient singular points is the same as that of \cite{CPR} with a single exception.
We need a hard construction due to \cite{CP} for the exceptional case.
The construction of birational involutions centered at $cA/n$ points is based on explicit global descriptions of divisorial extractions centered at $cA/n$ points.
This is quite similar to \cite{AK}.  
The rest of the paper is devoted to exclusion.
Nonsingular points, some terminal quotient singular points, some $cA/n$ and $cD/3$ points, and curves are excluded in Sections \ref{sec:exclnspt}, \ref{sec:excltqpt}, \ref{sec:exclcA} and \ref{sec:exclcurve}, respectively.
We exclude points and most of the curves by the combination of methods in \cite{CPR,CP}.
We refer the readers to \cite[Section 2.4]{Okada2} for various excluding methods.
A large volume of the paper is devoted to excluding curves of low degree passing through the $cA/n$ point.
The method is simple but we need careful and quite complicated computations. 
Finally, the table in Section \ref{sec:bigtable}, {\it the big table}, summarizes the results from which one can see what happens at each singular point.

\section{Divisorial extractions centered at $cA/n$ and $cD/3$ points} \label{sec:divextr}

In this section, we describe divisorial extractions centered at the point $\msp := \msp_4$ of a member $X'$ of the family $\mcG'_i$ with $i \in I^*_{cA/n} \cup I_{cD/3}$.
This is based on the classification results due to Hayakawa \cite{Hayakawa} and Kawakita \cite{Kawakita1, Kawakita2}.

\subsection{$cA/n$ point}

Let $X'$ be a member of $\mcG'_i$ with $i \in I^*_{cA/n}$ and $\msp = \msp_4$.
Throughout this section, we assume that $X'$ is defined by a standard polynomial $F' = w^2 x_2 x_3 + w f + g$ in $\mbP (1,n,a_2,a_3,n)$.
Set $d := \deg F'$.
Note that $d \ge 5$.
Moreover $a_2 + a_3 \equiv 0 \pmod{n}$ if $n > 1$.

We see that a general hyperplane section of the index $1$ cover of the singularity $(X', \msp)$ is a du Val singularity of type $A_{d-n-1}$.
If $n \ne 2,4$, then the classification of $3$-dimensional terminal singularities immediately implies that $(X',\msp)$ is of type $cA/n$.
If $n = 2$ or $4$, then $(X',\msp)$ is of type $cA/n$, $cAx/2$ or $cAx/4$.
It is straightforward to see that $(X',\msp)$ is not of type $cAx/2$ and $cAx/4$, and hence $(X',\msp)$ is of type $cA/n$.
Since $(X',\msp_4)$ is of type $cA/n$, we have an identification,
\[
(\msp \in X') \cong (o \in (s_1 s_2 + h (s_3,s_4) = 0)/\mbZ_n (a_2,-a_2,1,0)),
\]
for some $h (s_3,s_4) \in \mbC \{ s_3,s_4\}$.

\begin{Lem} \label{lem:isomwps}
Let $\mbP := \mbP (c_1,c_2,c_3,c_4)$ be a weighted projective space with homogeneous coordinates $s_1, s_2, s_3$ and $s_4$.
Suppose that $c_1 \ge c_2$, $c_1 > c_i$ for $i = 3,4$, and there is an automorphism $\sigma$ of $\mbP$ that induces an isomorphism 
\[
\sigma |_{H_1} \colon H_1 := (s_0 s_1 + g_1 (s_2,s_3) = 0) \xrightarrow{\cong} H_2:= (s_0 s_1 + g_2 (s_2,s_3) = 0)
\]
between weighted hypersurfaces in $\mbP$, where $g_i$ is homogeneous of degree $c_0 + c_1$.
Then there is an automorphism $\tau$ of $\mbP (c_3,c_4)$ such that $g_1 = \tau^* g_2$.
\end{Lem}

\begin{proof}
Let $\sigma^*$ be the automorphism of $\mbC [s_1,s_2,s_3,s_4]$ induced by $\sigma$.
We have 
\begin{equation} \label{eq:uniqexc}
\sigma^* (s_1 s_2 + g_2) = \alpha (s_1 s_2 + g_1)
\end{equation}
for some $\alpha \ne 0$.
After re-scaling coordinates, we may assume $\alpha = 1$.

We first treat the case where $c_0 > c_1$.
After re-scaling $s_0$, we have $\sigma^*s_1 = s_1 + a$ for some $a \in \mbC [s_2,s_3,s_4]$, and $\sigma^* s_i$ does not involve $s_1$ for $i = 2,3,4$ since $c_1 > c_i$.
By comparing the terms involving $s_1$ in \eqref{eq:uniqexc}, we have $\sigma^* s_2 = s_2$ and thus $g_1 = a s_2 + \sigma^*g_2$.
It follows that $\sigma$ restricts to an automorphism $\bar{\sigma}$ of $(s_2 = 0) \cong \mbP (c_1,c_3,c_4)$.
Moreover, since $\sigma^* s_i$ does not involve $s_1$ for $i = 3,4$, the correspondence $s_i \mapsto \bar{\sigma}^* s_i$ for $i = 3,4$ defines an automorphism of $\mbP (c_3,c_4)$, which we denote by $\tau$.
By the construction, we have $g_1 = \tau^* g_2$.

We treat the case where $c_1 = c_2$.
We have $\sigma^* s_1 = \alpha_1 s_1 + \alpha_2 s_2 + a$ and $\sigma^* s_2 = \beta_1 s_1 + \beta_2 s_2 + b$ for some $\alpha_i, \beta_i \in \mbC$ and $a, b \in \mbC [s_3,s_4]$.
Note that $\sigma^*s_i \in \mbC [s_3,s_4]$ for $i = 3,4$ since $c_3, c_4 < c_1 = c_2$.
We have $(\alpha_1, \beta_1) \ne (0,0)$ since $\sigma$ is an automorphism.
Possibly interchanging $s_1$ and $s_2$, we may assume $\alpha_1 \ne 0$.
Then, by comparing terms involving $s_1$ and $s_2$ in \eqref{eq:uniqexc}, we have $\alpha_2 = \beta_1 = 0$, $\alpha_1 \beta_2 = 1$, $a = b = 0$ and $g_1 = \sigma^* g_2$.
Thus $\sigma$ restricts to an automorphism $\tau$ of $(s_1 = s_2 = 0) \cong \mbP (c_3,c_4)$ and we have $g_1 = \tau^* g_2$.
This completes the proof.
\end{proof}

\begin{Lem}
We have an equivalence
\[
(\msp \in X') \cong (o \in (s_1 s_2 + h (s_3,s_4) = 0) / \mbZ_n (a_2, - a_2,1,0))
\]
of singularities, where the lowest weight part of $h$ with respect to the weight $\wt (s_3,s_4) = (1,n)$ is $h_{d-n} = f (s_3,s_4,0,0)$.
\end{Lem}

\begin{proof}
Let $\varphi \colon Y' \to X'$ be the weighted blowup with $\wt (x_0,x_1,x_2,x_3) = \frac{1}{n} (1,n,a_2 + n,a_3)$ at $\msp$ with exceptional divisor $E$.
It is proved in \cite[Section 4.2]{Okada1} that $\varphi$ is a divisorial contraction (see also Section \ref{sec:Slink}).
We have an isomorphism
\[
E \cong (x_2 x_3 + f (x_0,x_1,0,x_3) = 0) \subset \mbP (1,n,a_2+n,a_3).
\]
By filtering off terms divisible by $x_3$ and then replacing $x_2$, we see that $E \cong (x_2 x_3 + f (x_0,x_1,0,0) = 0)$.
Let 
\[
(\msp \in X') \cong (o \in (s_1 s_2 + h (s_3,s_4) = 0) / \mbZ_n (a_2, - a_2,1,0)).
\]
be any identification of the $cA/n$ point $\msp$, where $h (s_3,s_4) \in \mbC \{s_3,s_4\}$.
Let $\psi$ be the divisorial extraction of $(s_1 s_2 + h = 0)/\mbZ_n$ which corresponds to $\varphi$, and let $F$ be the $\psi$-exceptional divisor.
$\psi$ is a weighted blowup with $\wt (s_1,s_2,s_3,s_4) = \frac{1}{n} (c_1,c_2,c_3,c_4)$ for some $c_1,\dots,c_4$.
The identification of singularities induces an isomorphism
\[
\sigma \colon \mbP (1,n,a_2+n,a_3) \to \mbP (c_1,c_2,c_3,c_4)
\]
which restricts to an isomorphism $\sigma|_E \colon E \to F$ between exceptional divisors.
In particular, we have $\{c_1,c_2,c_3,c_4\} = \{1,n,a_2+n,a_3\}$ and we may assume that $c_1 = a_2 + n$, $c_2 = a_3$, $c_3 = 1$ and $c_4 = n$ after interchanging $s_1$ and $s_2$, and $s_3$ and $s_4$. 
Here $F = (s_1 s_2 + h_m (s_3,s_4) = 0)$, where $h_m$ is the lowest weight part of $h$.
Since $a_2 + n > 1,n$, we can apply Lemma \ref{lem:isomwps} for the isomorphism $\sigma$ and there is an isomorphism $\tau \colon \mbP (1,n) \to \mbP (c_3,c_4)$ such that $\tau^* h_m = f (x_0,x_1,0,0)$.
We see that $\tau$ extends to an $\mbZ_n$-equivariant automorphism of  $\mbA^4$ with coordinates $s_1,s_2,s_3,s_4$ by setting $\tau^*s_i = s_i$ for $i = 1,2$.
Thus, by replacing the germ $(s_1 s_2 + h = 0)/\mbZ_n$ with the automorphic image $(s_1 s_2 + \tau^*h = 0)/\mbZ_n$, we see that $m = \deg f = d-n$ and $h_{d-n} = f (s_3,s_4,0,0)$.
\end{proof}

\begin{Def}
Under the above identification, for positive integers $r_1$ and $r_2$, let 
\[
\varphi_{(r_1,r_2)} \colon Y'_{(r_1,r_2)} \to X'
\] 
be the birational morphism that is the weighted blowup of $X'$ at $\msp$ with weight $\wt (s_1,s_2,s_3,s_4) = \frac{1}{n} (r_1,r_2,1,n)$.
We call $\varphi_{(r_1,r_2)}$ the $\frac{1}{n} (r_1,r_2)$-{\it blowup}.
\end{Def}

\begin{Lem}
\begin{enumerate}
\item
Let $X'$ be a member of $\mcG'_i$ with $i \in I^*_{cA/1}$.
Then, 
\[
\{ \varphi_{(r_1,r_2)} \mid r_1, r_2 > 0, r_1 + r_2 = d-1 \}
\]
are the divisorial extractions centered at $\msp$.
\item
Let $X' \in \mcG'_i$ be a member of $i \in I_{cA/n}$ with $n \ge 2$.
Then,
\[
\{ \varphi_{(r_1,r_2)} \mid r_1, r_2 > 0, r_1 + r_2 = d-n, r_1 \equiv a_2 \imod{n} \},
\]
are the divisorial extractions centered at $\msp$.
\end{enumerate}
\end{Lem}

\begin{proof}
(2) follows immediately from \cite[\S 6]{Hayakawa}.
We will prove (1).
According to the classification \cite[Theorem 1.13]{Kawakita1}, we need to show that a weighted blowup ``of type $(r_1,r_2,k,1)$" cannot be a divisorial extraction of $(\msp \in X')$ for $k \ge 2$.
Suppose that $(\msp \in X')$ admits such an extraction.
This means that there are $k \ge 2$ and an identification
\[
(\msp \in X') \cong (o \in (s_1 s_2 + h' (s_3,s_4) = 0))
\]
for some $h' \in \mbC \{s_3,s_4\}$ such that $s_3^{d-n} \in h'$ and $s_3^i s_4^j \notin h'$ for $2 k i + j < k (d-n)$.
In this case, the weighted blowup with $\wt (s_1,s_2,s_3,s_4) = (r_1,r_2,k,1)$ is a divisorial extraction of $(\msp \in X')$ for any $r_1, r_2$ such that $r_1+r_2$ is divisible by $k$ and $k$ is co-prime to $r_1$ and $r_2$.
But, in this case, $Y'_{(r_1,r_2)}$ has a (unique) $cA$ point along the exceptional divisor for any $r_1,r_2$ such that $r_1 + r_2 = d-n$.
On the other hand, by \cite[Section 4,2]{Okada1}, there is a pair $(r_1,r_2)$ with $r_1+r_2 = d-n$ such that $Y'_{(r_1,r_2)}$ has only terminal quotient singularities.
This is a contradiction and the proof is completed. 
\end{proof}

\subsection{$cD/3$ point}

Let $X'$ be a member of $\mcG'_i$ with $i \in I_{cD/3} = \{61,62\}$ and $\msp = \msp_4 \in X'$ the $cD/3$ point.
The singularity $(X',\msp)$ is indeed of type $cD/3$ since a general hyperplane section of index $1$ cover of $(X',\msp)$ is of type $D$.

\begin{Lem}
Suppose that $i = 61$ $($resp.\ $62$$)$ and $X'$ is defined by a standard polynomial.
Then the weighted blowup $\varphi' \colon Y' \to X'$ of $X'$ at $\msp$, which is defined as the weighted blowup of $X'$ at $\msp$ with $\wt (x_0,x_1,y,z) = \frac{1}{3} (4,1,5,6)$ $($resp.\ $\wt (x,y,z,t) = \frac{1}{3} (1,6,4,5)$$)$, is the unique divisorial extraction centered at $\msp$.
\end{Lem} 

\begin{proof}
It is proved in \cite{Okada1} that $\varphi \colon Y' \to X'$ is an extremal extraction.
According to the classification of extremal divisorial extraction of $cD/3$ point in \cite[\S 9]{Hayakawa}, every extremal divisorial extraction centered at $\msp$ is a weighted blowup and if it admits a weighted blowup with weight $(1,4,5,6)$ as an extremal divisorial extraction, then it is the unique extremal divisorial extraction of $\msp$.
\end{proof}

\section{Sarkisov links between $X$ and $X'$} \label{sec:Slink}

Let $X'$ be a member of $\mcG'_i$ with $i \in I^*_{cA/n} \cup I_{cD/3}$ and $\msp = \msp_4$.
We will show that there is a Sarkisov link $X' \ratmap X$ to the birational counterpart $X \in \mcG_i$ starting with each $\varphi_{(k,l)}$-blowup marked $X' \ratmap X \ni \frac{1}{r} (\alpha,\beta,\gamma)$.
The mark $\frac{1}{r} (\alpha,\beta,\gamma)$ indicates that the link $X' \ratmap X$ ends with the Kawamata blowup of $X$ at a singular point of type $\frac{1}{r} (\alpha,\beta,\gamma)$.

\begin{Thm} \label{thm:SlinkcD}
Suppose $i \in I_{cD/3} = \{61,62\}$.
Then, there is a Sarkisov link $X' \ratmap X$ starting with the unique divisorial extraction centered at $\msp$.
\end{Thm}

\begin{proof}
In \cite[Theorem 4.10]{Okada1}, we constructed a Sarkisov link $X \ratmap X'$ which ends with the divisorial contraction centered at $\msp$.
Thus its inverse $X' \ratmap X$ is the desired one.
\end{proof}

In the following, we assume that $i \in I^*_{cA/n}$ and $X' \subset \mbP (1,n,a_2,a_3,n)$ is defined by a standard polynomial $F' = w^2 x_2 x_3 + w f + g$.
Let $X \in \mcG_i$ be the birational counterpart.
In \cite[Section 4.2]{Okada1}, Sarkisov links $X \ratmap X'$ are constructed.
Each link $X \ratmap X'$ ends with a divisorial extraction centered at $\msp$.
Such an extraction is a weighted blowup and the complete description is given.
We recall the construction.

Let $\varphi' \colon Y' \to X'$ be the weighted blowup with $\wt (x_0,x_1,x_2,x_3) = \frac{1}{n} (1,n,a_2 + n,a_3)$ (resp.\ $\frac{1}{n} (1,n,a_2,a_3+n)$).
It is proved in \cite[Section 4.2]{Okada1} that $\varphi$ is a divisorial extraction and hence we have $\varphi' = \varphi_{(a_2 + n, a_3)}$ (resp.\ $\varphi_{(a_2,a_3+n)}$).
We define
\[
Z := (x_4 (x_4 x_3 + f) + x_2 g = 0) \subset \mbP (1,n,a_2,a_3,a_2 +n),
\]
\[
(\text{resp}.\ Z := (x_5 (x_5 x_2 + f) + x_3 g = 0) \subset \mbP (1,n,a_2,a_3,a_3+n)),
\]
where $\deg x_4 = a_2 + n$ and $\deg x_5 = a_3+n$.
By taking the ratio in two ways
\[
x_5 := - \frac{x_4 x_3 + f}{x_2} = \frac{g}{x_4} \text{ and } w := \frac{x_4}{x_2} = - \frac{g}{x_3+f},
\]
\[
\left( \text{resp}.\ x_4 := - \frac{x_5 x_2 + f}{x_3} = \frac{g}{x_5} \text{ and } w := \frac{x_5}{x_3} = - \frac{g}{x_5 x_2 + f} \right),
\]
we obtain birational maps $X \ratmap Z$ and $X' \ratmap Z$.
The Kawamata blowup $\varphi \colon Y \to X$ at $\msp_5$ (resp.\ $\msp_4$) and $\varphi$ resolves the indeterminacies of $X \ratmap Z$ and $X' \ratmap Z$, respectively, and we have the following diagram
\[
\xymatrix{
Y' \ar[d]_{\varphi'} \ar@{-->}[rr]^{\tau} \ar[rd] & & Y \ar[d]^{\varphi} \ar[ld] \\
X' & Z & X}
\]
where $\tau$ is a flop.
This is the Sarkisov link, denoted by $\sigma_{(a_2+n,a_3)}$ (resp.\ $\sigma_{(a_2,a_3+n)}$), starting with $\varphi_{(a_2 + n,a_3)}$ (resp.\ $\varphi_{(a_2,a_3+n)}$).

If $i = \{7,10\}$, then the above construction gives all the Sarkisov links $X' \ratmap X$ centered at $\msp$.

Suppose $i \ne \{7,10\}$, that is, 
\[
i \in \{6,9,16,18,22,26,28,33,44,48,57\}.
\]
In the following, let $(k,l)$ be either $(a_2+n,a_3)$ or $(a_2,a_3+n)$. 
We have constructructed a Sarkisov link $X' \ratmap X$ starting with $\frac{1}{n} (k,l)$-blowup.
We see that $X'$ admits $\frac{1}{n} (l,k)$-blowup.
We will explain a relation between $\varphi_{(k,l)}$ and $\varphi_{(l,k)}$ and show that there is a Sarkisov link $X' \ratmap X$ starting with $\varphi_{(l,k)}$.

Suppose $i \in \{9,22,28,33,48,57\}$.
Then $\deg F' = 2 a_3$ and we can write
\[
F' = w^2 x_2 x_3 + w (x_3 a + b) + x_3^2 + x_3 c + d
\]
for some $a,b,c,d \in \mbC [x_0,x_1,x_2]$.

\begin{Def} \label{def:mu}
For $i \in \{9,22,28,33,48,57\}$, we define $\mu$ to be the biregular involution of $X'$ defined by the replacement $x_3 \mapsto - x_3 - w^2 x_2 - w a - c$.
\end{Def}

We see that $\mu (\msp) = \msp$ and that the composite $\mu \circ \varphi_{(k,l)}$ defines the $\varphi_{(l,k)}$-blowup. 
The diagram 
\[
\xymatrix{
 & Y' \ar[ld]_{\varphi_{(l,k)}} \ar[d]^{\varphi_{(k,l)}} \ar@{-->}[rr]^{\text{flop}} & & Y \ar[d]_{\varphi} \\
X' \ar[r]^{\mu} \ar@{-->}@/_1.1pc/[rrr]_{\sigma_{(l,k)}} & X' \ar@{-->}[rr]^{\sigma_{(k,l)}} & & X}
\]
gives the Sarkisov link $\sigma_{(l,k)} \colon X' \ratmap X$ starting with $\varphi_{(l,k)}$.

Suppose $i \in \{6,16,18,26,44\}$.
Then, $\deg F' = 2 a_3 + n$ and we can write
\[
F' = w^2 x_2 x_3 + w (x_3^2 + x_3 a + b) + x_3 c + d
\]
for some $a,b,c,d \in \mbC [x_0,x_1,x_2]$.
Here, a priori, there is a term $x_3^2 e$ in $F'$ for some $e \in \mbC [x_0,x_1,x_2]$ but we can eliminate the term by replacing $w$ with $w-e$.
The birational counterpart $X \in \mcG_i$ is defined by
\[
\begin{split}
F_1 &= x_5 x_2 + x_4 x_3 + (x_3^2 + x_3 a + b) \\
F_2 &= x_5 x_4 - (x_3 c + d)
\end{split}
\]
in $\mbP (1,n,a_2,a_3,a_2+n,a_3+n)$.

\begin{Def} \label{def:nu}
Suppose $i \in \{6,16,18,26,44\}$.
We define
\[
\begin{split}
\tilde{F}' &:= F' (x_0,x_1,x_2,- x_3 - w x_2 - a,w) \\
&= w^2 x_2 x_3 + w (x_3^2 + x_3 a + b - x_2 c) - x_3 c + d - ac
\end{split}
\]
and then define $\tilde{X}'$ to be the weighted hypersurface defined by $\tilde{F}'$.
We denote by $\nu' \colon X' \to \tilde{X}'$ the isomorphism defined by the replacement $x_3 \mapsto - x_3 - w x_2 - a$.

We define
\[
\begin{split}
\tilde{F}_1 &:= F_1 (x_0,x_1,x_2,-x_3-x_4-a,x_4,x_5 +c) \\
&= x_5 x_2 + x_4 x_3 + (x_3^2 + x_3 a + b - x_2 c) \\
\tilde{F}_2 &:= F_2 (x_0,x_1,x_2,-x_3-x_4-a,x_4,x_5 +c) \\
&= x_5 x_4 - (-x_3 c + d - ac)
\end{split}
\]
and then define $\tilde{X}$ to be the WCI defined by $\tilde{F}_1 = \tilde{F}_2 = 0$.
We denote by $\mu \colon X \to \tilde{X}$ the isomorphism defined by the replacements $x_3 \mapsto - x_3 - x_4 - a$ and $x_5 \mapsto x_5 + c$.
\end{Def}

Since $\tilde{X}' \in \mcG'_i$ and $\tilde{X}$ is the birational counterpart of $\tilde{X}'$, there is a Sarkisov link $\tilde{\sigma}_{(k,l)} \colon \tilde{X}' \ratmap \tilde{X}$ starting with $(k,l)$-blowup $\tilde{\varphi}_{(k,l)}$ of $\tilde{X'}$ and ending with Kawamata blowup $\tilde{\varphi} \colon \tilde{Y} \to \tilde{X}$.
We see that $\varphi_{(l,k)} = {\nu'}^{-1} \circ \tilde{\varphi}_{(k,l)}$ and that the composite $\varphi := \nu^{-1} \circ \tilde{\varphi}$ is the Kawamata blowup of $X$.
Therefore the diagram
\[
\xymatrix{
 & \tilde{Y}' \ar[d]^{\tilde{\varphi}_{(k,l)}} \ar[ld]_{\varphi_{(l,k)}} \ar@{-->}[rr]^{\text{flop}} & & \tilde{Y} \ar[d]_{\tilde{\varphi}} \ar[rd]^{\varphi_{\msp}} & \\
 X' \ar[r]^{\nu'} \ar@{-->}@/_1.2pc/[rrrr]_{\sigma_{(l,k)}} & \tilde{X}' \ar@{-->}[rr]^{\tilde{\sigma}_{(k,l)}} & & \tilde{X} & \ar[l]_{\nu} X}
\] 
gives the Sarkisov link $\sigma_{(l,k)} \colon X' \ratmap X$ starting with $\varphi_{(l,k)}$ and ending with $\varphi$.
As a conclusion, we have the following.

\begin{Thm} \label{thm:SlinkcA}
Let $X'$ be a member of $\mcG'_i$ with $i \in I^*_{cA/n}$ and let $\varphi_{(k,l)}$ be a divisorial extraction centered at $\msp = \msp_4$ marked $X' \ratmap X$ in the big table.
Then there is a Sarkisov link $\sigma_{(k,l)} \colon X' \ratmap X$ starting with $\varphi_{(k,l)}$.
\end{Thm}

\section{Birational involutions centered at quotient singular points} \label{sec:birinvtqpt}

In this section, we construct a birational involution of $X' \in \mcG'_i$ that is a Sarkisov link centered at a suitable terminal quotient singular point.
Throughout this section, let $\msp \in X'$ be a terminal quotient singular point with non-empty third column in the big table and let $\varphi' \colon Y' \to X'$ be the Kawamata blowup of $X'$ at $\msp$ with exceptional divisor $E$.
For a divisor or a curve $\Delta$ on $X'$, we denote by $\tilde{\Delta}$ the proper transform of $\Delta$ on $Y'$.

\subsection{Quadratic involutions}

\begin{Thm} \label{thm:birivQI}
Let $X'$ be a member of $\mcG'_i$ with $i \in I^*_{cA/n} \cup I_{cD/3}$ and $\msp \in X'$ a terminal quotient singular point marked Q.I. in the third column.
Then there exists a birational involution of $X'$ that is a Sarkisov link centered at $\msp$.
\end{Thm}

\begin{proof}
Let $\mbP (a_0,\dots,a_4)$ be the ambient space of $X'$ with coordinates $x_0,\dots,x_3,w$.
After replacing coordinates, we may assume $\msp = \msp_j$ for some $j \ne 4$.
Then, after replacing coordinates, the defining polynomial can be written as $F' = x_j^2 x_k + x_j h_1 + h_2$, where $k \ne j$ and $h_1,h_2$ are polynomials in $x_0,\dots,x_3,w$ that do not involve $x_j$.
It then follows from \cite[Theorem 4.9]{CPR} that there is a birational involution of $X'$ that is a Sarkisov link centered at $\msp$.
\end{proof}

\subsection{Family $\mcG'_{18}$ and the point of type $\frac{1}{2} (1,1,1)$} \label{sec:birinv1}

Let $X' = X'_8 \subset \mbP (1,1,2,3,2)$ be a member of $\mcG'_{18}$ and $\msp \in X'$ the singular point of type $\frac{1}{2} (1,1,1)$.

The defining polynomial of $X'$ is of the form $F' = w^2 x_0 z + w f_6 + g_8$, where $f_6,g_8 \in \mbC [x_0,x_1,y,z]$.
If $y^3 \notin f_6$, then there is no $\frac{1}{2} (1,1,1)$ point on $X'$.
Hence we may assume that $y^3 \in f_6$.
After replacing $w$ suitably, we assume that $\msp = \msp_2$ and there is no monomial in $g_8$ that is divisible by $y^3$.

We first treat the case where $z^2 y \in g_8$.
By replacing coordinates suitably, we have
\[
F' = y z^2 + a_5 z - w y^3 - b_4 y^2 - c_6 y + d_8,
\]
where $a_5,b_4,c_6,d_8 \in \mbC [x_0,x_1,w]$.
By \cite[Theorem 4.13]{CPR}, the sections
\[
u := z^2 - w y^2 - b_4 y - c_6 \text{ and } v := u z + a_5 w y + a_5 b_4
\]
lift to plurianticanonical sections of $Y'$.
Moreover, the anticanonical model $Z'$ of $Y'$ is the weighted hypersurface defined by the equation
\[
- v^2 + a_6 b_4 v + u^3 + u^2 c_6 - (b_4 d_8 + a_5^2 w)v + (-a_5^2 c_6 + d_8^2) w = 0
\]
in $\mbP (1_{x_0},1_{x_1},2_w,6_u,9_v)$ and the corresponding map $\psi' \colon Y' \ratmap Z'$ is a morphism.
By \cite[Lemma 3.2]{Okada2}, either $\msp$ is not a maximal center or there is a birational involution of $X'$ that is a Sarkisov link centered at $\msp$.

In the following, we treat the case where $z^2 y \notin g_8$.
We will construct a birational involution of a suitable model of $Y'$ and observe that the induced birational involution of $X'$ gives a Sarkisov link starting with $\varphi$.
This kind of involution is called an {\it invisible involution} in \cite{CP} and its construction is first introduced there.
One can find its construction in a relatively general setting in \cite[Section 7.1]{Okada2} and the rest of this section is to verify \cite[Condition 7.1]{Okada2}.
This involves complicated computations and a reader not interested in special members may skip this part since this does not happen for a general $X' \in \mcG'_{33}$.

After re-scaling $y,z,w$, we assume that the coefficients of $y^3$ and $z^2$ in $f_6$ are both $1$.
In this case, either $y^2 z x_0 \in g_8$ or $y^2 z x_1 \in g_8$ because otherwise $X'$ is not quasismooth at $(0 \!:\! 0 \!:\! -1 \!:\! 1 \!:\! 0)$.
We see that $x_0$, $x_1$, $z$ vanish along $E$ to order $1/2$ and $w$ vanishes along $E$ to order $2/2$.
For $\lambda,\mu \in \mbC$, we define $S_{\lambda} := (x_1 - \lambda x_0 = 0)_{X'}$ and $T_{\mu} := (w - \mu x_0^2 = 0)_{X'}$.
We see that $S_{\lambda}$ is normal for a general $\lambda$.
We set 
\[
\bar{F}'_{\lambda,\mu} := F (x_0,\lambda x_0,y,z,\mu x_0^2) \in \mbC [x_0,y,z].
\]
We see that $\bar{F}'_{\lambda,\mu}$ is divisible by $x_0$ but not by $x_0^2$ since $y^2 z x_0 \in g_8$ or $y^2 z x_1 \in g_8$.
We have $T_{\mu} |_{S_{\lambda}} = \Gamma + C_{\lambda,\mu}$, where $\Gamma = (x_0 = x_1 = w = 0)$ and
\[
C_{\lambda,\mu} = (x_1 - \lambda x_0 = w - \mu x_0^2 = \bar{F}'_{\lambda,\mu}/x_0 = 0).
\]
Let $Z' \to Y'$ be the Kawamata blowup of $Y'$ at the $\frac{1}{3} (1,1,2)$ point that is the inverse image of $\msp_3$ by $\varphi$ with exceptional divisor $F \cong \mbP (1,1,2)$.
Let $W \to Z'$ be the Kawamata blowup of $Z'$ at the $\frac{1}{2} (1,1,1)$ point lying on $F$ with exceptional divisor $G$.
For a curve or a divisor $\Delta$ on $X'$, $Y'$ or $Z'$, we denote by $\hat{\Delta}$ the proper transform of $\Delta$ on $W$.

\begin{Lem}
We have
\[
(-K_{Y'} \cdot \tilde{C}_{\lambda,\mu}) = \frac{2}{3}, (-K_W \cdot \hat{C}_{\lambda,\mu}) = 0, (-K_W \cdot \hat{\Gamma}) = -1
\]
and
\[ 
(\hat{E} \cdot \hat{C}_{\lambda,\mu}) = 1, (\hat{F} \cdot \hat{C}_{\lambda,\mu}) = 0, (G \cdot \hat{C}_{\lambda,\mu}) = 1.
\]
\end{Lem}

\begin{proof}
In this proof, we write $S = S_{\lambda}$, $T = T_{\mu}$ and $C = C_{\lambda,\mu}$ for simplicity. 
We see that $\tilde{\Gamma}$ intersects $E$ at one point so that $(E \cdot \tilde{\Gamma}) = 1$.
For a curve or a divisor $\Delta$ on $X'$ or $Y'$, we denote by $\breve{\Delta}$ the proper transform of $\Delta$ on $Z'$.
We see that $\breve{\Gamma}$ intersects $F$ at the $\frac{1}{2} (1,1,1)$ point.
We compute the intersection number $(F \cdot \breve{\Gamma})$ by considering a suitable weighted blowup of the ambient space of $Y'$.
We may choose $x_0,x_1,y,z$ as local orbifold coordinates of $Y'$ at the $\frac{1}{3} (1,1,2)$ point.
The weighted blowup of the ambient space with weight $\wt (x_0,x_1,y,w) = \frac{1}{3} (1,1,2,2)$ restricts to the Kawamata blowup $Z' \to Y'$.
Since $\tilde{\Gamma}$ is defined by $(x_0 = x_1 = w = 0)$, we have
\[
(F \cdot \breve{\Gamma}) = (F \cdot - \frac{1}{3} F \cdot - \frac{1}{3} F \cdot - \frac{2}{3} F) = - \frac{2}{3^3} \times \frac{(-3)^3}{2 \times 2} = \frac{1}{2}.
\]
In the above equation, we think of $F$ as the exceptional divisor of the weighted blowup of the ambient space, which is isomorphic to $\mbP (1,1,2,2)$.
We see that $\hat{\Gamma}$ intersects $G$ at one point so that $(G \cdot \hat{\Gamma}) = 1$.

We have $(-K_{Y'} \cdot \tilde{\Gamma}) = (-K_{X'} \cdot \Gamma) - \frac{1}{2} (E \cdot \tilde{\Gamma}) = - \frac{1}{3}$.
Similarly, we have $(-K_Z \cdot \breve{\Gamma}) = (-K_{Y'} \cdot \tilde{\Gamma}) - \frac{1}{3} (F \cdot \breve{\Gamma}) = - \frac{1}{2}$ and $(-K_Z \cdot \hat{\Gamma}) = (-K_Z \cdot \breve{\Gamma}) - \frac{1}{2} (G \cdot \hat{\Gamma}) = -1$.
Since $\tilde{S} \sim_{\mbQ} -K_{Y'}$, $\tilde{T} \sim_{\mbQ} - 2 K_{Y'}$, $\tilde{T} |_{\tilde{S}} = \tilde{\Gamma} + \tilde{C}$ and $K_{Y'} = K_{X'} + \frac{1}{2} E$, we have
\[
\begin{split}
2 (-K_{Y'})^3 &= (-K_{Y'} \cdot \tilde{T} \cdot \tilde{S}) = (-K_{Y'} \cdot \tilde{\Gamma}) + (-K_{Y'} \cdot \tilde{C}), \\
(-K_{Y'} \cdot \tilde{C}) &= (-K_{X'} \cdot C) - \frac{1}{2} (E \cdot \tilde{C}).
\end{split}
\]
It follows that $(-K_{Y'} \cdot \tilde{C}) = 2/3$ and $(E \cdot \tilde{C}) = 1$ since $(E^3) = 4$ and $(-K_{Y'})^3 = (-K_{X'})^3 - (1/2^3) (E^3) = 1/6$.
Note that $(\hat{E} \cdot \hat{C}) = (E \cdot \tilde{C}) = 1$ since $W \to Y'$ is an isomorphism over an open set which entirely contains $E$.
Since $\breve{S} \sim_{\mbQ} -K_Z$, $\breve{T} \sim_{\mbQ} -2K_Z$, $\breve{T} |_{\breve{S}} = \breve{\Gamma} + \breve{C}$ and $K_Z = K_{Y'} + \frac{1}{3} F$, we have
\[
\begin{split}
2 (-K_Z)^3 &= (-K_Z \cdot \breve{T} \cdot \breve{S}) = (-K_Z \cdot \breve{\Gamma}) + (-K_Z \cdot \breve{C}), \\
(-K_Z \cdot \breve{C}) &= (-K_{Y'} \cdot \tilde{C}) - \frac{1}{3} (F \cdot \breve{C}).
\end{split}
\]
It follows that $(-K_Z \cdot \breve{C}) = (F \cdot \breve{C}) = 1/2$ since $(F^3) = 9/2$ and $(-K_Z)^3 = (-K_{Y'})^3 - (1/3^3) (F^3) = 0$.
Similarly, since $\hat{S} \sim_{\mbQ} -K_W$, $\hat{T} \sim_{\mbQ} -2K_W$, $\hat{T} |_{\hat{S}} = \hat{\Gamma} + \hat{C}$ and $K_W = K_Z + \frac{1}{2} G$, we have
\[
\begin{split}
2 (-K_Z)^3 &= (-K_Z \cdot \hat{T} \cdot \hat{S}) = (-K_Z \cdot \hat{\Gamma}) + (-K_W \cdot \hat{C}), \\
(-K_W \cdot \hat{C}) &= (-K_Z \cdot \breve{C}) - \frac{1}{2} (G \cdot \hat{C}).
\end{split}
\]
It follows that $(-K_W \cdot \hat{C}) = 0$ and $(G \cdot \hat{C}) = 1$ since $(G^3) = 4$ and $(-K_W)^3 = (-K_Z)^3 - (1/2^3) (G^3) = - 1/2$.
Finally, the pullback of $F$ on $W$ is the divisor $\hat{F} + \frac{1}{2} G$, hence we have $(F \cdot \breve{C}) = (\hat{F} \cdot \hat{C}) + \frac{1}{2} (G \cdot \hat{C})$, which implies $(\hat{F} \cdot \hat{C}) = 0$.
\end{proof}

\begin{Lem} \label{lem:invbirexclG18}
If, for a general $\lambda \in \mbC$, there is $\mu \in \mbC$ $($depending on $\lambda$$)$ such that $C_{\lambda,\mu}$ is reducible, then $\msp$ is not a maximal center.
\end{Lem}

\begin{proof}
We can write
\[
\bar{F}'_{\lambda,\mu}/x_0 = \alpha x_0 z^2 + (\beta y^2 + \gamma y x_0^2 + \delta x_0^4) z + \mu y^3 x_0 + \varepsilon y^2 x_0^3 + \eta y x_0^5 + \theta x_0^7,
\]
for some $\alpha,\beta,\dots,\theta \in \mbC$.
Note that $\alpha,\dots,\theta$ depend on $\lambda$ and $\mu$, and $\beta$ depends only on $\lambda$.
Let $\lambda$ be a general complex number so that $\beta \ne 0$ and take $\mu \in \mbC$ such that $C_{\lambda,\mu}$ is reducible.
Since $\bar{F}'_{\lambda,\mu}/x_0$ is reducible, we have 
\[
\bar{F}'_{\lambda,\mu}/x_0 = (z + e_3) (\alpha x_0 z + e_4)
\]
for some $e_3, e_4 \in \mbC [x_0,y]$.
We have $y^2 \in e_4$ since $\beta \ne 0$.
It follows that $C_{\lambda,\mu} = \Delta_1 + \Delta_2$, where $\Delta_1 = (x_1 - \lambda x_0 = w - \mu x_0^2 = z + e_3 = 0)$ and $\Delta_2 = (x_1 - \lambda x_0 = w - \mu x_0^2 = \alpha x_0 z + e_4 = 0)$.
Note that $\Delta_1$ is irreducible and $\Delta_2$ does not pass through $\msp$ since $y^2 \in e_4$.
Thus $(-K_{Y'} \cdot \tilde{\Delta}_2) = (-K_{X'} \cdot \Delta_2) = 2/3$.
This implies 
\[
(-K_{Y'} \cdot \tilde{\Delta}_1) = (-K_{Y'} \cdot \tilde{C}_{\lambda,\mu}) - (-K_{Y'} \cdot \tilde{\Delta}_2) = 0.
\]
Clearly we have $(E \cdot \tilde{\Delta}_1) > 0$.
Therefore, there are infinitely many irreducible curves on $Y'$ that intersect $-K_{Y'}$ non-positively and $E$ negatively.
By \cite[Lemma 2.18]{Okada2}, $\msp$ is not a maximal center.
\end{proof}

Let $\pi \colon X' \ratmap \mbP (1_{x_0},1_{x_1},2_w)$ be the projection which is defined outside $\Gamma$.
Let $\mcH \subset \left| - 2 K_{X'} \right|$ be the linear system on $X'$ generated by $x_0^2$, $x_0 x_1$, $x_1^2$ and $w$, and let $\mcH_{Y'}$, $\mcH_W$ be the proper transform of $\mcH$ on $Y'$, $W$, respectively.
We see that $\mcH_{Y'} = \left| -2 K_{Y'} \right|$ and $\mcH_W = \left| -2 K_W \right|$.
Let $\pi_{\lambda} \colon S_{\lambda} \ratmap \mbP (1,2) \cong \mbP^1$ be the restriction of $\pi$ to $S_{\lambda}$ and $\hat{\pi} \colon \hat{S}_{\lambda} \ratmap \mbP^1$ be the composite of $(\varphi' \circ \psi)|_{\hat{S}_{\lambda}} \colon \hat{S}_{\lambda} \to S_{\lambda}$ and $\pi_{\lambda}$.
We set $\hat{E}_{\lambda} = \hat{E} |_{\hat{S}_{\lambda}}$, $\hat{F}_{\lambda} = \hat{F}|_{\hat{S}_{\lambda}}$ and $\hat{G}_{\lambda} = G |_{\hat{S}_{\lambda}}$.
Note that $\msp$ and $\msp_3$ are the indeterminacy points of $\pi_{\lambda}$.

\begin{Lem}
The base locus of $\mcH_W$ is the curve $\Gamma$ and the pair $(W,\frac{1}{2} \mcH_W)$ is canonical.
\end{Lem}

\begin{proof}
It is straightforward to see that $\Bs \mcH_W = \Gamma$.
We see that $W$ is nonsingular along $\Gamma$, a general member of $\mcH_W$ vanishes along $\Gamma$ with multiplicity $1$ and the blowing-up of $W$ along $\Gamma$ resolves the base locus of $\mcH_W$.
This shows that $(X,\frac{1}{2} \mcH_W)$ is canonical (in fact, terminal).
\end{proof}

\begin{Lem}
The intersection matrix of curves in $(x_0 = 0)|_{S_{\lambda}}$ is non-degenerate.
\end{Lem}

\begin{proof}
We have $(x_0 = 0)_{S_{\lambda}} = \Gamma + \Delta$, where $\Delta = (x_0 = x_1 = y^3 + z^2 = 0)$.
We have $(\Gamma \cdot \Delta) = 1$ since $\Gamma$ intersects $\Delta$ at one nonsingular point.
Since $(-K_{X'})|_{S_{\lambda}} \sim_{\mbQ} (x_0 = 0)|_{S_{\lambda}} = \Gamma + \Delta$, we have $1/6 = (-K_{X'} \cdot \Gamma) = (\Gamma^2) + (\Gamma \cdot \Delta)$ and $1/2 = (-K_{X'} \cdot \Delta) = (\Gamma \cdot \Delta) + (\Delta^2)$.
It follows that the intersection matrix
\[
\begin{pmatrix}
(\Gamma^2) & (\Gamma \cdot \Delta) \\
(\Gamma \cdot \Delta) & (\Delta^2)
\end{pmatrix} =
\begin{pmatrix}
-5/6 & 1 \\
1 & -1/2 
\end{pmatrix}
\]
is non-degenerate.
\end{proof}

By \cite[Lemma 7.2]{Okada2}, there is a birational involution of $X'$ centered at $\msp$.
As a conclusion, we have the following.

\begin{Thm} \label{thm:birinvG18-2}
Let $X'$ be a member of $\mcG'_{18}$ and $\msp$ a singular point of type $\frac{1}{2} (1,1,1)$.
Then either $\msp$ is not a maximal center or there is a birational involution of $X'$ which is a Sarkisov link centered at $\msp$.
\end{Thm}

\section{Birational involutions centered at $cA/n$ points} \label{sec:birinvcA}

Let $X'$ be a member of $\mcG'_i$ with $i \in I^*_{cA/n}$ and $\msp = \msp_4$ the $cA/n$ point.
We treat $6$ families, that is,
\[
i \in \{10,26,33,38,48,63\},
\]
and show that there is a birational involution of $X'$ starting with a divisorial extraction $\varphi$ marked B.I. in the big table.

We briefly explain the argument of this section.
Let $\varphi \colon Y' \to X'$ be an extraction marked B.I.
We first give an explicit global construction of $\varphi$, which enables us to give an explcit construction of anticanonical model $Z'$ of $Y'$.
We observe that the anticanonical map $Y' \ratmap Z'$ is a birational morphism and $Z'$ admits a double cover onto a suitable WPS.
By \cite[Lemma 3.2]{Okada2}, we conclude that either $\msp$ is not a maximal center or there is a birational involution starting with $\varphi$.

\subsection{Families $\mcG_i$ for $i \in \{10,26,38,48,63\}$}

We first treat the case where $i \in \{10,26,48\}$.
The standard polynomial of $X' \subset \mbP (1,n,a_2,a_3,n)$ can be written as
\[
F' = w^2 x_2 x_3 + w (x_2^2 a + x_2 b + c) + x_2^3 + x_2^2 d + x_2 e + h,
\]
where $a,b,\dots,e,h \in \mbC [x_0,x_1,x_3]$.
Note that we have $a_3 + 2 n = 2 a_2$.
Note also that we have $n = 1$ in this case but we use $n$ for a unified exposition.

We construct a divisorial extraction that corresponds to $(a_2-n,a_3+2n)$-blowup.
Filtering off terms divisible by $x_2$, the defining equation of $X'$ is written as
\begin{equation} \label{eq:BI-1}
F' = x_2 u + w c + h = 0,
\end{equation}
where
\begin{equation} \label{eq:BI-2}
u := w^2 x_3 + w (x_2 a + b) + x_2^2 + x_2 d + e.
\end{equation}
Note that $\deg u = a_3 + 2 n$.
Let $\mbP := \mbP (1,n,a_2,a_3,n,a_3 + 2 n)$ be the WPS with coordinates $x_0,\dots,x_3,w$ and $u$.
Then, $X'$ is a weighted complete intersection, defined by the equations \eqref{eq:BI-1} and \eqref{eq:BI-2}.
Let $\Phi \colon W \to \mbP$ be the weighted blowup at $(0 \!:\! 0 \!:\! 0 \!:\! 0 \!:\! 1 \!:\! 0)$ with 
\[
\wt (x_0,x_1,x_2,x_3,u) = \frac{1}{n} (1,n,a_2-n,a_3,a_3 + 2 n)
\]
and $Y'$ the proper transform of $X'$ by $\Phi$.
We denote by $\varphi \colon Y' \to X'$ the induced birational morphism and by $E$ its exceptional divisor.

\begin{Lem}
The weighted blowup $\varphi \colon Y' \to X'$ is a divisorial extraction centered at $\msp$.
\end{Lem}

\begin{proof}
We will show that $Y'$ has only terminal quotient singularities, which complete that proof.
We see that $X'$ is defined by 
\[
x_2 u + w c + h = - u+ w^2 x_3 + w (x_2 a + b) + x_2^2 + x_2 d + e = 0
\]
in $\mbP$ and we have an isomorphism
\[
E \cong (x_2 u + c = x_3 + x_2 a + x_2^2 = 0) \subset \mbP (1_{x_0},n_{x_1},(a_2-n)_{x_2},{a_3}_{x_3},(a_3+2n)_u).
\]
We see that $\deg a = a_2 - n < a_3$ and hence $a$ does not involve the variable $x_3$.
We have
\[
J_{\varphi} =
\begin{pmatrix}
\frac{\prt c}{\prt x_0} & \frac{\prt c}{\prt x_1} & u & \frac{\prt c}{\prt x_3} & x_2 & h \\[1.5mm]
x_2 \frac{\prt a}{\prt x_0} & x_2 \frac{\prt a}{\prt x_1} & a + 2 x_2 & 1 & 0 & b + x_2 d
\end{pmatrix}.
\]
We see that $J_{\varphi}$ is of rank $2$ outside the set
\[
\begin{split}
\Sigma &:= \left(x_2 = \frac{\prt c}{\prt x_0} = \frac{\prt c}{\prt x_1} = u - a \frac{\prt c}{\prt x_3} = h - b \frac{\prt c}{\prt x_3} = 0 \right) \cap E \\
&= \left(x_2 = x_3 = c = \frac{\prt c}{\prt x_0} = h - b \frac{\prt c}{\prt x_3} = \frac{\prt c}{\prt x_1} = u - a \frac{\prt c}{\prt x_3} = 0 \right).
\end{split}
\] 

We claim that the system of equations 
\[
x_3 = c = \frac{\prt c}{\prt x_0} = \frac{\prt c}{\prt x_1} = h - b \frac{\prt c}{\prt x_3} = 0
\]
has no non-trivial solution.
Assume to the contrary that the above equations have a common solution $(x_0,x_1,x_3) = (\alpha_0,\alpha_1,0) \ne (0,0,0)$. 
Let $X$ be the birational counterpart of $X'$ which is defined by
\[
\begin{split} 
F_1 &= x_5 x_2 + x_4 x_3 + (x_2^2 a + x_2 b + c) = 0, \\
F_2 &= x_5 x_4 - (x_2^3 + x_2^2 d + x_2 e + h) = 0,
\end{split}
\]
in $\mbP (1,n,a_2,a_3,a_2+n,a_3+n)$ and set $\msq := (\alpha_0 \!:\! \alpha_1 \!:\! 0 \!:\! 0 \!:\! - \bar{c}' \!:\! - \bar{b})$, where $\bar{c}' = (\prt c/\prt x_3) (\alpha_0,\alpha_1,0)$ and $\bar{b} = b (\alpha_0,\alpha_1,0)$.
Then, $\msq \in X$ and $\prt F_1 / \prt x_i$ vanishes at $\msq$ for $i = 0,1,\dots,5$.
Thus $X$ is not quasismooth at $\msq$ and this is a contradiction.
The above claim implies that $\Sigma = \emptyset$.
By Lemma \ref{lem:singwbl}, $Y'$ has only cyclic quotient singularities.
Straightforward computations in each instance show that $Y'$ has only terminal quotients singularities (see Table \ref{table:singY'} for the singularities of $Y'$ along $E$), which implies that $\varphi$ is a divisorial extraction.
\end{proof}

\begin{table}[tb]
\begin{center}
\caption{Singularities of $Y'$ along $E$}
\label{table:singY'}
\begin{tabular}{cc}
\hline
No. & Singular points \\
\hline \\[-3.5mm]
10 & $1 \times \frac{1}{3} (1,1,2)$ \\[0.5mm]
26 & $1 \times \frac{1}{2} (1,1,1)$ \\[0.5mm]
38 & $1 \times \frac{1}{3} (1,1,2)$ if $z^3 \notin f_9$, $1 \times \frac{1}{8} (1,3,5)$  \\[0.5mm]
48 & $1 \times \frac{1}{3} (1,1,2)$, $1 \times \frac{1}{8} (1,1,7)$ \\[0.5mm]
63 & $1 \times \frac{1}{3} (1,1,2)$ if $y^4 \in f_{12}$, $1 \times \frac{1}{10} (1,3,7)$
\end{tabular}
\end{center}
\end{table}

We see that $x_0, x_1,x_3$ and $u$ lift to plurianticanonical sections on $Y'$.
We construct one more section $v$ that lifts to a plurianticanonical section on $Y'$ and then determine the anticanonical map of $Y'$.
Multiplying $u$ to \eqref{eq:BI-2} and then eliminating $x_2 u = - w c - h$ by \eqref{eq:BI-1}, we have
\begin{equation} \label{eq:BI-3}
w v - u^2 + u e - h (x_2 + d) = 0,
\end{equation}
where
\begin{equation} \label{eq:BI-4}
v := w (u x_3 - a c) + u b - c x_2 - c d - a h.
\end{equation}

We have $\deg v = 4 a_2 - n$.
The sectins $u^2$, $u e$ and $h d$ vanish along $E$ to order $4 a_2/n$, and the section $h x_2$ vanishes along $E$ to order $(4 a_2 - n)/n$.
Hence, by \eqref{eq:BI-3}, $v$ vanishes along $E$ to order at least $(4 a_2-n)/n = \deg v/n$.
This shows that $v$ lifts to a plurianticanonical section.
Let
\[
\psi \colon X' \ratmap \mbP (1_{x_0},n_{x_1},{a_3}_{x_3},(a_3 + 2 n)_u, (4 a_2 - n)_v)
\] 
the rational map defined by plurianticanonical sections $x_0,x_1,x_3,u,v$.
We see that the intersection of zero loci of proper transforms on $Y'$ of $(x_0 = 0)_{X'}$, $(x_1 = 0)_{X'}$, $(x_3 = 0)_{X'}$, $(u=0)_{X'}$ and $(v=0)_{X'}$ is empty.
This implies that $\psi$ is a morphism. 
The equations \eqref{eq:BI-1}, \eqref{eq:BI-3} and \eqref{eq:BI-4} can be expressed as
\[
M \cdot 
\begin{pmatrix}
x_2 \\
w \\
1
\end{pmatrix} =
\begin{pmatrix}
0 \\
0 \\
0
\end{pmatrix},
\]
where
\[
M =
\begin{pmatrix}
u & c & h \\
- h & v & - u^2 + u e - h d \\
c & - u x_3 + a c & v - u b + c d + a h
\end{pmatrix}.
\] 
We see that $\det M$ is divisible by $u$ and $G := \det M/u$ is quadratic with respect to $v$.
It follows that $\psi$ is birational and the image of $\psi$ is the weighted hypersurface $Z'$ in $\mbP (1,n,a_3,a_3 + 2 n,4 a_2 - n)$ defined by the equation $G = 0$.
Moreover, the projection to the coordinates $x_0,x_1,x_3,u$ defines a double cover $\pi \colon Z' \to \mbP (1,n,a_3,a_3+2 n)$.

We explain that the birational morphism $\psi \colon Y' \to Z'$ is not an isomorphism.
Let $\tilde{S}$ be the proper transform on $Y'$ of $S := (c = h = u = 0) \subset X'$.
Here the section $u$, considered as a section on $X'$, is a polynomial in $x_0,x_1,x_2,x_3,$ and $w$.
We see that $\psi (T) = (c = h = u = v = 0) \subset Z'$.
We have $1 \le \dim T \le 2$ and $\dim \psi (T) = \dim T - 1$.
This shows that $\psi$ cannot be an isomorphism. 

By \cite[Lemma 3.2]{Okada2}, either $\varphi$ is not a maximal extraction or there is a birational involution of $X'$ which is a Sarkisov link starting with $\varphi$.

We next treat the case where $i \in \{38,63\}$.
The standard polynomial of $X' \subset \mbP (1,n,a_2,a_3,n)$ can be written as
\[
F' = w^2 x_2 x_3 + w (x_3^2 a + x_3 b + c) + x_3^3 + x_3^2 d + x_3 e + h,
\]
where $a,b,\dots,e,h \in \mbC [x_0,x_1,x_2]$.
Note that $a_2 + 2 n = 2 a_2$ and $n = 3$ in this case.
We have a symmetry between families $\mcG'_i$ with $i \in \{10,26,48\}$ and with $i \in \{38,63\}$: the situation coincides after interchanging the role of $x_2$ and $x_3$.
Thus, by the symmetric argument, we can construct the section $u = w^2 x_2 + w (x_3 a + b) + x_3^2 + x_3 d + e$ of degree $a_2 + 2 n$, and the weighted blowup $\varphi \colon Y' \to X'$ with
\[
\wt (x_0,x_1,x_2,x_3,u) = \frac{1}{n} (1,n,a_2,a_3-n,a_2+2 n),
\]
which is a divisorial extraction (see Table \ref{table:singY'} for the singularities of $Y'$ along $E$).
Moreover, we have the anticanonical morphism $\psi \colon Y' \to Z'$ whose base admits a double cover $Z' \to \mbP (1,n,a_2,a_2+2n)$.
By \cite[Lemma 3.2]{Okada2}, either $\varphi$ is not a maximal extraction or there is a birational involution of $X'$ which is a Sarkisov link starting with $\varphi$.
Therefore, we have the following.

\begin{Thm} \label{thm:birinvcA1}
Let $X'$ be a member of $\mcG'_i$ with $i \in \{10,26,38,48,63\}$, $\msp = \msp_4$ the $cA/n$ point, and let $\varphi_{(k,l)} \colon Y'_{(k,l)} \to X'$ an extremal divisorial extraction centered at $\msp$ marked \emph{B.I.} in the big table.
Then, one of the following holds.
\begin{enumerate}
\item $\varphi_{(k,l)}$ is not a maximal extraction.
\item There is a birational involution of $X'$ which is the Sarkisov link starting with $\varphi_{(k,l)}$.
\end{enumerate}
\end{Thm}

\begin{proof}
If $i \in \{38,63\}$ (resp.\ $i \in \{10,26,48\}$), then there is a unique extraction (resp.\ two extractions) marked B.I.
The extraction $\varphi \colon Y' \to X'$ corresponds to (one of) $\varphi_{(k,l)}$ marked B.I.
Thus the proof is completed for $i \in \{38,63\}$.
Suppose that $i \in \{10,26,48\}$.
The proof is completed for $\varphi = \varphi_{(a_2-n,a_3+2 n)}$.
We need to consider the case where $\varphi = \varphi_{(a_3+2n,a_2-n)}$.
If $i = 10$, then, by interchanging the role of $x_2$ and $x_3$ in the above construction, we obtain the $(a_3+2n,a_2-n)$-blowup $\varphi$.
If $i = 26$ (resp.\ $48$), then $\varphi$ is obtained as the composite of $\varphi_{(a_2-n,a_3+2n)}$ and the isomorphism $\nu'$ (resp.\ the automorphism $\mu$) defined in Definition \ref{def:nu} (resp.\ \ref{def:mu}).
Therefore, the proof for $\varphi_{(a_3+2n,a_2-n)}$ follows from that for $\varphi_{(a_2-n,a_3+2n)}$.
\end{proof}

\subsection{Family $\mcG'_{33}$ and $(2,7), (7,2)$-blowups}

Let $X' = X'_{10} \subset \mbP (1,1,3,5,1)$ be a member of $\mcG'_{33}$ and $\msp = \msp_4$.
The aim of this subsection is to construct birational involutions starting with $(2,7)$- and $(7,2)$-blowups.
This construction is a version of that of ``invisible involutions" introduced in \cite{CP} (see also Section \ref{sec:birinv1} and \cite[Section 7]{Okada2}).
We first give an explicit global description of $(2,7)$- and $(7,2)$-blowups $\varphi \colon Y' \to X'$ and then construct a birational involution of a suitable model of $Y'$.
The induced birational involution of $Y'$ turns out to be a composite of inverse flips, flops and flips (in fact, it is a flop), so that it gives the desired Sarkisov link.
 
We can write the defining polynomial of $X'$ as
\begin{equation} \label{eq:BI33-1}
\begin{split}
F' = w^2 y z + w (z y a_1 + z a_4 + & y^3 + y^2 a_3 + y a_6 + a_9) \\
& + z^2 + z y b_2 + z b_5 + y^2 b_4 + y b_7 + b_{10},
\end{split}
\end{equation}
where $a_i, b_i \in \mbC [x_0,x_1]$.
Filtering off terms divisible by $w y$, we obtain
\begin{equation} \label{eq:BI33-2}
F' = w y u + w (z a_4 + a_9) + z^2 + z y b_2 + z b_5 + y^2 b_4 + y b_7 + b_{10},
\end{equation}
where
\begin{equation} \label{eq:BI33-3}
u := w z + z a_1 + y^2 + y a_3 + a_6.
\end{equation}
Multiplying $F'$ by $w$, eliminating $wz$ by the equation \eqref{eq:BI33-3} and then filtering off terms divisible by $y$, we obtain
\begin{equation} \label{eq:BI33-4}
w F' = y v + w^2 a_9 + (w a_4 + z + b_5)(u - z a_1 - a_6) + w b_{10},
\end{equation}
where
\begin{equation} \label{eq:BI33-5}
\begin{split}
v := w^2 u - (w a_4 + & z + b_5) (y + a_3) \\ 
& + b_2 (u - z a_1 - y^2 - y a_3 - a_6) + w y b_4 + w b_7.
\end{split}
\end{equation}
Let $U$ be the open subset subset of $X'$ where $w = 1$.
We see that $U$ is naturally isomorphic to the subvariety of $\mbA^6$ with affine coordinates $x_0,x_1,y,z,u$ and $v$ defined by three polynomials which are obtained by setting $w = 1$ in \eqref{eq:BI33-3}, \eqref{eq:BI33-4} and \eqref{eq:BI33-5}.
Let $\varphi \colon Y' \to X'$ be the weighted blowup of $X'$ with
\[
\wt (x_0,x_1,y,z,u,v) = (1,1,2,4,6,7)
\]
and let $E$ be the exceptional divisor of $\varphi'$.

\begin{Lem}
The weighted blowup $\varphi \colon Y' \to X'$ is a divisorial extraction centered at $\msp$.
\end{Lem}

\begin{proof}
We have
\[
E \cong (z + y^2 = y v + a_9 - (a_4 + z) z a_1 = u - (a_4 + z)y - y^2 b_2 + y b_4 = 0),
\]
where the right-hand side is a WCI in $\mbP (1_{x_0},1_{x_1},2_y,4_z,6_u,7_v)$ and
\[
J_{\varphi} =
\begin{pmatrix}
0 & 0 & 2 y & 1 & 0 & 0 & * \\
\frac{\prt a_9}{\prt x_0} + * & \frac{\prt a_9}{\prt x_1} + * & v & - a_1 (2 z + a_1) & 0 & y & b_{10} - a_4 a_6 + * \\
* & * & b_4 + * & - y & 1 & 0 & - v + b_7 + *
\end{pmatrix},
\]  
where $*$ means a polynomial that is contained in the ideal $(y,z)$.
By an explicit computation, we see that $J_{\psi}$ is of rank $2$ outside the set
\[
\Sigma := \left( y = z = u = v = a_9 = \frac{\prt a_9}{\prt x_0} = \frac{\prt a_9}{\prt x_1} = b_{10} - a_4 b_6 = 0 \right) \subset \mbP (1,1,2,4,6,7).
\]
We show that the system of equations
\[
a_9 = \frac{\prt a_9}{\prt x_0} = \frac{\prt a_9}{\prt x_1} = b_{10} - a_4 a_6 = 0
\]
does not have a non-trivial solution, which will imply $\Sigma = \emptyset$.
We assume that it has a non-trivial solution $(x_0,x_1) = (\alpha_0,\alpha_1)$. 
Set $\alpha_i = a_i (\alpha_0,\alpha_1)$ for $i = 4,6,9$ and $\beta_{10} = b_{10} (\alpha_0,\alpha_1)$.
Let $X \in \mcG_{33}$ be the birational counterpart of $X'$, which is defined by
\[
\begin{split}
F_1 &= t y + s z + (z y a_1 + z a_4 + y^3 + y^2 a_3 + y a_6 + a_9) = 0, \\
F_2 &= t s - (z^2 + z y b_2 + z b_5 + y^2 b_4 + y b_7 + b_{10}) = 0,
\end{split}
\]
in $\mbP (1_{x_0},1_{x_1},3_y,5_z,4_s,6_t)$.
Set $\msq = (\alpha_0 \!:\! \alpha_1 \!:\! 0 \!:\! 0 \!:\! - \alpha_4 \!:\! - \alpha_6)$.
We have $\msq \in X$ since $\beta_{10} - \alpha_4 \alpha_6 = \alpha_9 = 0$.
It is easy to verify that every partial derivative of $F_1$ vanishes at $\msq$, which implies that $X$ is not quasismooth.
This is a contradiction.

It follows that $\Sigma = \emptyset$ and thus $Y'$ has only cyclic quotient singular points.
It is then easy to see that $Y'$ has singular points of type $\frac{1}{2} (1,1,1)$ and $\frac{1}{7} (1,1,6)$ at $(0 \!:\! 0 \!:\! 1 \!:\! -1 \!:\! -1 \!:\! 0)$ and $(0 \!:\! 0 \!:\! 0 \!:\! 0 \!:\! 0 \!:\! 1)$, respectively.
It follows that $Y'$ has only terminal singularities and thus $\varphi$ is a divisorial extraction.
\end{proof}

We see that $\varphi = \varphi_{(2,7)}$ is a divisorial extraction which is a  $(2,7)$-blowup.
Before going to the construction of birational involutions, we construct $(7,2)$-blowup by taking the composite of $(2,7)$-blowup and an automorphism of $X'$.

\begin{Def}
Let $X'$ be a member of $\mcG'_{33}$ with defining polynomial $F' = w^2 y z + w (z h_4 + h_9) + z^2 + z h_5 + h_{10}$, where $h_j \in \mbC [x_0,x_1,y]$.
We define $\mu$ to be the automorphism of $X'$ defined by the replacement $z \mapsto - z - w^2 y - w h_4 - h_5$.
\end{Def}

We see that $\mu (\msp) = \msp$ and the composite $\mu \circ \varphi_{(2,7)}$ is the $(7,2)$-blowup. 

We return to the case of $(2,7)$-blowup $\varphi = \varphi_{(2,7)}$.
Let $\psi \colon W \to Y'$ be the Kawamata blowup of $Y'$ at the $\frac{1}{7} (1,1,6)$ point lying on $E$ and let $F \cong \mbP (1,1,6)$ be its exceptional divisor.
For $\lambda,\mu \in \mbC$, we set $S_{\lambda} := (x_1 - \lambda x_0 = 0)_{X'}$ and $T_{\mu} := (u - \mu x_0^6 = 0)_{X'}$.
We see that $S_{\lambda}$ is normal for a general $\lambda \in \mbC$ and we define $C_{\lambda,\mu}$ to be the scheme-theoretic intersection $S_{\lambda} \cap T_{\mu}$.
Let $\mcM$ be the linear system on $X'$ generated by $x_0^6, x_0^5 x_1, \dots, x_1^6$ and $u$, and let $\mcM_W$ be its proper transform on $W$.
We see that $\mcM_W \subset \left| - 6 K_{W} \right|$ (in fact, equality holds) is base point free.
Let $\eta \colon W \to \mbP (1,1,6)$ be the morphism defined by $\mcM_W$, which resolves the indeterminacy of the projection $\pi \colon X' \ratmap \mbP (1_{x_0}, 1_{x_1},6_u)$.
For a curve of a divisor $\Delta$ on $X'$ or $Y'$, we denote by $\hat{\Delta}$ the proper transform of $\Delta$ on $W$.
Note that $\hat{C}_{\lambda,\mu}$ is the fiber of $\eta$ over the point $(1 \!:\! \lambda \!:\! \mu)$.

\begin{Lem}
We have
\[
(\hat{E} \cdot \hat{C}_{\lambda,\mu}) = 3, (F \cdot \hat{C}_{\lambda,\mu}) = 1, (-K_W \cdot \hat{C}_{\lambda,\mu}) = 0.
\]
\end{Lem}

\begin{proof}
Since $K_{Y'} = \varphi^*K_{X'} + E$ and $K_W = \psi^*K_{Y'} + \frac{1}{7} F$, we have 
\[
\begin{split}
(-K_{Y'}^3) &= (-K_{X'}^3) - (E^3) = \frac{2}{3} - \frac{9}{14} = \frac{1}{42}, \\
(-K_{W'}^3) &= (-K_{Y'}^3) - \frac{1}{7^3} (F^3) = \frac{1}{42} - \frac{1}{42} = 0.
\end{split}
\]
Since $\tilde{S}_{\lambda} \sim_{\mbQ} - K_{Y'}$, $\tilde{T}_{\mu} \sim_{\mbQ} -6 K_{Y'}$ and $\tilde{T}_{\mu} |_{\tilde{S}_{\lambda}} = \tilde{C}_{\lambda,\mu}$, we have
\[
(-K_{Y'} \cdot \tilde{C}_{\lambda,\mu}) = (-K_{Y'} \cdot \tilde{T}_{\mu} \cdot \tilde{S}_{\lambda}) = 6 (-K_{Y'})^3 = \frac{1}{7},
\]
which implies
\[
(E \cdot \tilde{C}_{\lambda,\mu}) = (-K_{X'} \cdot \tilde{C}_{\lambda,\mu}) - (-K_{Y'} \cdot C_{\lambda,\mu}) = 4 - \frac{1}{7} = \frac{27}{7}.
\]
Similarly, since $\hat{S}_{\lambda} \sim_{\mbQ} - K_W$, $\hat{T}_{\mu} \sim_{\mbQ} - 6 K_W$ and $\hat{T}_{\mu} |_{\hat{S}_{\lambda}}$, we have
\[
(-K_W \cdot \hat{C}_{\lambda,\mu}) = (-K_W \cdot \hat{T}_{\mu} \cdot \hat{S}_{\lambda}) = 6 (-K_W)^3 = 0,
\]
which implies
\[
(F \cdot \hat{C}_{\lambda,\mu}) = 7 ( (-K_{Y'} \cdot \tilde{C}_{\lambda,\mu}) - (-K_W \cdot \hat{C}_{\lambda,\mu})) = 1.
\]
Finally, we have $\psi^*E = \hat{E} + \frac{6}{7} F$ and by taking the intersection number with $\hat{C}_{\lambda,\mu}$, we have $(\hat{E} \cdot \hat{C}_{\lambda,\mu}) = 3$.
\end{proof}

\begin{Lem} \label{lem:birinvexclG33}
Suppose that, for a general $\lambda \in \mbC$, there is $\mu \in \mbC$ $($depending on $\lambda$$)$ such that $C_{\lambda,\mu}$ is reducible.
Then, $\varphi$ is not a maximal extraction.
\end{Lem}

\begin{proof}
Assume that $C = C_{\lambda,\mu}$ is reducible.
Then, there is a unique component $C^{\circ}$ of $C$ such that $(G \cdot \hat{C}^{\circ}) = 1$.
Let $C'$ be any component of $C$ other than $C^{\circ}$.
Then $\hat{C}'$ is disjoint from $G$ and we have $(-K_{W'} \cdot \hat{C}') = 0$ since $-K_{W'}$ is nef and $(-K_{W'} \cdot \hat{C}) = 0$.
It follows that $(-K_{Y'} \cdot \tilde{C}') = (-K_W \cdot \hat{C}') + \frac{1}{7} (F \cdot \hat{C}') = 0$.
We have $(E \cdot \tilde{C}') = (K_{Y'} \cdot \tilde{C}') + (-K_{X'} \cdot C') > 0$.
This shows that there are infinitely many curves on $Y'$ which intersect $-K_{W'}$ non-positively and $E$ positively.
By \cite[Lemma 2.19]{Okada2}, $\varphi'$ is not a maximal singularity.
\end{proof}

\begin{Thm} \label{thm:birinvcA2}
Let $X'$ be a member of $\mcG'_{33}$ and let $\varphi \colon Y' \to X'$ be a $(2,7)$- or a $(7,2)$-blowup centered at $\msp = \msp_4$.
Then, either $\varphi$ is not a maximal extraction, or there is a birational involution of $X'$ that is a Sarkisov link starting with $\varphi$.
\end{Thm}

\begin{proof}
We prove this for $\varphi = \varphi_{(2,7)}$.
The proof for $\varphi_{(7,2)}$ follows by composing $\varphi_{(2,7)}$ with the automorphism $\mu$.
The following argument is based on \cite{CP}.

By Lemma \ref{lem:birinvexclG33}, we may assume that, for a general $\lambda \in \mbC$, $C_{\lambda,\mu}$ is irreducible for every $\mu \in \mbC$.
The morphism $\eta \colon W \to \mbP (1,1,6)$ is an elliptic fibration  and let $\tau_{W} \colon W \ratmap W$ be the birational involution defined as the reflection of the generic fiber with respect to the section $F$.
\[
\xymatrix{
W \ar[d]_{\psi} \ar@{-->}[rr]^{\tau_W} \ar[rdd]^{\eta} & & W \ar[d]^{\psi} \ar[ldd]_{\eta} \\
Y' \ar[d]_{\varphi} & & Y' \ar[d]^{\varphi} \\
X' \ar@{-->}[r]_{\pi \hspace{5mm}} & \mbP (1,1,6) & \ar@{-->}[l]^{\hspace{5mm} \pi} X'}
\]
We see that $\tau_W$ is an isomorphism in codimension $1$ since $K_W$ is $\eta$-nef and it induces birational involutions $\tau_{Y'} \colon Y' \ratmap Y'$ and $\tau \colon X' \ratmap X'$.
Note that $\tau_{Y'}$ is an isomorphism in codimension $1$ since $F$ is $\tau_W$-invariant.

We will show that $\tau$ is not biregular.
Assume to the contrary that $\tau$ is biregular.
We fix a general $\lambda \in \mbC$ so that $C_{\lambda,\mu}$ is irreducible for every $\mu \in \mbC$.
The surface $S_{\lambda}$ is $\tau$-invariant and $\tau$ induces a biregular involution $\tau_{\lambda}$ of $S_{\lambda}$, which induces a birational involution $\hat{\tau}_{\lambda}$ of $\hat{S}_{\lambda}$.
Note that $\hat{\tau}_{\lambda}$ may not be biregular.
Let $\bar{S}_{\lambda} \to \hat{S}_{\lambda}$ be a composite of suitable blowups such that the birational involution $\bar{\tau}_{\lambda}$ of $\bar{S}_{\lambda}$ induced by $\tau_{\lambda}$ is biregular.
We denote by $\sigma \colon \bar{S}_{\lambda} \to S_{\lambda}$ the composite of $\bar{S}_{\lambda} \to \hat{S}_{\lambda}$ and $\varphi |_{\hat{S}_{\lambda}} \colon \hat{S}_{\lambda} \to S_{\lambda}$ and by $\bar{\pi}_{\lambda} \colon \bar{S}_{\lambda} \to \mbP^1$ the composite of $\sigma$ and $\pi_{\lambda} = \pi |_{S_{\lambda}} \colon S_{\lambda} \ratmap \mbP (1,6) \cong \mbP^1$.
Let $\bar{E}_{\lambda}$ and $\bar{F}_{\lambda}$ be the proper transforms of $\hat{E}|_{\hat{S}_{\lambda}}$ and $F|_{\hat{S}_{\lambda}}$ on $\bar{S}_{\lambda}$, respectively, which are the prime $\sigma$-exceptional divisors that are not contracted by $\bar{\pi}_{\lambda}$.
Denote by $G_1,\dots,G_r$ the other prime $\sigma$-exceptional divisors.

Let $\bar{C}_{\lambda} \subset \bar{S}_{\lambda}$ be the proper transform of a general fiber $C_{\lambda}$ of $\pi_{\lambda}$.
Since $\bar{\tau}_{\lambda} |_{\bar{C}_{\lambda}}$ is the reflection with respect to the point $\bar{F}_{\lambda} \cap \bar{C}_{\lambda}$ and $\bar{E}_{\lambda}$ is $\bar{\tau}_{\lambda}$-invariant, $(\bar{E}_{\lambda} - 3 \bar{F}_{\lambda})|_{\bar{C}_{\lambda}} \in \Pic^0 (\bar{C}_{\lambda})$ is a $2$-torsion.
In particular, $\bar{E}_{\lambda} - 3 \bar{F}_{\lambda}$ is numerically equivalent to a linear combination of $\bar{\pi}_{\lambda}$-vertical divisors.

On the other hand, we have $(x_0 = 0)|_{S_{\lambda}} = \Gamma$, where 
\[
\Gamma = (x_0 = x_1 = w^2 y z + w y^3 + z^2 = 0)
\]
is an irreducible and reduced curve.
We see that $\Gamma$ and $\{ C_{\lambda,\mu} \mid \mu \in \mbC \}$ are all the fibers of $\pi_{\lambda}$.
Since $C_{\lambda,\mu}$ is irreducible for every $\mu \in \mbC$ and all the fibers of $\bar{\pi}_{\lambda}$ are numerically equivalent to each other, we have
\[
\bar{E}_{\lambda} - 3 \bar{F}_{\lambda} \sim_{\mbQ} \gamma \bar{\Gamma} + \sum_{i = 1}^r c_i G_i,
\]
for some $\gamma,c_1,\dots,c_r \in \mbC$, where $\bar{\Gamma}$ is the proper transform of $\Gamma$.
We see that $\gamma \ne 0$ since the curves $\bar{E}_{\lambda}, \bar{F}_{\lambda}, G_1, \dots,G_r$ are $\sigma$-exceptional and their intersection form is negative-definite.
This shows $\Gamma \sim_{\mbQ} 0$, which is a contradiction since $(A \cdot \Gamma) \ne 0$.
Therefore, $\tau \colon X' \ratmap X'$ is not biregular and, by \cite[Lemma 2.24]{Okada2}, $\tau$ is a Sarkisov link starting with $\varphi$.
\end{proof}

\begin{Rem}
It is straightforward to see that $-K_{Y'}$ is nef and $\tau_{Y'} \colon Y' \ratmap Y'$ is a flop.
The anticanonical model $Z$ of $Y'$ is a (non-$\mbQ$-factorial) Fano $3$-fold with degree $1/42$.
By looking at the Fletcher's list \cite{IF}, it is quite likely that $Z$ is a weighted hypersurface $Z = Z_{42} \subset \mbP (1,1,6,14,21)$ of degree $42$.
If one can find sections of degree $14$ and $21$ on $X'$ that lift to plurianticanonical sections on $Y'$, then we can construct $Z$ explicitly as in the argument of the previous subsection, and the existence of flop will follow from this.
\end{Rem}

\section{Exclusion of nonsingular points} \label{sec:exclnspt}

The aim of this section is to exclude nonsingular points as maximal center.
Let $X'$ be a member of $\mcG'_i$ with $i \in I^*_{cA/n} \cup I_{cD/3}$.

\begin{Def}[{\cite[Definition 5.2.4]{CPR}}]
Let $\msp \in X'$ be a point.
We say that a Weil divisor class $L$ on $X'$ {\it isolates} $\msp$ if $\msp$ is an isolated component of the linear system
\[
\mcL^s_{\msp} := |\mcI_{\msp} (s L)|
\]
for some integer $s > 0$.
\end{Def}

\begin{Lem}[{\cite{CPR}}] \label{lem:criexclnspt}
Let $\msp \in X'$ be a nonsingular point.
If $l A$ isolates $\msp$ for some $0 < l \le 4/(A^3)$, then $\msp$ is not a maximal center.
\end{Lem}

\begin{proof}
See \cite[Proof of (A)]{CPR}.
\end{proof}

Let $\mbP (a_0,\dots,a_4)$ be the ambient WPS of $X'$.
For $j,m = 0,\dots,4$ with $j \ne m$, we define
\[
\tilde{a}_j := \max_{0 \le k \le 4, k \ne j} \operatorname{lcm} (a_j,a_k), \
\tilde{a}_{j;m} := \max_{0 \le k \le 4, k \ne j, m} \operatorname{lcm} \{a_j,a_k\}.
\]
For $m = 0,\dots,4$, we denote by $\pi_m$ the restriction of the projection from $\msp_m$ to $X'$ and by $\Exc (\pi_m) \subset X'$ the locus contracted by $\pi_m$.

\begin{Prop}[{\cite[Proposition 5.1]{Okada2}}] \label{isolnspt}
Let $\msp = (\xi_0 \!:\! \cdots \!:\! \xi_4)$ be a nonsingular point of $X'$.
Then, the following assertions hold.
\begin{enumerate}
\item If $\xi_j \ne 0$ then $\tilde{a}_j A$ isolates $\msp$.
\item If $\xi_j \ne 0$ and $\msp \notin \Exc (\pi_m)$ for some $m \ne j$, then $\tilde{a}_{j;m} A$ isolates $\msp$.
\end{enumerate}
\end{Prop}

\begin{Thm} \label{thm:nspt}
Let $X'$ be a member of $\mcG'_i$ with $i \in I^*_{cA} \cup I_{cA/n} \cup I_{cD/3}$.
Then no nonsingular point of $X'$ is a maximal center.
\end{Thm}

\begin{proof}
Let $\mbP (a_0,\dots,a_4)$ be the ambient WPS of $X'$ as above and let $F'$ be the defining polynomial of $X'$.
Suppose $i \in \{7,10,18,21,22,36,38,44,52,57,62,63\}$. 
Then the inequality
\[
\tilde{a} := \max_{0 \le k < l \le 4} \operatorname{lcm} (a_j,a_k) \le 4/(A^3)
\]
holds.
Note that $\tilde{a}_j \le \tilde{a} \le 4/(A^3)$ for any $j$.
Thus Lemma \ref{lem:criexclnspt} and Proposition \ref{isolnspt} imply that no nonsingular point is a maximal center.

Suppose that $i \in \{26,28,33,48,61\}$.
Then, there is $0 \le m \le 4$ such that $x_m^e \in F'$ for some $e > 0$ and the inequality
\[
b_m := \max_{0 \le k < l \le 4, k, l \ne m} \operatorname{lcm} (a_k, a_l) \le 4/(A^3)
\]
holds.
The assertion $x_m^e \in F'$ implies $\Exc (\pi_m) = \emptyset$ and we have $\tilde{a}_{j;m} \le b_m$ for any $j \ne m$.
Thus Lemma \ref{lem:criexclnspt} and Proposition \ref{isolnspt} imply that non nonsingular point is a maximal center.

Let $X' = X'_7 \subset \mbP (1,1,2,3,1)$ be a member of $\mcG'_{16}$ and let $\msp = (\xi_0 \!:\! \xi_1 \!:\! \upsilon \!:\! \zeta \!:\! \omega) \in X'$ be a nonsingular point.
By the generality condition, $y^2 z \in F'$.
This implies that $(x_0 = x_1 = w = 0)_{X'}$ consists of singular points and hence $\msp$ is contained in one of the open subsets $(x_0 \ne 0)$, $(x_1 \ne 0)$ and $(w \ne 0)$.
By Proposition \ref{isolnspt}, $3 A$ isolates $\msp$ and thus $\msp$ is not a maximal center since $3 < 4/(A^3) = 24/7$.

Finally, let $X' = X'_5 \subset \mbP (1,1,1,2,1)$ be a member of $\mcG'_6$ and $\msp = (\xi_0 \!:\! \xi_1 \!:\! \xi_2 \!:\! \upsilon \!:\! \omega)$ a nonsingular point of $X$.
If $\msp \notin \Exc (\pi_3)$, then $A$ isolates $\msp$ by Proposition \ref{isolnspt}.
It follows that $\msp \not\in \Exc (\pi_3)$ is not a maximal center since $4/(A^3) = 8/5 > 1$.
If $\msp \in \Exc (\pi_3)$, then the proof \cite[Proof of (B) in Section 5.3]{CPR} works for our case and we have that $\msp$ is not a maximal center in this case.
\end{proof}

\section{Exclusion of quotient singular points} \label{sec:excltqpt}

Throughout this section, let $\msp$ be a terminal quotient singular point of $X' \in \mcG'_i$ with empty third column in the big table.
The aim of this section is to exclude $\msp$ as a maximal center.
We denote by $\varphi' \colon Y' \to X'$ the Kawamata blow up at $\msp$ with the exceptional divisor $E$ and put $B = -K_{Y'}$ as usual.
For a divisor or a curve $\Delta$ on $X'$, we denote by $\tilde{\Delta}$ the proper transform of $\Delta$ on $Y'$.

We first treat the points such that a set of polynomials and a divisor of the form $b B + e E$ are given in the second column of the table.
We note that $(B^3) \le 0$ in this case.
Let $\Lambda$ be the set of polynomials in the second column.
A coordinate with a prime means that it is the tangent coordinate of $X'$ at $\msp$.
For example, let $X' = X'_9 \subset \mbP (1,1,2,3,3)$ be a member of $\mcG'_{21}$ and $\msp$ the point of type $\frac{1}{2} (1,1,1)$.
We see that $y^3 \in f_6$ in the defining polynomial $w^2 x_0 y + w f_6 + g_9$ because otherwise $X'$ does not contain a singular point of type $\frac{1}{2} (1,1,1)$.
Then we have $w' = w + (\text{other terms})$.

\begin{Lem} \label{excltqsing}
Let $X'$ be a member of $\mcG'_i$ and $\msp \in X'$ a terminal quotient singular point marked a set of polynomials together with a divisor of the form $N := b B + e E$ in the third column of the big table.
Then, the divisor $N$ is nef and $(N \cdot B^2) \le 0$.
In particular, $\msp$ is not a maximal center.
\end{Lem} 

\begin{proof}
Let $\Lambda = \{h_1,\dots,h_l\}$ be the set of polynomials in the second table.
It is straightforward to see that $(h_1 = \cdots = h_l)_{X'}$ is a finite set of points including $\msp$ and we omit the proof (see Example \ref{ex:exclqtpt}).
Suppose that $h_j$ vanishes along $E$ to order $c_j/r$, where $r$ is the index of $\msp$, and set $b_j = \deg h_j$.
Then the proper transform of $(h_j = 0)_{X'}$ on $Y'$ defines the divisor $N_j \sim_{\mbQ} b _j B + e _j E$, where $e_j = (b_j - c_j)/r \ge 0$.
Let $k$ be an index such that
\[
e_k/b_k = \max_{1 \le j \le l} \{e_j/b_j\}.
\]
Then $N_k$ is the divisor $N = b B + e E$ given in the second column and we observe that the inequality $b > r e$ holds.
By \cite[Lemma 6.6]{Okada2}, $N = N_k$ is nef.
The verification of $(N \cdot B^2) \le 0$ is straightforward.
Therefore, by \cite[Corollary 2.16]{Okada2}, $\msp$ is not a maximal center.
\end{proof}

\begin{Ex} \label{ex:exclqtpt}
Let $X' = X'_{15} \subset \mbP (1,2,3,5,5)$ be a member of $\mcG'_{57}$ with defining polynomial $F' = w^2 z t + w f_{12} + g_{14}$ and $\msp$ a point of type $\frac{1}{2} (1,1,1)$.
Note that if $y^6 \notin f_{12}$, then there is no $\frac{1}{2} (1,1,1)$ point on $X'$.
We assume $y^6 \in f_6$.
After replacing $w$, we may assume that there is no monomial in divisible by $y^6$ in $g_{14}$.
In this setting, we have $w' = w$.
Since we have $F' (0,y,0,t,0) = \alpha t^2$ for some $\alpha \ne 0$, we see that $(x = z = w = 0)_{X'} = \{\msp\}$. 

Since we can choose $x,z,t$ as local orbifold coordinates of $X'$ at $\msp$, they vanish along $E$ to order $1/2$.
It is clear that $w$ vanishes along $E$ to order at least $2/2$ since $(w = 0)_{X'}$ is Cartier along $\msp$.
Thus, the proper transforms of $(x = 0)_{X'}$, $(z = 0)_{X'}$ and $(w = 0)_{X'}$ defines divisors $B$, $3 B + E$ and $2 B$, respectively.
It follows that $N = 3 B+E$ and the inequality $b > r e$ holds since $b = 3$, $e = 1$ and $r = 2$.
Finally, we have
\[
(B^2 \cdot 3 B + E) = 3 (A^3) - \frac{1}{2^3} (E^3) = \frac{1}{10} - \frac{1}{2} < 0.
\]
This completes all the computations for $X' \in \mcG'_{57}$ required in the proof of Lemma \ref{excltqsing}.
\end{Ex}

\begin{Lem} \label{excltqsingNo21_2}
Let $X' = X'_9 \subset \mbP (1,1,2,3,3)$ be a member of $\mcG'_{21}$. Then the singular point $\msp = \msp_2$ of type $\frac{1}{2} (1,1,1)$ is not a maximal center.
\end{Lem}

\begin{proof}
The defining polynomial of $X'$ is of the form $F' = w^2 x_0 y + w f_6 + g_9$.
Since $z^2 \in f_6$, we may assume that $z^3 \notin g_9$ and the coefficient of $z^2$ in $f_6$ is $1$ after replacing $w$ with $w + \eta z$ for a suitable $\eta \in \mbC$ and then re-scaling $z$.
We write $f_6 (0,0,y,z) = z^2 + \alpha y^3$ and $g_9 (0,0,y,z)= \beta z y^3$ for some $\alpha, \beta \in \mbC$. 
Note that neither $\alpha$ nor $\beta$ is zero by the generality condition.

We see that $x_0$ and $x_1$ vanish along $E$ to order $1/2$.
Let $S, T$ be general members of the pencil $|A|$.
Then $\tilde{S} \cap \tilde{T}$ is the proper transform of the curve 
\[
(x_0 = x_1 = 0)_{X'} = (x_0 = x_1 = w (z^2 + \alpha y^3) + \beta z y^3 = 0),
\] 
which is irreducible and reduced since $\alpha \ne 0$ and $\beta \ne 0$.
Thus, by \cite[Lemma 2.17]{Okada2}, $\tilde{S} \cap \tilde{T}$ is not a maximal center since $\tilde{S}, \tilde{T} \sim_{\mbQ} B$ and $(B \cdot \tilde{S} \cdot \tilde{T}) = (B^3) = 0$.
\end{proof}

\begin{Lem} \label{excltqsingNo44_2}
Let $X' = X'_{12} \subset \mbP (1,2,3,5,2)$ be a member of $\mcG'_{44}$.
Then a singular point $\msp \in X'$ of type $\frac{1}{2} (1,1,1)$ is not a maximal center.
\end{Lem}

\begin{proof}
The defining polynomial of $X'$ is of the form $F' = w^2 z t + w f_{10} + g_{12}$.
We may assume $y^5 \in f_{10}$ because otherwise $X'$ does not contain a point of type $\frac{1}{2} (1,1,1)$.
After replacing $w$, we may assume that there is no monomial divisible by $y^5$ in $g_{12}$.
We may moreover assume that the coefficient of $z^4$ in $g_{12}$ is $1$ after re-scaling $z$.
Let $\alpha,\beta$ and $\gamma$ be the coefficients of $y^3 z^2$, $y^2 z t$ and $y t^2$ in $F'$, respectively.
We see that $x,z,t$ vanish along $E$ to order $1/2$.
By looking at the defining equation, the section $w$ vanishes along $E$ to order $2/2$.
We define $S := (x = 0)_{X'}$ and $T := (w + \delta x^2 = 0)_{X'}$ for a general $\delta \in \mbC$. 
Then $\tilde{S} \sim_{\mbQ} B$ and $\tilde{T} \sim_{\mbQ} 2 B$.
We have
\[
S \cap T = (x = w = 0)_{X'} = (x = w = \alpha y^3 z^2 + \beta y^2 z t + \gamma y t^2 + z^4 = 0).
\]
If $\tilde{S} \cap \tilde{T}$ is irreducible (possibly non-reduced), then, by \cite[Lemma 2.17]{Okada2}, $\msp$ is not a maximal center since $(B \cdot \tilde{S} \cdot \tilde{T}) = 2 (B^3) < 0$.

We continue the proof assuming that $\tilde{S} \cap \tilde{T}$ is reducible, which is equivalent to the condition $\gamma = 0$ and $(\alpha,\beta) \ne (0,0)$.
Assume first $\gamma = 0$ and $\beta \ne 0$.
Then $S|_T = \Gamma + \Delta$, where $\Gamma = (x = w = z = 0)$ and $\Delta = (x = w = \alpha y^3 z + \beta y^2 t + z^3 = 0)$.
We have $(E \cdot \tilde{\Gamma}) = 1$ so that 
\[
(B \cdot \tilde{\Gamma}) = (A \cdot \Gamma) - \frac{1}{2} (E \cdot \tilde{\Gamma}) = \frac{1}{10} - \frac{1}{2} = - \frac{2}{5}.
\]
Since $\tilde{S} \sim_{\mbQ} B$, $\tilde{T} \sim_{\mbQ} 2 B$ and $\tilde{S}|_{\tilde{T}} = \tilde{\Gamma} + \tilde{\Delta}$, we compute
\[
- \frac{2}{5} = (\tilde{S}|_{\tilde{T}} \cdot \tilde{\Gamma})_{\tilde{T}} = (\tilde{\Gamma}^2)_{\tilde{T}} + (\tilde{\Gamma} \cdot \tilde{\Delta})_{\tilde{T}}
\]
and
\[
- \frac{3}{5} = 2 (B^3) = (\tilde{S}|_{\tilde{T}}^2)_{\tilde{T}} = (\tilde{\Gamma})^2_{\tilde{T}} + 2 (\tilde{\Gamma} \cdot \tilde{\Delta})_{\tilde{T}} + (\tilde{\Delta}^2)_{\tilde{T}}.
\]
Set $m := (\tilde{\Gamma} \cdot \tilde{\Delta})_{\tilde{T}} > 0$.
By the above displayed equations, the intersection matrix
\[
\begin{pmatrix}
(\tilde{\Gamma}^2) & (\tilde{\Gamma} \cdot \tilde{\Delta}) \\
(\tilde{\Gamma} \cdot \tilde{\Delta}) & (\tilde{\Delta}^2)
\end{pmatrix}
=
\begin{pmatrix}
- 2/5 - m & m \\
m & -1/5 - m
\end{pmatrix}
\]
is negative-definite.
Thus, by \cite[Lemma 2.18]{Okada2}, $\msp$ is not a maximal center.

Assume $\gamma = \beta = 0$.
Note that $\alpha \ne 0$ in this case.
Filtering off terms divisible by $y^3$, we write $F = y^3 G + H$.
We set $G_{\lambda} = G (x,y,z,t,\lambda x^2)$ and $H_{\lambda} = H (x,y,z,t,\lambda x^2)$.
Note that the coefficients of $z^2$ and $z^4$ in $G$ and $H$ are $\alpha \ne 0$ and $1$, respectively, and let $\varepsilon$ be the coefficient of $t z^2 x$ in $H$.
By eliminating $z^2$ in $H_{\lambda}$ in terms of $G_{\lambda} = 0$, we have
\[
(w - \lambda x^2 = G_{\lambda} = H_{\lambda} = 0)
= (w - \lambda x^2 = G_{\lambda} = x^2 P_{\lambda} = 0) \subset X',
\]
where $P_{\lambda} \in \mbC [x,y,z,t]$.
We set $C_{\lambda} := (w - \lambda x^2 = G_{\lambda} = P_{\lambda} = 0)$.
We see that $\tilde{C}_{\lambda}$ intersect $E \cong \mbP^2$ at $4$ points so that $(E \cdot \tilde{C}_{\lambda}) = 4$.
Hence, we have
\[
(-K_{Y'} \cdot \tilde{C}_{\lambda}) = (=K_X \cdot C_{\lambda}) - \frac{1}{2} (E \cdot \tilde{C}_{\lambda}) = 0.
\]
If $\tilde{C}_{\lambda}$ is reducible, then there is a component of $\tilde{C}_{\lambda}$ which intersects $-K_{Y'}$ non-positively and $E$ positively since there is at least one component that intersects $E$, and since a component of $C_{\lambda}$ that is disjoint from $E$ intersects $-K_{Y'}$ positively.
It follows that there are infinitely many irreducible curves on $Y'$ which intersects $-K_{Y'}$ non-positively and $E$ positively.
By \cite[Lemma 2.19]{Okada2}, $\msp$ is not a maximal center.
\end{proof}

\begin{Lem} \label{excltqsingNo6263}
Let $X' = X'_{15} \subset \mbP (1,3,4,5,3)$ be a member of $\mcG'_{62}$ or $\mcG'_{63}$.
Then a singular point $\msp \in X'$ of type $\frac{1}{3} (1,1,2)$ is not a maximal center.
\end{Lem}

\begin{proof}
If $X' \in \mcG'_{62}$ (resp.\ $\mcG'_{63}$), then the defining polynomial of $X'$ is of the form $F' = w^3 y^2 + w^2 y f_6 + w f_{12} + g_{15}$ (resp.\ $F' = w^2 z t + w f_{12} + g_{15}$).
If $X' \in \mcG'_{63}$, then we may assume $y^4 \in f_{12}$ because otherwise $X'$ does not contain a point of type $\frac{1}{3} (1,1,2)$.
After replacing $w$, we may assume that there is no monomial divisible by $y^4$ in $g_{15}$.
Let $\alpha, \beta$ and $\gamma$ be the coefficients of $y^2 z t$, $y z^3$ and $t^3$ in $F'$, respectively.
By quasismoothness of $X'$, we have $\gamma \ne 0$.
We see that $x$, $z$, $t$ vanish along $E$ to order $1/3$, $1/3$ and $2/3$, respectively.
By looking at the defining equation $F' = 0$, we see that $w$ vanishes along $E$ to order $3/3$.
We define $S = (x = 0)_{X'}$ and $T = (w + \delta x^3 = 0)_{X'}$ for a general $\delta \in \mbC$.
Then $\tilde{S} \sim_{\mbQ} B$ and $\tilde{T} \sim_{\mbQ} 3 B$.
We have 
\[
S \cap T = (x = w = 0)_{X'} = (x = w = \alpha y^2 z t + \beta y z^3 + \gamma t^3 = 0).
\]
If $\tilde{S} \cap \tilde{T}$ is irreducible (possibly non-reduced), then, by \cite[Lemma 2.17]{Okada2}, $\msp$ is not a maximal center since $(B \cdot \tilde{S} \cdot \tilde{T}) = 3 (B^3) < 0$.

We assume that $S \cap T$ is reducible, which is equivalent to the condition $\beta = 0$ and $\alpha \ne 0$.
We define $\Gamma := (x = w = z = 0)$ and $\Delta := (x = w = \alpha y^2 z + \gamma t^2 = 0)$ so that $S |_T = \Gamma + \Delta$.
Note that we have $\tilde{S} |_{\tilde{T}} = \tilde{\Gamma} + \tilde{\Delta}$.
We see that $\tilde{\Gamma}$ intersects $E$ transversally at a single nonsingular point so that $(E \cdot \tilde{\Gamma}) = 1$.
Hence,
\[
(B \cdot \tilde{\Gamma}) = (A \cdot \Gamma) - \frac{1}{3} (E \cdot \tilde{\Gamma}) = - \frac{1}{4}.
\]
Now we compute
\[
- \frac{1}{4} = (\tilde{S}|_{\tilde{T}} \cdot \tilde{\Gamma})_{\tilde{T}} = (\tilde{\Gamma}^2)_{\tilde{T}} + (\tilde{\Gamma} \cdot \tilde{\Delta})_{\tilde{T}},
\]
and
\[
- \frac{1}{4} = 3 (B^3) = (B|_T)^2_{\tilde{T}} = (\tilde{\Gamma} + \tilde{\Delta})^2_{\tilde{T}} = (\tilde{\Gamma}^2)_{\tilde{T}} + 2 (\tilde{\Gamma} \cdot \tilde{\Delta})_{\tilde{T}} + (\tilde{\Delta}^2)_{\tilde{T}}.
\]
Set $m := (\tilde{\Gamma} \cdot \tilde{\Delta}) > 0$.
By the above displayed equations, we see that the intersection matrix 
\[
\begin{pmatrix}
(\tilde{\Gamma}^2) & (\tilde{\Gamma} \cdot \tilde{\Delta}) \\
(\tilde{\Gamma} \cdot \tilde{\Delta}) & (\tilde{\Delta}^2)
\end{pmatrix}
=
\begin{pmatrix}
- \frac{1}{4} - m & m \\
m & - m
\end{pmatrix}
\]
is negative-definite.
Therefore, by Lemma \cite[Lemma 2.18]{Okada2}, $\msp$ is not a maximal center.
\end{proof}

\begin{Lem} \label{excltqsingNo33_3}
Let $X' = X'_{10} \subset \mbP (1,1,3,5,1)$ be a member of $\mcG'_{33}$.
Then, a singular point $\msp$ of type $\frac{1}{3} (1,1,2)$ is not a maximal center.
\end{Lem}

\begin{proof}
The defining polynomial of $X'$ is of the form $w^2 y z + w f_9 + g _{10}$.
Since $y^3 \in f_9$, we may assume that there is no monomial divisible by $y^3$ in $g_{10}$ after replacing $w$.
We see that $x_0,x_1,z,w$ vanish along $E$ to oder respectively $1/3,1/3,2/3,4/3$.
We set $S := (w = 0)_{X'}$ and $T := (x_0 = 0)_{X'}$. 
Then $\tilde{S} \sim_{\mbQ} B - E$ and $\tilde{T} \sim_{\mbQ} B$.
We have
\[
S \cap T = (x_0 = w = 0)_{X'} = (x_0 = w = g_{10} (0,x_1,y,z) = 0).
\]
Note that $z^2 \in g_{10}$.
It follows that if $S \cap T$ is reducible, then it is the union of two curves $\Gamma_1$ and $\Gamma_2$, where $\Gamma_i := (x_0 = w = \alpha_i z + \beta_i y x_1^2 + \gamma_i x_1^5 = 0)$ for some $\alpha_i \ne 0, \beta_i, \gamma_i \in \mbC$ for $i = 1,2$.
We have $\tilde{S} \cap \tilde{T} = \tilde{\Gamma}_1 + \tilde{\Gamma}_2$.
We see that $\tilde{\Gamma}_1$ is numerically equivalent to $\tilde{\Gamma}_2$ since $(\varphi^* A \cdot \tilde{\Gamma}_1) = (\varphi^* A \cdot \tilde{\Gamma}_2) = 1/3$ and $(E \cdot \tilde{\Gamma}_1) = (E \cdot \tilde{\Gamma}_2) = 1$.
Thus the support of $\tilde{\Gamma} := \tilde{S} \cap \tilde{T}$ is either irreducible or it is the union of two curves which are numerically equivalent to each other.
We have
\[
(\tilde{T} \cdot \tilde{\Gamma}) = (\tilde{T}^2 \cdot \tilde{S}) = (A^3) - \frac{4}{3^3} (E^3) = \frac{2}{3} - \frac{2}{3} = 0.
\]
By \cite[Lemma 2.17]{Okada2}, $\msp$ is not a maximal center.
\end{proof}

The following is the conclusion of this section.

\begin{Thm} \label{thm:excltqs}
Let $X'$ be a member of $\mcG'_i$ with $i \in I^*_{cA/n} \cup I_{cD/3}$.
Then, no terminal quotient singular point with empty third column in the big table is a maximal center.
\end{Thm}

\begin{proof}
This follows from Lemmas \ref{excltqsing}, \ref{excltqsingNo21_2},  \ref{excltqsingNo44_2}, \ref{excltqsingNo6263} and \ref{excltqsingNo33_3}.
\end{proof}

\section{Exclusion of divisorial extractions centered at $cA/n$ points} \label{sec:exclcA}

Let $X'$ be a member of $\mcG'_i$ with $i \in I^*_{cA/n}$ and $\msp = \msp_4 \in X'$ the $cA/n$ point.
The aim of this section is to show that no divisorial extraction $\varphi' \colon Y' \to X'$ centered at $\msp$ with $(-K_{Y'}^3) \le 0$ is a maximal extraction.

Let $\varphi' \colon Y' \to X'$ be a divisorial extraction centered at $\msp$ such that $(B^3) \le 0$, where $B := -K_{Y'}$.
Such an extraction is the one marked ``none" in the big table.
Suppose $i \in \{16,22,26,33,48\}$.
We define $S := (x_0 = 0)_{X'}$ and $T := (x_1 = 0)_{X'}$ if $i \in \{16,22,26,33,48\}$, and define $S := (x = 0)_{X'}$ and $T := (y=0)_{X'}$ if $i \in \{44,57,63\}$.
We set $m = 1$ if $i \in \{16,22,26,33,48\}$, and $m = \deg y$ if $i \in \{44,57,63\}$.
We have $S \sim_{\mbQ} A$ and $T \sim_{\mbQ} m A$.
For a divisor or a curve $\Delta$ on $X'$, we denote by $\tilde{\Delta}$ the proper transform of $\Delta$ on $Y'$.

\begin{Lem} \label{lem:exclcAirred}
The scheme-theoretic intersection $\Gamma := S \cap T$ is an irreducible and reduced curve.
Moreover, we have $\tilde{S} \sim_{\mbQ} B$, $\tilde{T} \sim_{\mbQ} m B$ and $\tilde{S} \cap \tilde{T} = \tilde{\Gamma}$.
\end{Lem}

\begin{proof}
If $i \in \{16,22,26,33,48\}$ (resp.\ $i \in \{44,57,63\}$), then $(x_0 = x_1 = 0)$ (resp.\ $(x = y = 0)$) is a weighted projective plane, which we denote by $\mbP$, and $\Gamma$ is isomorphic to the hypersurface in $\mbP$ defined by the equation given in Table \ref{table:defeqcA}.
The conditions on $\alpha, \beta, \gamma$ in the table are satisfied by the generality condition imposed on the family $\mcG'_i$.
Thus, by a straightforward argument, $\Gamma$ is irreducible and reduced. 

We see that $x_0$ and $x_1$ (resp.\ $x$ and $y$) vanish along $E$ to order $1/n$ and $1/n$ (resp.\ $1/n$ and $m/n$).
This implies $\tilde{S} \sim_{\mbQ} B$ and $\tilde{T} \sim_{\mbQ} m B$.
Moreover, $\tilde{S} \cap \tilde{T} \cap E$ consists of two points and, in particular, does not contain a curve.
This shows $\tilde{S} \cap \tilde{T} = \tilde{\Gamma}$.
\end{proof}

\begin{Thm} \label{thm:exclcAneg}
Let $X'$ be a member of $\mcG'_i$ with $i \in I^*_{cA} \cup I_{cA/n}$ and  $\msp = \msp_4 \in X'$ the $cA$ or $cA/n$ point.
Then no extremal extraction $\varphi' \colon Y' \to X'$ with $(-K_{Y'})^3 \le 0$ is a maximal center.
\end{Thm}

\begin{proof}
By Lemma \ref{lem:exclcAirred}, we have $\tilde{S} \sim_{\mbQ} B$, $\tilde{T} \sim_{\mbQ} m B$ and $\tilde{S} \cap \tilde{T} = \tilde{\Gamma}$ is an irreducible and reduced curve.
It follows that 
\[
(\tilde{T} \cdot \tilde{\Gamma}) = (\tilde{T} \cdot \tilde{S} \cdot \tilde{T}) = m^2 (B^3) \le 0.
\]
Therefore, by \cite[Lemma 2.17]{Okada2}, $\varphi'$ is not a maximal extraction.
\end{proof}

\begin{table}[bt] \label{table:defeqcA}
\begin{center}
\caption{Defining equation of $\Gamma$}
\label{table:gammaqsm}
\begin{tabular}{ccc}
\hline
No. & Equations & Conditions \\
\hline
16 & $w^2 y z + w (\alpha y^3 + \beta z^2) + \gamma y^2 z = 0$ & $\alpha, \beta, \gamma \ne 0$ \\
22 & $w^2 y z + \alpha y^4 + \beta y^2 z + z^2 $ & $\alpha \ne 0$ \\
26 & $w^2 y z + \alpha w z^2 + z^3 = 0$ & $\alpha \ne 0$ \\
33 & $w^2 y z + \alpha w y^3 +z^2 = 0$ & $\alpha \ne 0$ \\
44 & $w^2 z t + \alpha w t^2 + y^3 = 0$ & $\alpha \ne 0$ \\
48 & $w^2 y z + y^3 + z^2 = 0$ & \\
57 & $w^2 z t + \alpha w z^4 + t^2 = 0$ & $\alpha \ne 0$ \\
63 & $w^2 z t + \alpha w z^3 + t^3 = 0$ & $\alpha \ne 0$ 
\end{tabular}
\end{center} 
\end{table}

\section{Exclusion of curves} \label{sec:exclcurve}

\subsection{Exclusion of most of the curves}

We can exclude most of the curves as follows.

\begin{Lem} \label{exclmostcurves}
Let $X'$ be a member of the family $\mcG'_i$ and $\Gamma \subset X'$ an irreducible and reduced curve.
Then, $\Gamma$ is not a maximal center except possibly for the following cases.
\begin{itemize}
\item No. $6$  and $\deg \Gamma = 1, 2$.
\item No. $7$ and $\deg \Gamma = \frac{1}{2}, 1$.
\item No. $9$, $10$, $16$ and $\deg \Gamma = 1$.
\item No. $18$ and $\deg \Gamma = \frac{1}{2}$.
\item No. $21$ and $\deg \Gamma = \frac{1}{3}$.
\end{itemize}
Here, in any of the above exceptions, $\Gamma$ passes through the $cA/n$ point $\msp_4$ and does not pass through any terminal quotient singular point.
\end{Lem}

\begin{proof}
We may assume that $\Gamma$ does not pass through a terminal quotient singular point since there is no divisorial extraction centered along a curve through such a point (see \cite{Kawamata}).
We see $n \deg \Gamma = (n A \cdot \Gamma) \in \mbZ_{>0}$, where $n$ is the index of the singularity $(X',\msp_4)$.
Thus, $\deg \Gamma \in \frac{1}{n} \mbZ_{>0}$.
By \cite[Proof of Theorem 5.1.1]{CPR}, $\Gamma$ is not a maximal center if $\deg \Gamma \ge (A^3)$.
By checking each family individually, a curve $\Gamma$ on $X'$ such that $\deg \Gamma < (A^3)$ is one of the curves listed in the statement.
Finally, let $\Gamma$ be one of the curves in the statement.
If $\Gamma$ does not pass through $\msp_4$, then it is contained in the nonsingular locus of $X'$.
Then, the argument in Step 2 of \cite[Proof of Theorem 5.1.1]{CPR} to our case, and as a result, we see that $\Gamma$ is not a maximal center.
This completes the proof.
\end{proof}

In the rest of this section, we exclude the remaining curves by aplying the following results.

\begin{Lem} \label{lem:criexclC1}
Let $X$ be a $\mbQ$-Fano $3$-fold with Picard number $1$ and $\Gamma \subset X$ an irreducible and reduced curve.
Suppose that there is a movable linear system $\mcM$ on $X$ with the following properties.
\begin{enumerate}
\item $\Gamma$ is the unique base curve of $\mcM$.
\item A general member $S \in \mcM$ is a normal surface.
\item For a general $S \in \mcM$, $(\Gamma^2)_S \le 0$ and $((-K_X)^2 \cdot S) - 2 (-K_X \cdot \Gamma) + (\Gamma^2)_S \le 0$.
\end{enumerate}
Then, $\Gamma$ is not a maximal center.
\end{Lem}

\begin{proof}
Let $\mcH \sim_{\mbQ} - n K_X$ be a movable linear system on $X$.
It is enough to show that $\mult_{\Gamma} \mcH \le 1$.
Let $S \in \mcM$ be a general member so that it does not contain base curves of $\mcH$ other than $\Gamma$.
Then, we can write
\[
(-K_X)|_S \sim_{\mbQ} \frac{1}{n} \mcH |_S = \frac{1}{n} \mcL + \gamma \Gamma,
\]
where $\mcL$ is a movable linear system on $S$ and $\gamma \ge \mult_{\Gamma} \mcH$.
Since $\mcL$ is nef, we have
\[
0 \le \frac{1}{n^2} (\mcL^2)_S = ((-K_X)|_S - \gamma \Gamma)_S^2
= ((-K_X)^2 \cdot S) - 2 (-K_X \cdot \Gamma) \gamma + (\Gamma^2)_S \gamma^2.
\]
The right-hand side of the above equation is a strictly decreasing function of $\gamma$ for $\gamma \ge 0$ since $(-K_X \cdot \Gamma) > 0$ and $(\Gamma^2)_S \le 0$.
Thus, the inequality $((-K_X)^2 \cdot S) - 2 (-K_X \cdot \Gamma) + (\Gamma^2)_S \le 0$ implies $\gamma \le 1$.
Therefore, $\mult_{\Gamma} \mcH \le \gamma \le 1$ and $\Gamma$ is not a maximal center.
\end{proof}

\begin{Lem}[{\cite[Lemma 2.10]{Okada2}}] \label{lem:criexclC2}
Let $X$ be a $\mbQ$-Fano $3$-fold with Picard number $1$ and $\Gamma \subset X$ an irreducible and reduced curve.
Suppose that there is an effective divisor $T$ on $X$ containing $\Gamma$ and a movable linear system $\mcM$ on $X$ whose base locus contains $\Gamma$ with the following properties.
\begin{enumerate}
\item $T \sim_{\mbQ} m A$ for some rational number $m \ge 1$.
\item For a general member $S \in \mcM$, $S$ is a normal surface, the intersection $S \cap T$ is contained in the base locus of $\mcM$ set-theoretically and $S \cap T$ is reduced along $\Gamma$.
\item Let $S \in \mcM$ be a general member and let $\Gamma_1,\dots,\Gamma_l$ be the irreducible and reduced curves contained in the base locus of $\mcM$.
For each $i = 1,\dots,l$, there is an effective $1$-cycle $\Delta_i$ on $S$ such that $(\Gamma \cdot \Delta_i)_S \ge (A \cdot \Delta_i)_S > 0$ and $(\Gamma_j \cdot \Delta_i)_S \ge 0$ for $j \ne i$.
\end{enumerate}
\end{Lem}

\subsection{Computation of intersection numbers of on a surface}

As we explained in the previous subsection, the exclusion of curves is reduced to the computation of intersection numbers on a surface.
In this paper, the pullback of a Weil divisor and the intersection number of two Weil divisors on a normal surface are those in the sense of Mumford (see \cite{Mumford}).
We explain ideas of the proof.

Let $\Gamma \subset X'$ be an irreducible and reduced curve.
We try to find a divisor $T$ on $X'$ and a linear system $\mcM$ such that the assumptions of Lemma \ref{lem:criexclC2} hold.
The crucial part here is to conclude $(\Gamma \cdot \Delta_i) \ge \deg \Delta_i$.
Here, we use notation of Lemma \ref{lem:criexclC2}.
In most of the cases, we set $\Delta_i = \Gamma_i$ and prove the inequality $(\Gamma \cdot \Gamma_i) \ge \deg \Gamma_i$.
In many cases, the computation of $(\Gamma \cdot \Gamma_i)$ will be done by counting the number of intersection points $\Gamma \cap \Gamma_i$ along the nonsingular locus of a surface $S \in \mcM$.
To this end, we need to know nonsingular locus of $S$.

\begin{Lem} \label{lem:qsmcri}
Let $V$ be a weighted hypersurface in $\mbP (a_0,\dots,a_4)$ with defining polynomial $F$ and $\Gamma = (x_0 + f_0 = x_1 + f_1 = x_2 + f_2 = 0)$ for some $f_i \in \mbC [x_0,\dots,x_4]$.
Suppose that $a_0 \le a_1$.
Let $\Lambda \subset |\mcO_V (a_1)|$ be the linear system generated by $x_1 + f_1$ and $\{(x_0+f_0) \prod x_i^{m_i} \mid m_i \ge 0, m_0 + \cdots + m_4 = a_1-a_0\}$, and let $S \in \Lambda$ be a general member.
If $V$ does not contain $(x_0 + f_0 = x_1 + f_1 = 0)$ and $(x_0 + f_0 = x_1 + f_1 = 0)_V$ is reduced along $\Gamma$, then $S$ is quasismooth at any point of
\[
\Gamma \setminus \NQsm (V) \cup \left( \Bs |\mcO_{\mbP} (a_1-a_0)| \cap \left(\frac{\prt F}{\prt x_0} = 0 \right) \right), 
\]
where $\NQsm (V)$ denotes the non-quasismooth locus of $V$.
\end{Lem}

\begin{proof}
After replacing $x_0,x_1,x_2$, we may assume $f_0 = f_1 = f_2 = 0$.
Since $\Gamma \subset V$, we have $F = x_0 g_0 + x_1 g_1 + x_2 g_2$ for some $g_i \in \mbC [x_0,\dots,x_4]$.
Note that $S$ is cut out on $V$ by the section $x_0 q + \lambda x_1$, where $q$ is a general homogeneous polynomial of degree $a_1-a_0$ and $\lambda \in \mbC$ is general.
We have
\[
J_S |_{\Gamma} =
\begin{pmatrix}
g_0 & g_1 & g_2 & 0 & 0 \\
q & \lambda & 0 & 0 & 0
\end{pmatrix}.
\]
We see that $g_2$ is a non-zero polynomial and it does not vanish along $\Gamma$ since $(x_0 = x_1 = 0) \not\subset V$ and $(x_0 = x_1 = 0)_V$ is reduced along $\Gamma$.
It follows that $\NQsm (S)$ is contained in the finite set $\Gamma \cap (g_2 = 0)$.
Let $\msp \in \NQsm (S) \setminus \NQsm (V)$.
Note that either $g_0 (\msp) \ne 0$ or $g_1 (\msp) \ne 0$ since $V$ is quasismooth at $\msp$ and $g_2 (\msp) = 0$.
If $\msp \notin \Bs |\mcO_{\mbP} (a_1-a_0)|$, then $q (\msp) \ne 0$ for a general $q$ and hence $\rank J_S (\msp) = 2$ for a general $\lambda$.
If $(\prt F\prt x_0) (\msp) = g_0 (\msp) \ne 0$, then $\rank J_S (\msp) = 2$ for a general $\lambda$ since $g_1 (\msp) \ne 0$.
This completes the proof.
\end{proof}

We encounter with the case where $\Gamma \cap \Gamma_i = \{\msp_4\}$, which makes the computation of $(\Gamma \cdot \Gamma_i)$ difficult.
When we apply Lemma \ref{lem:criexclC1}, we need to obtain the inequality in (3).
In other words, we need to bound $(\Gamma^2)_S$ from above.
The computation of $(\Gamma^2)_S$ is not so straightforward since $\Gamma$ passes through the $cA/n$ point $\msp_4$,
In these cases, the computation will be done by the following method.

\begin{Def}
Let $(S,\msp)$ be a germ of a normal surface and $\Gamma$ an irreducible and reduced curve on $S$.
Let $\hat{S} \to S$ be the minimal resolution of $(S,\msp)$ and denote by $E_1,\dots,E_m$ the prime exceptional divisors.
We define $G (S,\msp,\Gamma)$ to be the dual graph of $E_1,\dots,E_m$ and the proper transform $\hat{\Gamma}$ of $\Gamma$ on $\hat{S}$: vertices of $G(S,\msp,\Gamma)$ corresponds to $E_1,\dots,E_m$ and $\hat{\Gamma}$, and two vertices corresponding to $E_i$ and $E_j$ (resp.\ $E_i$ and $\hat{\Gamma}$) are joined by $(E_i \cdot E_j)$-ple edge (resp.\ $(E_i \cdot \hat{\Gamma})$-ple edge).
We call $G (S,\msp,\Gamma)$ the {\it extended dual graph} of $(S,\msp,\Gamma)$.
\end{Def}

In this paper, we only treat germs $(S,\msp)$ of singularity of type $A_n$.

\begin{Def}
We say that $G (S,\msp,\Gamma)$ is {\it of type} $A_{n,k}$ if it is of the form
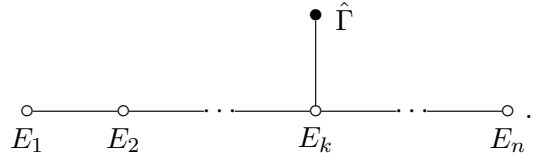
\[
\objectmargin={-0.5pt}
\xygraph{
\circ ([]!{+(0,-.3)} {E_1}) - [r]
\circ ([]!{+(0,-.3)} {E_2}) - [r] \cdots - [r]
\circ ([]!{+(0,-.3)} {E_k})(
-[u] \bullet ([]!{+(.3,0)} {\hat{\Gamma}}),
 - [r] \cdots - [r]
\circ ([]!{+(0,-.3)}{E_n}))}.
\]
Here, $\circ$ means that the corresponding exceptional divisor is a $(-2)$-curve.
In other words, $G (S,\msp,\Gamma)$ is of type $A_{n,k}$ if $(S,\msp)$ is of type $A_n$, $(\hat{\Gamma} \cdot E_i) = 0$ for $i \ne k$ and $(\hat{\Gamma} \cdot E_k) = 1$.
\end{Def}

\begin{Lem} \label{lem:compselfint}
Let $S$ be a normal projective surface and $\Gamma$ a nonsingular rational curve on $S$.
Let $\msp$ be a singular point of $S$ and suppose that $S$ is nonsingular along $\Gamma \setminus \{\msp\}$.
If $G (S,\msp,\Gamma)$ is of type $A_{n,k}$, then 
\[
(\Gamma^2)_S = -2 - (K_S \cdot \Gamma)_S + \frac{k (n-k+1)}{n+1}.
\]
\end{Lem}

\begin{proof}
Let $\psi \colon \hat{S} \to S$ be the minimal resolution of $S$ with prime exceptional divisors $E_1,\dots,E_n$.
We denote by $\hat{\Gamma}$ the proper transform of $\Gamma$ on $\hat{S}$.
We have $\psi^*\Gamma = \hat{\Gamma} + a_1 E_1 + \cdots + a_n E_n$ for some rational numbers $a_1,\dots,a_n$.

Suppose that $G (S,\msp,\Gamma)$ is of type $A_{n,k}$.
We have
\[
\begin{split}
0 = (\varphi^*\Gamma \cdot E_i) &= (\hat{\Gamma} \cdot E_i) + a_1 (E_1 \cdot E_i) + \cdots + a_n (E_n \cdot E_i) \\
&= \begin{cases}
a_{i-1} - 2 a_i + a_{i+1}, & \text{if $i \ne k$}, \\
1 + a_{k-1} - 2 a_k + a_{k+1}, & \text{if $i = k$}
\end{cases}
\end{split}
\]
Here, we define $a_0 = a_{n+1} = 0$.
By $a_{i-1} - 2 a_i + a_{i+1} = 0$ for $i \ne k$, we have $a_i = i a_1$ for $i = 1,2,\dots,k$ and $a_{n-i+1} = i a_n$ for $i = 1,2,\dots,n-k+1$.
In particular, $a_k = k a_1 = (n-k+1) a_n$, and thus $a_n = \frac{k}{n-k+1} a_1$.
Now, by the equations
\[
\begin{split}
0 &= (\varphi^*\Gamma \cdot E_i) 
= 1 + a_{k-1} - 2 a_k + a_{k+1} \\
&= 1 + (k-1) a_1 - 2 k a_1 + \frac{k(n-k)}{n-k+1} a_1 
= 1 - \frac{n+1}{n-k+1} a_1,
\end{split}
\]
we have
\[
a_1 = \frac{n-k+1}{n+1}, \quad a_k = k a_1 = \frac{k (n-k+1)}{n+1}.
\]
The assertion follows from the combination of equations
\[
(\Gamma^2) = (\hat{\Gamma} \cdot \psi^*\Gamma) = (\hat{\Gamma}^2) + a_k
\]
and
\[
(\hat{\Gamma}^2) = -2 + (K_{\hat{S}} \cdot \hat{\Gamma}) = -2 + (K_S \cdot \Gamma).
\]
This completes the proof.
\end{proof}

In what follows, we sometimes consider a suitable weighted blowup $\varphi \colon Y' \to X'$, which is a divisorial extraction.
In this case, we always denote by $E$ its exceptional divisor and by $\tilde{\Delta}$ the proper transform on $Y'$ of a curve or a divisor $\Delta \subset X'$.

Finally, we define suitable coordinate change.

\begin{Def}
Let $X' \subset \mbP (1,n,a_2,a_3,n)$ be a member of $\mcG'_i$ with $i \in I^*_{cA/n}$ defined by a standard defining polynomial $F' = w^2 x_2 x_3 + w f + g$.
{\it An admissible coordinate change} of $X'$ is a coordinates change that defines an automorphism $\theta$ of $\mbP (1,n,a_2,a_3,n)$ such that $\theta^* F' = \alpha (w^2 x_2 x_3 + w f' + g')$ for some non-zero $\alpha \in \mbC$ and $f', g' \in \mbC [x_0,\dots,x_3]$.
\end{Def}

A change of coordinates $\theta$ is admissible if and only if $\theta^*w = \beta w + (\text{other terms})$ for some non-zero $\beta \in \mbC$, $\theta^*x_j$ does not involve $w$ for $j = 0,1,2,3$ and $\theta^* (x_2 x_3) = \gamma x_2 x_3$ for some non-zero $\gamma \in \mbC$.

\subsection{Curves of degree $1$ on $X' \in \mcG'_6$} \label{sec:CG6-1}

Let $X' = X'_5 \subset \mbP (1,1,1,2,1)$ be a member of $\mcG'_6$ with defining polynomial $F' = w^2 x_0 y + w f_4 + g_5$ and $\Gamma$ an irreducible and reduced curve of degree $1$ on $X'$ that passes through $\msp_4$ but does not pass through the other singular points.

We claim that $\Gamma$ is a WCI curve of type $(1,1,2)$.
Indeed, the projection $\pi \colon X' \ratmap \mbP^3$ to the coordinates $x_0,x_1,x_2,w$ induces the finite morphism $\pi|_{\Gamma} \colon \Gamma \to \pi (\Gamma)$.
We have $1 = \deg \Gamma = \deg (\pi|_{\Gamma}) \deg \pi (\Gamma)$.
Hence $\pi|_{\Gamma}$ is an isomorphism and $\pi (\Gamma)$ is a line in $\mbP^3$.
It follows that there are linear forms $\ell_1,\ell_2 \in \mbC [x_0,x_1,x_2,w]$ such that $\Gamma \subset (\ell_1 = \ell_2 = 0)_{X'}$.
From this we see that $\Gamma = (\ell_1 = \ell_2 = q = 0)$ for some $q \in \mbC [x_0,x_1,x_2,y,z,w]$ of degree $2$. 

Let $S$ and $T$ be general members of the pencil $|\mcI_{\Gamma} (A)|$.

\begin{Lem} \label{admCdeg1G6}
After an admissible coordinate change, we have $T|_S = \Gamma + \Delta$, where the pair $(\Gamma,\Delta)$ of curves on $S$ is  one of the following.
\begin{enumerate}
\item $\Gamma = (x_1 = x_2 = y + w x_0 + \beta x_0^2 = 0)$ and $\Delta = (x_1 = x_2 = w y + \nu x_0^3 = 0)$ for some $\beta,\nu \in \mbC$.
\item $\Gamma = (x_1 = x_2 = y = 0)$ and $\Delta = (x_1 = x_2 = w^2 x_0 + w y + \lambda w x_0^2 + \nu x_0^3 = 0)$ for some $\lambda,\mu \in \mbC$.
\item $\Gamma = (x_0 = x_1 = y + \beta x_2^2 = 0)$ and $\Delta = (x_0 = x_1 = w y + \lambda w x_2^2 + \nu x_2^3 = 0)$ for some $\beta,\lambda,\nu \in \mbC$ with $\nu \ne 0$ and $(\beta,\lambda) \ne (0,0)$.
\end{enumerate}
\end{Lem}

\begin{proof}
Since $y^2 \in f_4$, we may assume that the coefficient of $y^2$ in $f_4$ is $1$ and that there is no monomial divisible by $y^2$ in $g_5$ after replacing $w$.

Suppose $\Gamma \not\subset (x_0 = 0)$.
Then, after replacing $x_1,x_2$, we may assume $\Gamma = (x_1 = x_2 = y + \alpha w x_0 + \beta x_0^2 = 0)$ for some $\alpha,\beta \in \mbC$.
We have
\[
\bar{F}' := F' (x_0,0,0,y,w) = w^2 x_0 y + w (y^2 + \gamma y x_0^2 + \delta x_0^4) + \varepsilon y x_0^3 + \zeta x_0^5,
\]
where $\gamma,\dots,\zeta \in \mbC$.
Since $\Gamma \subset X'$, we have
\[
\bar{F}' = (y + \alpha w x_0 + \beta x_0^2)(\eta w^2 x_0 + \theta w y + \lambda w x_0^2 + \mu y x_0 + \nu x_0^3)
\]
for some $\eta,\dots,\nu \in \mbC$.
By comparing the coefficients of $w^3 x_0^2$, $w^2 x_0 y$, $w^2 x_0^3$, $w y^2$ and $y^2 x_0$, we have
\[
\alpha \eta = 0, \eta + \alpha \theta = 1, \alpha \lambda + \beta \eta = 0, \theta = 1, \mu = 0.
\]
If $\alpha \ne 0$, then $\eta = 0$, $\alpha = \theta = 1$, $\lambda = 0$, $\mu = 0$, and this case corresponds to (1).
If $\alpha = 0$, then $\eta =1$, $\beta = 0$, $\theta = 1$, $\mu = 0$, and this case corresponds to (2).

Suppose $\Gamma \subset (x_0 = 0)$.
Then, after replacing $x_1,x_2$, we may assume $\Gamma = (x_0 = x_1 = y + \alpha w x_2 + \beta x_2^2 = 0)$.
We have
\[
\bar{F}' := F' (0,0,x_2,y,w) = w (y^2 + \gamma y x_2^2 + \delta x_2^4) + \varepsilon y x_2^3 + \zeta x_2^5,
\]
where $\gamma,\dots,\zeta \in \mbC$.
Since $\Gamma \subset X'$, we have
\[
\bar{F}' = (y + \alpha w x_2 + \beta x_2^2)(\theta w y + \lambda w x_2^2 + \mu y x_2 + \nu x_2^3)
\]
for some $\theta,\dots,\nu \in \mbC$.
By comparing the coefficients of $w^2 x_2 y$, $w^2 x_2^3$, $w y^2$ and $y^2 x_2$, we have
\[
\alpha \theta = 0, \alpha \lambda = 0, \theta = 1, \mu = 0,
\]
and hence $\alpha = \mu = 0$, $\theta = 1$.
We claim that $\nu \ne 0$ and $(\beta,\lambda) \ne (0,0)$.
Suppose $\nu = 0$.
Then, $\varepsilon = \nu = 0$ and $\zeta = \beta \nu = 0$.
Since $\varepsilon$ and $\zeta$ are the coefficients of $x_2^3$ and $x_2^5$ in $b_3$ and $b_5$, respectively, this implies that $x_0 = b_3 = b_5 = 0$ has a solution $(x_0,x_1,x_2) = (0,0,1)$.
This is impossible by Condition \ref{addcond}.
Thus, this case corresponds to (3).
Suppose $(\beta,\lambda) = (0,0)$.
Then $\gamma = \beta + \lambda = 0$ and $\delta = \beta \lambda = 0$, which implies that $x_0 = a_2 = a_4 = 0$ has a solution $(x_0,x_1,x_2) = (0,0,1)$.
This is again impossible by Condition \ref{addcond}.
\end{proof}

Note that $S$ is nonsingular along $\Gamma \setminus \{\msp_4\}$ by Lemma \ref{lem:qsmcri}.

\begin{Lem}
Let $T|_S = \Gamma + \Delta$ be as in \emph{Lemma \ref{admCdeg1G6}}.
Then the following assertions hold.
\begin{enumerate}
\item If $\Delta$ is irreducible, then $(\Gamma \cdot \Delta)_S \ge \deg \Delta$.
\item If $\Delta$ is reducible, then it splits as $\Delta = \Delta_1 + \Delta_2$, where $\Delta_1, \Delta_2$ are irreducible and reduced curves of degree respectively $1$ and $1/2$, respectively, such that $(\Gamma \cdot \Delta_i)_S \ge \deg \Delta_i$ for $i = 1,2$.
\end{enumerate}
In particular, $\Gamma$ is not a maximal center.
\end{Lem}

\begin{proof}
Suppose that we are in case (1) of Lemma \ref{admCdeg1G6}.
We see that $\Delta$ is clearly reduced and it is irreducible if and only if $\nu \ne 0$.
If $\Delta$ is irreducible, then it intersects $\Gamma$ at two nonsingular points so that $(\Gamma \cdot \Delta)_S \ge 2 > \deg \Delta$.
If $\Delta$ is reducible, then $\Delta = \Delta_1 + \Delta_2$, where $\Delta_1 = (x_1 = x_2 = y = 0)$ and $\Delta_2 = (x_1 = x_2 = w = 0)$.
Both $\Delta_1$ and $\Delta_2$ intersect $\Gamma$ at a nonsingular point so that $(\Gamma \cdot \Delta_i)_S \ge 1 \ge \deg \Delta_i$ for $i = 1,2$.
The computation of intersection numbers in case (2) of Lemma \ref{admCdeg1G6} can be done in the same way and we omit it.

Suppose that we are in case (3) of Lemma \ref{admCdeg1G6}.
Then $\Delta$ is irreducible and reduced since $\nu \ne 0$.
If $\lambda \ne \beta$, then $\Delta$ intersects $\Gamma$ at two points away from $\msp_4$, so that $(\Gamma \cdot \Delta)_S \ge 2 > \deg \Delta$.
In the following, we assume $\lambda = \beta$.
Note that $\beta \ne 0$ in this case since $(\beta,\lambda) \ne (0,0)$.
In this case, $\msp_4$ is the unique intersection point of $\Gamma$ and $\Delta$ and we cannot compute $(\Gamma \cdot \Delta)_S$ directly.
We will compute $(\Gamma^2)_S$ by determining the extended dual graph $G(S,\msp_4,\Gamma)$.

Let $\varphi \colon Y' \to X'$ be the weighted blowup of $X'$ at $\msp_4$ with weight $\wt (x_0,x_1,x_2,y) = (2,1,1,2)$ and with exceptional divisor $E$.
We define $\psi := \varphi|_{\tilde{S}} \colon \tilde{S} \to S$ and set $E_{\psi} = E|_{\tilde{S}}$.
We see that $S$ is cut out on $X'$ by the section
\[
s := x_0 \ell_0 + x_1 \ell_1 + \mu (y+\beta x_2^2),
\]
where $\ell_1,\ell_2$ are general linear forms in $x_0,x_1,x_2,w$ and $\mu \in \mbC$ is general.
Note that $K_{\tilde{S}} = \psi^*K_S$ by adjunction since $K_{Y'} = \varphi^*K_{X'} + E$ and $s$ vanishes along $E$ to order $1$.
Let $\xi_i \in \mbC$ be the coefficient of $w$ in $\ell_i$.
Let $G$ and $H$ be the $\varphi$-weight $= 4$ and $5$ parts of $F (x_0,x_1,x_2,y,1) = x_0 y + f_4 + g_5$, respectively.
We have $G = x_0 y + \bar{f}_4$.
Here, $\bar{f}_4 = f_4 (0,x_1,x_2,y)$.
Then, we have an isomorphism
\[
E_{\psi} \cong (x_0 y + \bar{f}_4 = x_1 = 0) \subset \mbP (2_{x_0},1_{x_1},1_{x_2},2_{y}),
\]
and
\[
J_{\psi} =
\begin{pmatrix}
y & \frac{\prt \bar{f}_4}{\prt x_1} & \frac{\prt \bar{f}_4}{\prt x_2} & \frac{\prt \bar{f}_4}{\prt y} & H \\
0 & \xi_1 & 0 & 0 & \xi_0 x_0 + x_1 \bar{\ell}_1 + \mu (y + \beta x_2^2) 
\end{pmatrix}.
\]
Set 
\[
\Sigma := \left( y = \frac{\prt \bar{f}_4}{\prt x_2} = \frac{\prt \bar{f}_4}{\prt y} = 0 \right) \cap E_{\psi}
= \left(x_1 = y = \bar{f}_4 = \frac{\prt \bar{f}_4}{\prt x_2} = \frac{\prt \bar{f}_4}{\prt y} = 0 \right).
\]
We see that $\Sigma = \emptyset$ if and only if $x_2^4 \in f_4$.
The latter holds true since
\[
F (0,0,x_2,y,w) = (y + \beta x_2^2)(w y + \beta w x_2^2 + \nu x_2^3)
\]
and $\beta \ne 0$.
This shows that $J_{\psi}$ is of rank $2$ at every point of $E_{\psi}$ and hence singular points of $\tilde{S}$ consists of two points $\msq_1 = (1 \!:\! 0 \!:\! 0 \!:\! 0)$ and $\msq_2 = (1 \!:\! 0 \!:\! 0 \!:\! -1)$ both of type $A_1$.
Note that $\tilde{\Gamma}$ intersects $E_{\psi}$ at $(0 \!:\! 0 \!:\! 1 \!:\! - \beta) \ne \msq_1, \msq_2$.
By considering the blowups of $\tilde{S}$ at $\msq_1$ and $\msq_2$, we see that $G(S,\msp_4,\Gamma)$ is of type $A_{3,2}$.
Thus, by Lemma \ref{lem:compselfint}, 
\[
(\Gamma^2)_S = -2 - \deg \Gamma + \frac{3}{2} = - \frac{3}{2}.
\]
By taking the intersection number of $\Gamma$ and $T|_S = \Gamma + \Delta$, we have
\[
1 = (\Gamma \cdot T|_S)_S = (\Gamma^2)_S + (\Gamma \cdot \Delta)_S = - \frac{3}{2} + (\Gamma \cdot \Delta)_S,
\]
and thus $(\Gamma \cdot \Delta)_S = \frac{5}{2} > \deg \Delta$.
This completes the proof.
\end{proof}

\subsection{Curves of degree $2$ on $X' \in \mcG'_6$.}

Let $X' = X'_5 \subset \mbP (1,1,1,2,1)$ be a member of $\mcG'_6$ and $\Gamma \subset X'$ an irreducible and reduced curve of degree $2$ that passes through $\msp_4$ but does not pass through the other singular points.
We see that $\Gamma$ is a WCI curve of type either $(1,1,4)$ or $(1,2,2)$.

\begin{Lem}
No curve of type $(1,1,4)$ is a maximal center.
\end{Lem}

\begin{proof} 
Let $\Gamma \subset X'$ be a curve of type $(1,1,4)$.
We claim that $\Gamma \not\subset (x_0 = 0)$.
Indeed, if $\Gamma \subset (x_0 = 0)$, then we may assume $\Gamma = (x_0 = x_1 = h_4 = 0)$ for some $h_4 \in \mbC [x_2,y,w]$ after replacing $x_0,x_1$.
We have
\[
G := F (0,0,x_2,y,w) = w (y^2 + \alpha y x_2^2 + \beta x_2^4) + \gamma y x_2^3 + \delta x_2^5
\]
for some $\alpha,\dots,\delta \in \mbC$.
Since $\Gamma \subset X'$, $G$ is reducible, which is equivalent to the condition $\gamma = \delta = 0$.
Then, we have $h_4 = y^2 + \alpha y x_2^2 + \beta x_2^4$, which  implies that $\Gamma$ is either reducible or non-reduced.
This is a contradiction.

Thus, $\Gamma \not \subset (x_0 = 0)$ and we may assume $\Gamma = (x_1 = x_2 = h_4 = 0)$ for some $h_4 \in \mbC [x_0,y,w]$ after replacing $x_0,x_1$.
Let $S$ and $T$ be general members of the pencil $|\mcI_{\Gamma} (A)|$ generated by $x_1,x_2$.
By an explicit computation, we have $h_4 = y^2 + y w x_0 + \alpha y x_0^2 + \beta x_0^4$ for some $\alpha,\beta \in \mbC$ and $T|_S = \Gamma + \Delta$, where $\Delta = (x_1 = x_2 = w + \gamma x_0 = 0)$ for some $\gamma \in \mbC$.
Note that $S$ is nonsingular along $\Gamma \setminus \{\msp_4\}$ by Lemma \ref{lem:qsmcri}.
We see that $\Gamma$ intersects $\Delta$ at two nonsingular points so that $(\Gamma \cdot \Delta) \ge 2 > (A \cdot \Delta) = 1/2$.
Therefore, $\Gamma$ is not a maximal center.
\end{proof}

In the following, we treat the case where $\Gamma$ is of type $(1,2,2)$.
We write the defining polynomial of $X'$ as $F' = w^2 x_0 y + w (y^2 + y a_2 + a_4) + y b_3 + b_5$, where $a_i, b_i \in \mbC [x_0,x_1,x_2]$.
For a polynomial $h = h (x_0,x_1,x_2)$, we set $\bar{h} = h (0,x_1,x_2)$ and $\tilde{h} = h (x_0,x_1,0)$.

\begin{Lem} \label{lem:G6Cdeg2classif}
Let $\Gamma \subset X'$ be a curve of type $(1,2,2)$ passing through $\msp_4$.
Then, by an admissible change of coordinates, $\Gamma$ is one of the following.
\begin{enumerate}
\item $(x_0 = y - c_2 = w x_1 + d_2 = 0)$ for some $c_2, d_2 \in \mbC [x_1,x_2]$ such that $x_1 \nmid d_2$.
\item $(x_2 = y - c_2 = w x_0 + d_2 = 0)$ for some $c_2,d_2 \in \mbC [x_0,x_1]$ such that $x_0 \nmid d_2$.
In this case, if $a_4 (0,1,0) = 0$, then either $c_2 (1,0) \ne 0$ or $a_2 (0,1,0) \ne d_2 (1,0)$.
\item $(x_2 = y - \lambda w x_0 - c_2 = w x_1 + d_2 = 0)$ for some $\lambda \in \mbC$,  $c_2, d_2 \in \mbC [x_0,x_1]$ such that $x_1 \nmid d_2$.
\end{enumerate}
\end{Lem}

\begin{proof}
Suppose $\Gamma \subset (x_0 = 0)$. 
We will show that we are in case (1).
We can write $\Gamma = (x_0 = y - w c_1 - c_2 = w d_1 + d_2 = 0)$ for some $c_1,c_2,d_1,d_2 \in \mbC [x_1,x_2]$.
Since $\Gamma$ is irreducible and reduced, we have $d_1 \ne 0$ and $d_2$ is not divisible by $d_1$.
Hence, we may assume $d_1 = x_1$ and $x_1 \nmid d_2$.
After replacing $y - w c_1 - c_2$ with $(y-w c_1 - c_2) - \eta (w x_1 + d_2)$ for some $\eta \in \mbC$, we may assume that either $c_1 = 0$ or $x_1 \nmid c_1$.
Since $\Gamma \subset X'$, there exists $h = h (x_1,x_2,w)$ such that $F' (0,x_1,x_2,w c_1 + c_2,w) = (w x_1 + d_2) h$.
We have
\[
F' (0,x_1,x_2,w c_1 + c_2,w) = w^3 c_1^2 + \cdots,
\]
where $\cdots$ consists of low degree terms with respect to $w$.
By writing $h = w^2 e_1 + w e_2 + e_3$ for some $e_i \in \mbC [x_1,x_2]$ and comparing terms divisible by $w^3$, we have $w^3 c_1^2 = w^3 x_1 e_1$.
This shows that $c_1$ is divisible by $x_1$ and thus $c_1 = 0$.
Thus we are in case (1).

In the following, we assume $\Gamma \not\subset (x_0 = 0)$.
Then, we may assume $\Gamma \subset (x_2 = 0)$ and $\Gamma = (x_2 = y - w c_1 - c_2 = w d_1 + d_2 = 0)$ for some $c_1,c_2,d_1,d_2 \in \mbC [x_0,x_1]$.
As in the above argument, we see that $d_1 \ne0$, $d_1 \nmid d_2$ and we may assume that either $c_1 = 0$ or $d_1 \nmid c_1$.
If $d_1$ is not proportional to $x_0$, then we may assume $d_1 = x_1$ and thus we are in case (3).
Suppose that $d_1$ is proportional to $x_0$.
Then we may assume $d_1 = x_0$.
Since $\Gamma \subset X'$, there exists $h = h (x_0,x_1,w)$ such that $F' (x_0,x_1,0,w c_1 + c_2,w) = (w x_0 + d_2)h$.
We have
\[
\begin{split}
F' (x_0,x_1, 0,w c_1 + c_2,w) =  w^3 c_1 & (x_0 + c_1) + w^2 (2 c_1 c_2 + c_1 \tilde{a}_2 + x_0 c_2) \\
& + w (c_2^2 + \tilde{a}_2 c_2 + \tilde{a}_4 + c_1 \tilde{b}_3) + \tilde{b}_3 c_2 + \tilde{b}_5.
\end{split}
\]
By writing $h = w^2 e_1 + w e_2 + e_3$, where $e_i \in \mbC [x_0,x_1]$, and comparing terms divisible by $w^3$, we have $c_1 (x_0 + c_1) = x_0 e_1$.
If $c_1 \ne 0$, then $e_1 \mid c_1$ since $x_0 \nmid c_1$ by our choice of $c_1$.
But then the equation $c_1 (x_0 + c_1) = x_0 e_1$ implies $x_0 \mid c_1$.
This is a contradiction and we have $c_1 = e_1 = 0$.
Thus we are in case (2).
It remains to show that if $a_4 (0,1,0) = 0$, then either $c_2 (1,0) \ne 0$ or $a_2 (0,1,0) \ne d_2 (1,0)$.
Since $c_1 = e_1 = 0$, we have $e_2 = c_2$ and the equations
\begin{equation} \label{eq:G6curvedeg2-1}
c_2^2 + \tilde{a}_2 c_2 + \tilde{a}_4 = x_0 e_3 + d_2 c_2 \text{ and } \tilde{b}_3 c_2 + \tilde{b}_5 = d_2 e_3.
\end{equation}
Let $\alpha_2$, $\alpha_3$ and $\beta_5$ be the coefficients of $x_1^2$, $x_0 x_1^3$ and $x_1^5$ in $\tilde{a}_2$, $\tilde{a}_4$ and $\tilde{b}_5$, respectively.
Now we assume $a_4 (0,1,0) = c_2 (1,0) = a_2 (0,1,0) - d_2 (1,0) = 0$.
This implies that the coefficients of $x_1^2$ in $d_2$ is $\alpha_2$.
Let $\varepsilon$ be the coefficient of $x_1^3$ in $e_3$.
By comparing the coefficients of $x_0 x_1^3$ in the first equation of \eqref{eq:G6curvedeg2-1}, we have $\alpha_2 \gamma + \alpha_2 = \varepsilon + \alpha_2 \gamma$ since $x_1 \mid c_2$.
By comparing the coefficients of $x_1^5$ in the second equation in \eqref{eq:G6curvedeg2-1}, we have $\beta_5 = \alpha_2 \varepsilon$.
This shows $\beta_5 = \alpha_2 \alpha_3$ and thus $(x_0,x_1.x_2) = (0,1,0)$ is a solution of $x_0 = a_4 = b_5 - a_2 \prt a_4/\prt x_0 = 0$.
This is impossible by Condition \ref{addcond} and the proof is completed.
\end{proof}

Let $\varphi \colon Y' \to X'$ be the weighted blowup of $X'$ at $\msp_4$ with $\wt (x_0,x_1,x_2,y) = (2,1,1,2)$.
Let $S \in |\mcI_{\Gamma} (2A)|$ be a general member.
We set $\psi = \varphi|_{\tilde{S}} \colon \tilde{S} \to S$ and $E_{\psi} = E|_{\tilde{S}}$.
Note that $K_{\tilde{S}} = \psi^* K_S$.
Let $G$ and $H$ be the $\varphi$-weight $= 4$ and $= 5$ parts of $F' (x_0,x_1,x_2,y,1)$, respectively.
We write $a_2 = \bar{a}_2 + x_0 a_1$ and $a_4 = \bar{a}_4 + x_0 a_3$, where $a_1, a_3 \in \mbC [x_1,x_2]$.
Then, we have $G = x_0 y + y^2 + y \bar{a}_2 + \bar{a}_4$ and $H = y x_0 \bar{a}_1 + x_0 \bar{a}_3 + y \bar{b}_3 + \bar{b}_5$. 
We see that $E$ is isomorphic to $(G = 0) \subset \mbP (2,1,1,2)$.

We compute $(\Gamma^2)_S$ by determining the extended resolution graph of $(S,\msp_4,\Gamma)$.

\begin{Lem}
The type of the extended resolution graph $G (S,\msp_4,\Gamma)$ is one of $A_{3,1}$, $A_{3,2}$ and $A_{4,2}$.
In particular, $(\Gamma^2) \le -14/5$ and $\Gamma$ is not a maximal center.
\end{Lem}

\begin{proof}
It is easy to compute $(\Gamma^2)_S$ once $G (S,\msp_4,\Gamma)$ is determined.
Indeed, if $G(S,\msp_4,\Gamma)$ is of type one of $A_{3,1}$, $A_{3,2}$ and $A_{4,2}$, then, by Lemma \ref{lem:compselfint}, we have
\[
(\Gamma^2)_S \le -2 - \deg \Gamma + \frac{6}{5} = - \frac{14}{5}.
\]
Thus
\[
(A^2 \cdot S) - 2 (A \cdot \Gamma) + (\Gamma^2)_S \le 5 - 4 - \frac{14}{5} < 0,
\]
and, by Lemma \ref{lem:criexclC1}, $\Gamma$ is not a maximal center.
The rest is devoted to the determination of $G (S,\msp_4,\Gamma)$.

Suppose that we are in case (1) of Lemma \ref{lem:G6Cdeg2classif}.
Then, the surface $S$ is cut out on $X'$ by the section
\[
s := x_0 (\alpha_0 x_0 + \alpha_1 x_1 + \alpha_2 x_2 + \beta w) + \gamma (y - c_2) + \delta (w x_1 + d_2),
\]
where $\alpha_i, \beta, \gamma,\delta \in \mbC$ are general, and we have an isomorphism
\[
E_{\psi} \cong (x_0 y + y^2 + y \bar{a}_2 + \bar{a}_4 = x_1 = 0) \subset \mbP (2_{x_0},1_{x_1},1_{x_2},2_y).
\]
We have
\[
J_{\psi} =
\begin{pmatrix}
\frac{\prt G}{\prt x_0} & \frac{\prt G}{\prt x_1} & \frac{\prt G}{\prt x_2} & \frac{\prt G}{\prt y} & H \\
0 & \delta & 0 & 0 & \beta x_0 + \gamma (y-c_2) + \delta d_2
\end{pmatrix}.
\]
It is clear that $J_{\psi}$ is of rank $2$ except possibly along 
\[
\Sigma := \left( \frac{\prt G}{\prt x_0} = \frac{\prt G}{\prt x_2} = \frac{\prt G}{\prt y} = 0 \right) \cap E_{\psi} 
= \left( x_1 = y = \bar{a}_4 = x_0 + \bar{a}_2 = \frac{\prt \bar{a}_4}{\prt x_1} = 0 \right).
\]
Note that $\Sigma \ne \emptyset$ if and only if $x_1 \mid \bar{a}_4$, and that $\Sigma = \{\msq_0\}$, where $\msq_0 = (- \bar{a}_2 (0,1) \!:\! 0 \!:\! 1 \!:\! 0)$,  if $\Sigma \ne \emptyset$.
We set $\msq_1 = (1 \!:\! 0 \!:\! 0 \!:\! 0)$ and $\msq_2 = (1 \!:\! 0 \!:\! 0 \!:\! -1)$.
Clearly $\msq_1$ and $\msq_2$ are $A_1$ points of $\tilde{S}$, and $\tilde{\Gamma}$ intersects $E_{\psi}$ at a single point $(0 \!:\! 0 \!:\! 1 \!:\! c_2 (0,1)) \in E_{\psi}$.

Assume that $E_{\psi}$ is irreducible.
This is equivalent to the condition $x_1 \nmid \bar{a}_4$.
Then $\Sigma = \emptyset$ and $\Sing (\tilde{S}) = \{\msq_1,\msq_2\}$.
Moreover, $\tilde{\Gamma}$ does not pass through $\msq_1, \msq_2$.
Hence, by considering the blowup of $\tilde{S}$ at $\msq_1$ and $\msq_2$, we see that $G (S,\msp_4,\Gamma)$ is of type $A_{3,2}$.

Assume that $E_{\psi}$ is reducible, that is, $x_1 \mid \bar{a}_4$.
We have $E_{\psi} = E_1 + E_2$, where $E_1 = (x_1 = y = 0)$ and $E_2 = (x_1 = x_0 + y + \bar{a}_2 = 0)$.
Note that $\msq_1 \in E_1$, $\msq_2 \in E_2$ and $\msq_0$ is the intersection point $E_1 \cap E_2$.
Note also that $x_1 \nmid \bar{a}_2$ because otherwise $x_0 = a_2 = a_4 = 0$ has a solution $(x_0,x_1,x_2) = (0,0,1)$.
This in particular implies that $\tilde{\Gamma}$ does not pass through $\msq_0$. 
We have
\[
J_{\psi} (\msq_0) =
\begin{pmatrix}
0 & \frac{\prt G}{\prt x_0} (\msq_0) & 0 & 0 & (-\bar{a}_2 \bar{a}_3 + \bar{b}_5) (\msq_0) \\
0 & \delta & 0 & 0 & (-\beta \bar{a}_2 - \gamma c_2 + \delta d_2)(\msq_0)
\end{pmatrix}.
\] 
If $(-\bar{a}_2 \bar{a}_3 + \bar{b}_5)(\msq_0) = 0$, then $x_0 = a_4 = b_5 - a_3 \prt a_4/\prt x_0 = 0$ has a solution $(x_0,x_1,x_2) = (0,0,1)$.
This is impossible by Condition \ref{addcond}. 
Hence $(-\bar{a}_2 \bar{a}_3 + \bar{b}_5)(\msq_0) \ne 0$ and this implies $\rank J_{\psi} (\msq_0) = 2$ since $\delta$ is general.
It follows $\Sing (\tilde{S}) = \{\msq_1,\msq_2\}$.
By considering blowups of $\tilde{S}$ at $\msq_1$ and $\msq_2$, we see that $G (S,\msp_4,\Gamma)$ is of type $A_{4,2}$.
 
 Suppose that we are in case (2) of Lemma \ref{lem:G6Cdeg2classif}.
The surface $S$ is cut out on $X'$ by the section
\[
s := x_2 (\alpha_0 x_0 + \alpha_1 x_1 + \alpha_2 x_2 + \beta w) + \gamma (y-c_2) + \delta (w x_0 + d_2),
\]
where $\alpha_i, \beta, \gamma,\delta \in \mbC$ are general, and we have an isomorphism
\[
E_{\psi} \cong (x_0 y + y^2 + y \bar{a}_2 + \bar{a}_4 = x_2 = 0) \subset \mbP (2_{x_0},1_{x_1},1_{x_2},2_y).
\]
We have
\[
J_{\psi} =
\begin{pmatrix}
\frac{\prt G}{\prt x_0} & \frac{\prt G}{\prt x_1} & \frac{\prt G}{\prt x_2} & \frac{\prt G}{\prt y} & H \\
0 & 0 & \beta & 0 & x_2 (\alpha_1 x_1 + \alpha_2 x_2) + \gamma (y-\bar{c}_2) + \delta (x_0 + \bar{d}_2)
\end{pmatrix}.
\]
It is clear that $J_{\psi}$ is of rank $2$ outside the set 
\[
\Sigma := \left( \frac{\prt G}{\prt x_0} = \frac{\prt G}{\prt x_1} = \frac{\prt G}{\prt y} = 0 \right) \cap E_{\psi} 
=  \left(y = \frac{\prt \bar{a}_4}{\prt x_1} = x_0 + \bar{a}_2 = \bar{a}_4 = x_2 = 0 \right).
\]
Note that $\Sigma \ne \emptyset$ if and only if $x_2 \mid \bar{a}_4$ and $\Sigma = \{\msq_0\}$, where $\msq_0 = (-\bar{a}_2 (1,0) \!:\! 1 \!:\! 0 \!:\! 0) \in \mbP (2,1,1,2)$, if $\Sigma \ne \emptyset$. 
Clearly $\msq_1 := (1 \!:\! 0 \!:\! 0 \!:\! 0)$ and $\msq_2 := (1 \!:\! 0 \!:\! 0 \!:\! -1)$ are $A_1$ points of $\tilde{S}$.
Note also that $\tilde{\Gamma}$ intersects $E_{\psi}$ at a single point $(- d_2 (1,0) \!:\! 1 \!:\! 0 \!:\! c_2 (1,0))$.

Assume that $E_{\psi}$ is irreducible, which is equivalent to the condition  $x_2 \nmid \bar{a}_4$.
Then $\Sigma = \emptyset$ and $\Sing (\tilde{S}) = \{\msq_1,\msq_2\}$. 
Hence $G (S,\msp_4,\Gamma)$ is of type $A_{3,2}$.

Assume that $E_{\psi}$ is reducible, that is, $x_2 \mid \bar{a}_4$.
Then $\Sigma = \{\msq\}$ and $E_{\psi} = E_1 + E_2$, where $E_1 = (y = x_2 = 0)$ and $E_2 = (x_0 + y + \bar{a}_2 = x_2 = 0)$.
Note that $\msq_1 \in E_1$, $\msq_2 \in E_2$ and $E_1 \cap E_2 = \{\msq_0\}$.
Note also that $\tilde{\Gamma}$ does not pass through $\msq_0$ since either $c_2 (1,0) \ne 0$ or $\bar{a}_2 (1,0) = a_2 (0,1,0) \ne d_2 (1,0)$ by Lemma \ref{lem:G6Cdeg2classif}.
By the same argument as in case (1), we see that $H (\msq_0) = (-\bar{a}_2 \bar{a}_3 + \bar{b}_5)(\msq_0) \ne 0$, and hence $\rank J_{\psi} (\msq_0) = 2$ since $\beta$ is general.
It follows that $\Sing (\msp_4) = \{\msq_1,\msq_2\}$ and $G (S,\msp_4,\Gamma)$ is of type $A_{4,2}$.

Suppose that we are in case (3) of Lemma \ref{lem:G6Cdeg2classif}.
The surface $S$ is cut out on $X'$ by the section
\[
s := x_2 (\alpha x_0 + \alpha_1 x_1 + \alpha_2 x_2 + \beta w) + \gamma (y- \lambda w x_0 - c_2) + \delta (w x_1 + d_2),
\]
where $\alpha_i, \beta, \gamma, \delta \in \mbC$ are general.
We have an isomorphism
\[
E_{\psi} \cong (x_0 y + y^2 + y \bar{a}_2 + \bar{a}_4 = \delta x_1 + \beta x_2 = 0) \subset \mbP (2_{x_0},1_{x_1},1_{x_2},2_y),
\]
which is irreducible since $\bar{a}_4 \ne 0$ and $\beta, \delta$ are general.
We have
\[
J_{\varphi} =
\begin{pmatrix}
\frac{\prt G}{\prt x_0} & \frac{\prt G}{\prt x_1} & \frac{\prt G}{\prt x_2} & \frac{\prt G}{\prt y} & H \\
0 & \delta & \beta & 0 & x_2 (\alpha_1 x_1 + \alpha_2 x_2) + \gamma (y-\lambda x_0 - \bar{c}_2) + \delta \bar{d}_2
\end{pmatrix}.
\]
We see that
\[
\left( \frac{\prt G}{\prt x_0} = \frac{\prt G}{\prt y} = 0 \right) \cap E_{\psi}
= (y = x_0 + \bar{a}_2 = \bar{a}_4 = \delta x_1 + \beta x_2 = 0) = \emptyset
\]
for a general $\beta, \delta$.
This shows that $J_{\varphi}$ is of rank $2$ at every point of $E_{\varphi}$ and thus $\tilde{S}$ has only two singular points $\msq_1 = (1 \!:\! 0 \!:\! 0 \!:\! 0)$ and $\msq_2 = (1 \!:\! 0 \!:\! 0 \!:\! -1)$ of type $A_1$.
We see that $\tilde{\Gamma}$ intersects $E_{\varphi}$ transversally at $\msq_1$.
Hence, by considering the blow-up of $\tilde{S}$ at $\msq_1$ and $\msq_2$, we see that $G (S,\msp,\Gamma)$ is of type $A_{3,1}$.
This completes the proof.
\end{proof}

\subsection{Curves of degree $1/2$ on $X' \in \mcG'_7$.}

Let $X' = X'_6 \subset \mbP (1,1,1,2,2)$ be a member of $\mcG'_7$ with defining polynomial $F' = w^2 x_0 x_1 + w f_4 + g_6$ and let $\Gamma \subset X'$ be an irreducible and reduced curve of degree $1/2$ that passes through $\msp_4$ but does not pass through the other singular points.
We see that $\Gamma$ is a WCI curve of type $(1,1,2)$ and it is contained in either $(x_0 = 0)$ or $(x_1 = 0)$.
By Condition \ref{addcond}, $F'$ can be written as $F' = w^2 x_0 x_1 + w (y^2 + y a_2 + a_4) + y b_4 + b_5$ for some $a_i, b_i \in \mbC [x_0,x_1,x_2]$.
Let $S,T$ be a general member of $|\mcI_{\Gamma} (A)|$. 

\begin{Lem} \label{lem:deltaCG7}
We have $T|_S = \Gamma + \Delta$, where $\Delta$ is an irreducible and reduced curve of degree $1$.
Moreover, one of the following hold.
\begin{enumerate}
\item $\Delta$ intersects $\Gamma$ away from $\msp_4$.
\item After an admissible coordinates change, $\Gamma = (x_0 = x_2 = y = 0)$ and $\Delta = (x_0 = x_2 = w y + \xi x_0^4 = 0)$ for some nonzero $\xi \in \mbC$.
\end{enumerate}
\end{Lem}

\begin{proof}
Suppose $\Gamma \subset (x_0 = x_1 = 0)$.
Then, $\Gamma = (x_0 = x_1 = y + \gamma x_2^2 = 0)$ for some $\alpha \in \mbC$.
We have
\[
\bar{F}' := F (0,0,x_2,y,w) = w (y^2 + \alpha_2 y x_2^2 + \alpha_4 x_2^4) + \beta_4 y x_2^4 + \beta_6 x_2^6,
\]
where $\alpha_i$ and $\beta_i$ are the coefficients of $x_2^i$ in $a_i$ and $b_i$, respectively.
Since $\Gamma \subset X'$, we have
\[
\bar{F}' = (y+\gamma x_2^2)(w y + \delta w x_2^2 + \varepsilon x_2^4)
\]
for some $\delta,\varepsilon \in \mbC$.
By comparing the coefficients of each term of $\bar{F}$, we have $\alpha_2 = \gamma + \delta$, $\alpha_4 = \gamma \delta$, $\beta_4 = \varepsilon$ and $\beta_6 = \gamma \varepsilon$.
Note that $\Delta = (x_0 = x_1 = w y + \delta w x_2^2 + \varepsilon x_2^4 = 0)$, which is irreducible if and only if $\varepsilon \ne 0$.
If $\varepsilon = 0$, then $\beta_4 = \beta_6 = 0$, which in particular implies that $x_0 = b_4 = b_6 = 0$ has a solution $(x_0,x_1,x_2) = (0,0,1)$.
This is impossible by Condition \ref{addcond} and hence $\varepsilon \ne 0$.
If $\gamma = \delta$, then $4 \alpha_4 - \alpha_2^2 = 0$, which implies that $x_0 = x_1 = 4 a_4 - a_2^2 = 0$ has a solution $(x_0,x_1,x_2) = (0,0,1)$.
This is again impossible.
Hence $\gamma \ne \delta$.
It follows that $\Delta$ intersects $\Gamma$ at a point other than $\msp_4$ and we are in case (1).

Suppose $\Gamma \not\subset (x_0 = x_1 = 0)$.
Then, since $\Gamma$ is contained in either $(x_0 = 0)$ or $(x_1 = 0)$, we may assume $\Gamma = (x_0 = x_2 = y + \gamma x_1^2 = 0)$ after possibly interchanging $x_0$ with $x_1$ and replacing $x_2$. 
As in the above argument, we have
\[
\bar{F}' := F (0,x_1,0,y,w) = w (y^2 + \alpha_2 y x_1^2 + \alpha_4 x_1^4) + \beta_4 y x_1^4 + \beta_6 x_1^6,
\]
where $\alpha_i$ and $\beta_i$ are the coefficients of $x_1^i$ in $a_i$ and $b_i$, respectively, and
\[
\bar{F}' = (y+\gamma x_1^2)(w y + \delta w x_1^2 + \varepsilon x_1^4)
\]
for some $\delta,\varepsilon \in \mbC$.
We see that $\varepsilon \ne 0$ because otherwise $\beta_4 = \beta_6 = 0$ and this contradicts to Condition \ref{addcond}.
It follows that $\Delta = (x_0 = x_2 = w y + \delta y x_1^2 + \varepsilon x_1^4 = 0)$ is irreducible and reduced.
If $\gamma \ne \delta$, then $\Delta$ intersects $\Gamma$ at a point other than $\msp_4$ and we are in case (1).
If $\beta = \gamma$, then, after replacing $y+\gamma x_1^2$ with $y$, we have $\Gamma = (x_0 = x_2 = y = 0)$ and $\Delta = x_0 = x_2 = w y + \varepsilon x_1^4 = 0)$, that is, we are in case (2).
This completes the proof.
\end{proof}

Note that $S$ is nonsingular along $\Gamma \setminus \{\msp_4\}$.

\begin{Lem} \label{lem:exclCG7}
Notation as in \emph{Lemma \ref{lem:deltaCG7}}.
Then, $(\Gamma \cdot \Delta)_S \ge \deg \Delta$.
In particular, $\Gamma$ is not a maximal center.
\end{Lem}

\begin{proof}
If we are in case (1) of Lemma \ref{lem:deltaCG7}, then $(\Gamma \cdot \Delta)_S \ge 1 > \deg \Delta$.
In the following, we assume that we are in case (2).
Let $\varphi \colon Y' \to X'$ be the weighted blowup of $X'$ at $\msp_4$ with $\wt (x_0,x_1,x_2,y) = \frac{1}{2} (3,1,1,2)$ with exceptional divisor $E$.
Let $\psi = \varphi|_{\tilde{S}} \colon \tilde{S} \to S$ and set $E_{\psi} = E|_{\tilde{S}}$.
Note that $a_2,a_4,b_6 \in (x_0,x_2)$ and $b_4 \notin (x_0,x_2)$.
We see that $S$ is cut out on $X'$ by the section $s := \lambda x_0 + \mu x_2$, where $\lambda,\mu \in \mbC$ are general.
We have
\[
E_{\psi} \cong (x_0 x_1 + y^2 + y \bar{a}_2 + \bar{a}_4 = x_2 = 0) \subset \mbP (3_{x_0},1_{x_1},1_{x_2},2_y).
\]
Since $\bar{a}_i$ is divisible by $x_2$, we have $E_{\psi} = (x_0 x_1 + y^2 = x_2 = 0)$.
In particular, $E_{\psi}$ is quasismooth in $\mbP (3,1,1,2)$.
This implies that $\rank J_{E_{\psi}} = 2$, and thus $\rank J_{\psi} = 2$, at every point of $E_{\psi}$.
It follows that $\tilde{S}$ has a singular point of type $A_2$ at $\msq = (1 \!:\! 0 \!:\! 0 \!:\! 0)$ and it is nonsingular along $E_{\psi} \setminus \{\msq\}$.
Note that $\tilde{\Gamma}$ intersects $E_{\psi}$ at $(0 \!:\! 1 \!:\! 0 \!:\! 0) \ne \msq$.
By considering the resolution of $\tilde{S}$ at $\msq$, we see that $G(S,\msp_4,\Gamma)$ is of type $A_{3,1}$.
Thus we have 
\[
(\Gamma^2)_S = -2 + \frac{3}{4} = - \frac{5}{4}.
\]
By taking the intersection number of $T|_S = \Gamma + \Delta$ and $\Gamma$, we have
\[
\frac{1}{2} = (\Gamma \cdot T|_S)_S = (\Gamma^2)_S + (\Gamma \cdot \Delta)_S,
\]
and then we have $(\Gamma \cdot \Delta)_S = 7/4 > \deg \Delta$.
Therefore, $\Gamma$ is not a maximal center by Lemma \ref{lem:criexclC2}.
\end{proof}

\begin{Rem} \label{rem:intCG7}
Let $\Gamma \subset X'$ be an irreducible and reduced curve of degree $1/2$ passing through $\msp_4$.
Then, Lemma \ref{lem:exclCG7} shows that $(\Gamma^2)_S \le 0$ for a general member $S \in |\mcI_{\Gamma} (A)|$.
This follows by considering $1/2 = (\Gamma \cdot T|_S)_S = (\Gamma^2)_S + (\Gamma \cdot \Delta)_S$ and $(\Gamma \cdot \Delta)_S \ge \deg \Delta = 1/2$.
This observation will be used in the next subsection.
\end{Rem}

\subsection{Curves of degree $1$ on $X' \in \mcG'_7$.}

Let $\Gamma$ be an irreducible and reduced curve on a member $X'$ of $\mcG'_7$ that passes through the $cA/2$ point but does not pass through the other singular points.
We see that $\Gamma$ is a WCI curve of type either $(1,2,2)$ or $(1,1,4)$.

We claim that $\Gamma$ cannot be of type $(1,2,2)$.
If $\Gamma$ is of type $(1,2,2)$, then $\Gamma = (\ell = y + c_2 = d_2 = 0)$ for some $\ell \in \mbC [x_0,x_1,x_2]$ with $\deg \ell = 1$ and $c_2, d_2 \in \mbC [x_0,x_1,x_2]$ since $\Gamma$ passes through $\msp_4$.
In this case $\Gamma$ is either reducible or non-reduced.
This shows that $\Gamma$ cannot be of type $(1,2,2)$. 

\begin{Lem}
An irreducible and reduced curve of degree $1$ on $X'$ passing through $\msp_4$ is not a maximal center.
\end{Lem}

\begin{proof}
Let $\Gamma \subset X'$ be an irreducible and reduced curve of degree $1$ passing through $\msp_4$.
Then, by the above argument, $\Gamma$ is of type $(1,1,4)$.
Let $S, T \in |\mcI_{\Gamma} (A)|$ be a general member and let $\ell_1, \ell_2 \in \mbC [x_0,x_1,x_2]$ be linear forms such that $S = (\ell_1 = 0)$ and $T = (\ell_2 = 0)$.
Then $T|_S = \Gamma + \Delta$, where $\Delta = (\ell_1 = \ell_2 = d_2 = 0)$ for some $d_2 \in \mbC [x_0,x_1,x_2,y,w]$ of degree $2$.
We claim that $d_2 \notin \mbC [x_0,x_1,x_2]$.
Indeed, if $d_2 \in \mbC [x_0,x_1,x_2]$, then the defining polynomial $F'$ of $X'$ is contained in the ideal $(\ell_1,\ell_2,d_2) \subset (x_0,x_1,x_2)$.
This is a contradiction since $w y^2 \in F'$.
Hence, either $w \in d_2$ or $y \in d_2$, and in particular $\Delta$ is irreducible and reduced.
Suppose that $w \in d_2$.
Then $\Delta$ intersects $\Gamma$ at a nonsingular point so that $(\Gamma \cdot \Delta)_S \ge 1 > 1/2 = \deg \Delta$.
Thus $\Gamma$ is not a maximal center.
Suppose that $w \notin d_2$.
Then $y \in d_2$ and $\Delta$ does not pass through $\msp_4$.
Since $S$ is a general member of $|\mcI_{\Gamma} (A)| = |\mcI_{\Delta} (A)|$, we have $(\Delta^2)_S \le 0$ by Remark \ref{rem:intCG7}.
Then, we compute
\[
1/2 = (\Delta \cdot T|_S) = (\Gamma \cdot \Delta)_S + (\Delta^2)_S \le (\Gamma \cdot \Delta)_S.
\]
This shows that $\Gamma$ is not a maximal center.
\end{proof}

\subsection{Curves of degree $1$ on $X' \in \mcG'_9$.}

Let $X' = X'_6 \subset \mbP (1,1,1,3,1)$ be a member of $\mcG'_9$ and $\Gamma \subset X'$ an irreducible and reduced curve of degree $1$ that passes through $\msp_4$ but does not pass through the other singular points.
The defining polynomial of $X'$ can be written as $F' = w^2 x_0 y + w (y a_2 + a_5) + y^2 + y b_3 + b_6$, where $a_j,b_j \in \mbC [x_0,x_1,x_2]$.

\begin{Lem}
An irreducible and reduced curve of degree $1$ on $X'$ passing through $\msp_4$ is not a maximal center.
\end{Lem}

\begin{proof}
We see that $\Gamma$ is a WCI curve of type $(1,1,3)$.
Let $S, T$ be general members of the pencil $|\mcI_{\Gamma} (A)|$.
We will show that $T|_S = \Gamma + \Delta$, where $\Delta \ne \Gamma$ is an irreducible and reduced curve of degree $1$ such that $(\Gamma \cdot \Delta)_S \ge \deg\Delta$.

Suppose $\Gamma \not\subset (x_0 = 0)$.
Then, after replacing $x_1$ and $x_2$, we may assume $\Gamma \subset (x_1 = x_2 = 0)$.
We have
\[
F (x_0,0,0,y,w) = w^2 x_0 y + w (\alpha_2 y x_0^2 + \alpha_5 x_0^5) + y^2 + \beta_3 x_0^3 y + \beta_6 x_0^6,
\]
where $\alpha_i$ and $\beta_i$ are coefficients of $x_0^i$ in $a_i$ and $b_i$, respectively.
Since $\Gamma \subset X'$, we have $F' (x_0,0,0,y,w) = c_3 d_3$ for some $c_3, d_3 \in \mbC [x_0,y,w]$ of degree $3$.
Then, by an explicit computation, we have $\alpha_5 = \beta_6 = 0$, $c_3 = y$ and $d_3 = w^2 x_0 + \alpha_2 w x_0 + y + \beta_3 x_0^3$.
Note that either $\Gamma = (x_1 = x_2 = c_3 = 0)$ or $(x_1 = x_2 = d_3 = 0)$.
In any case, $\Delta$ is an irreducible and reduced curve of degree $1$ and we see that $S$ is nonsingular along $\Gamma \setminus \{\msp_4\}$ by Lemma \ref{lem:qsmcri}.
Since two curves $(x_1 = x_2 = c_3 = 0)$ and $(x_1 = x_2 = d_3 = 0)$ have a intersection point away from $\msp_4$, we have $(\Gamma \cdot \Delta)_S \ge 1 = (A \cdot \Delta)$.
This completes the proof.

Suppose $\Gamma \subset (x_0 = 0)$.
After replacing $x_1, x_2$, we may assume $\Gamma = (x_0 = x_1 = h_3 = 0)$ for some $h_3 (x_2,y,w)$.
We have 
\[
G := F(0,0,x_2,y,w) = w (\alpha_2 y x_2^2 + \alpha_5 x_2^5) + y^2 + \beta_3 y x_2^3 + \beta_6 x_2^6 = 0,
\]
where $\alpha_i, \beta_i \in \mbC$.
Note that $\alpha_i$ is the coefficient of $x_2^i$ in $a_i$ so that $(\alpha_2,\alpha_5) \ne (0,0)$ by Condition \ref{addcond}.
Since $\Gamma \subset X'$, $G$ is divisible by $h_3$ and we can write
\[
G = (y + \gamma w x_2^2 + \delta x_2^3)(y + \varepsilon w x_2^2 + \zeta x_2^3),
\]
where $h_3 = y + \gamma w x_2^2 + \delta x_2^3$ and $\gamma,\dots,\zeta \in \mbC$.
Note that $T|_S = \Gamma + \Delta$, where $\Delta = (x_0 = x_1 = y + \varepsilon w x_2^2 + \zeta x_2^3)$.
By comparing the coefficients of $w^2 x_2^4$, we have $\gamma \varepsilon = 0$.
If $(\gamma,\varepsilon) \ne (0,0)$, then $\gamma \ne \varepsilon$ and $\Delta$ intesects $\Gamma$ at a nonsingular point so that $(\Gamma \cdot \Delta)_S \ge 1 = \deg \Delta$.
Now suppose $\gamma = 0$.
By comparing coefficients of $w y x_2^2$ and $w x_2^5$, we have $\alpha_2 = \varepsilon$ and $\alpha_5 = \delta \varepsilon$.
This shows $\varepsilon \ne 0$.
It follows that $(\gamma,\varepsilon) \ne (0,0)$ and the proof is completed.
\end{proof}

\subsection{Curves of degree $1$ on $X' \in \mcG'_{10}$.}

Let $X' = X'_6 \subset \mbP (1,1,2,2,1)$ be a member of $\mcG'_{10}$ and $\Gamma \subset X'$ an irreducible and reduced curve of degree $1$ that passes through $\msp_4$ but does not pass through the other singular points.
The defining polynomial of $X'$ is of the form $w^2 y_0 y_1 + w f_5 + g_6$.
Note that $y_0^3, y_1^3 \in F'$ and we assume that the coefficients of $y_0^3$ and $y_1^3$ in $F'$ are both $1$ after re-scaling $y_0,y_1$. 
We see that $\Gamma$ is a WCI curve of type either $(1,1,4)$ or $(1,2,2)$.

We claim that $\Gamma$ cannot be of type $(1,1,4)$.
Indeed, a curve of type $(1,1,4)$ passing through $\msp_4$ is contained in $(x_0 = x_1 = 0)$ and we have
\[
F' (0,0,y_0,y_1,w) = w^2 y_0 y_1 + y_0^3 + y_1^3,
\]
which is clearly irreducible.
Thus, $X'$ cannot contain a curve of type $(1,1,4)$ passing through $\msp_4$.

In the following, we treat the case where $\Gamma$ is of type $(1,2,2)$.

\begin{Lem} \label{lem:classifCG10}
After replacing $x_0$ and $x_1$, and interchanging $y_0$ with $y_1$, we are in one of the following cases.
\begin{enumerate}
\item $\Gamma = (x_0 = y_0 - \beta x_1^2 = y_1 - \gamma w x_1 - \delta x_1^2 = 0)$ for some $\beta, \gamma,\delta \in \mbC$ with $\gamma \ne 0$.
\item $\Gamma = (x_0 = y_0 - \beta x_1^2 = y_1 = 0)$ for some non-zero $\beta \in \mbC$.
In this case, if $f_5 (x_0,x_1,0,0)$ is divisible by $x_0^2$, then $\tau - \sigma_0 \sigma_1 \ne 0$, where $\sigma_i$ and $\tau$ are the coefficients of $w y_i x_1^3$ and $x_1^6$ in $F'$, respectively.
\end{enumerate} 
\end{Lem}

\begin{proof}
After replacing $x_0,x_1$, we can write
\[
\Gamma = (x_0 = y_0 - \alpha w x_1 - \beta x_1^2 = y_1 - \gamma w x_1 - \delta x_1^2 = 0)
\] 
for some $\alpha,\beta,\gamma,\delta \in \mbC$.
It follows that
\[
G_{\alpha,\beta,\gamma,\delta} := F (0,x_1,\alpha w x_1 + \beta x_1^2, \gamma w x_1 + \delta x_1^2) = 0
\]
as a polynomial.
The coefficient of $w^4 x_1^2$ in $G_{\alpha,\beta,\gamma,\delta}$ is $\alpha \gamma$.
Hence $\alpha \gamma = 0$.
After interchanging $y_0$ and $y_1$ if necessary, we may assume $\alpha = 0$.
If $\gamma \ne 0$, then we are in case (1).
Suppose $\alpha = \gamma = 0$.
Then the coefficient of $w^2 x_1^4$ in $G_{0,\beta,0,\delta}$ is $\beta \gamma$.
Hence $\beta \gamma = 0$.
After interchanging $y_0$ and $y_1$, we may assume $\delta = 0$.
Hence $\Gamma = (x_0 = y_0 - \beta x_1^2 = y_1 = 0)$.
If further $\beta = 0$, then $F (0,x_1,0,0,w) = w f_5 (0,x_1,0,0) + g_6 (0,x_1,0,0) = 0$ as a polynomial, which implies that $f_5 (x_0,x_1,0,0)$ and $g_6 (x_0,x_1,0,0)$ share a common component $x_0$.
This is impossible by Condition \ref{addcond}.
Hence $\beta \ne 0$.
It remains to show that $\tau - \sigma_0 \sigma_1 \ne 0$ if $f_5 (x_0,x_1,0,0)$ is divisible by $x_0^2$.
Let $X \in \mcG_{10}$ be the birational counterpart of $X'$ which is defined in $\mbP (1,1,2,2,3,3)$, by
\[
\begin{split}
F_1 &= z_1 y_1 + z_0 y_0 + f_5 (x_0,x_1,y_0,y_1), \\
F_2 &= z_1 z_0 - g_6 (x_0,x_1,y_0,y_1).
\end{split}
\]
Now we assume that $f_5 (x_0,x_1,0,0)$ is divisible by $x_0^2$ and $\tau - \sigma_0 \sigma_1 = 0$.
Then $X$ contains $\msp := (0 \!:\! 1 \!:\! 0 \!:\! 0 \!:\! - \sigma_0 \!:\!: -\sigma_1)$ and every partial derivative of $F_1$ vanishes at $\msp$.
This is a contradiction since $X$ is quasismooth.
Therefore $\tau - \sigma_0 \sigma_1 \ne 0$ if $f_5 (x_0,x_1,0,0)$ is divisible by $x_0^2$ and the proof is completed.
\end{proof}

\begin{Lem}
No curve of type $(1,2,2)$ on $X' \in \mcG'_{10}$ passing through $\msp_4$ is a maximal cetner.
\end{Lem}

\begin{proof}
Let $S \in |\mcI_{\Gamma} (2 A)|$ be a general member and let $\varphi \colon Y' \to X'$ be the weighted blowup of $X'$ at $\msp_4$ with $\wt (x_0,x_1,y_0,y_1) = (1,1,3,2)$ with exceptional divisor $E$.
We set $\psi = \varphi |_{\tilde{S}} \colon \tilde{S} \to S$ and $E_{\psi} = E|_{\tilde{S}}$.
We see that $S$ is nonsingular along $\Gamma \setminus \{\msp_4\}$ by \cite[Lemma 2.5]{Okada3}.

Suppose that $\Gamma$ is as in (1) of Lemma \ref{lem:classifCG10}.
Then, $S$ is cut out by the section
\[
s := x_0 (\lambda_0 x_0 + \lambda_1 x_1 + \lambda w) + \mu (y_0 - \beta x_1^2) + \nu (y_1 - \gamma w x_1 - \delta x_1^2),
\]
where $\lambda_i, \lambda,\mu,\nu \in \mbC$ are general.
Note that $K_{\tilde{S}} = \psi^*K_S$.
We have
\[
E_{\psi} = (y_0 y_1 + \bar{f}_5 = \lambda x_0 - \nu \gamma x_1 = 0) \subset \mbP (1,1,2,3)
\]
and
\[
J_{\psi} =
\begin{pmatrix}
\frac{\prt \bar{f}_5}{\prt x_0} & \frac{\prt \bar{f}_5}{\prt x_1} & y_1 + \frac{\prt \bar{f}_5}{\prt y_0} & y_0 & H \\
\lambda & - \nu \gamma & 0 & 0 & x_0 (\lambda_0 x_0 + \lambda_1 x_1) + \mu (y_0-\beta x_1^2) - \mu \delta x_1^2
\end{pmatrix}.
\]
Since $f_5 (x_0,x_1,0,0) \ne 0$ as a polynomial and $\lambda, \nu$ are general, we may assume that $f_5 (x_0,x_1,0,0)$ is not divisible by $\lambda x_0 - \nu \gamma x_1$.
Then, we have
\[ 
\left( y_1 + \frac{\prt \bar{f}_6}{\prt y_0} = y_0 = 0 \right) \cap E_{\psi} = \emptyset,
\]
which implies that $J_{\psi}$ is of rank $2$ at every point of $E_{\psi}$.
Thus, singular points of $\tilde{S}$ consist of two points $\msq_1 := (0 \!:\! 0 \!:\! 1 \!:\! 0)$ and $\msq_2 := (0 \!:\! 0 \!:\! 0 \!:\! 1)$ that are of type $A_1$ and $A_2$, respectively. 
Let $\mbA^4$ be the orbifold chart of $Y'$ with affine coordinates $\tilde{x}_0,\tilde{x}_1,\tilde{y}_0,\tilde{y}_1$ such that $x_i = \tilde{x}_i \tilde{y}_1$, $y_0 = \tilde{y}_0 \tilde{y}_1^2$ and $y_0 = \tilde{y}_0^3$.
We see that $\mbZ_3$ acts on $\mbA^4$ as $\frac{1}{3} (1_{\tilde{x}_0},1_{\tilde{x}_1}, 2_{\tilde{y}_0}, 2_{\tilde{y}_1})$,
and the quotient $U := \mbA^4/\mbZ_3$ is an open subset of $Y'$ whose origin is $\msq_2$.
We see that $\tilde{S}$ is defined by $\tilde{y}_0 + \cdots = \lambda \tilde{x}_0 - \nu \gamma \tilde{x}_1 + \cdots = 0$ on $U$.
By eliminating $\tilde{y}_0$ and $\tilde{x}_0$, the germ $(\tilde{S}, \msq_2)$ is analytically isomorphic to $(\mbA^2_{\tilde{x}_1, \tilde{y}_1}/\mbZ_3, o)$.
Under the above isomorphism, $E_{\psi}$ and $\tilde{\Gamma}$ corresponds to $(\tilde{y}_1 = 0)$ and $(\tilde{x}_1 = 0)$, respectively.
Let $\hat{S} \to \tilde{S}$ be the weighted blowup of $\tilde{S}$ at $\msq_2$ with $\wt (\tilde{x}_0,\tilde{y}_1) = \frac{1}{3} (1,2)$ and denote by $F \cong \mbP (1,2)$ its exceptional divisor.
We see that $\hat{S}$ has a singular point $\hat{\msq}$ of type $A_1$ along $F$ and it is nonsingular along $F \setminus \{\hat{\msq}\}$.
Let $\hat{E}_{\psi}$ and $\hat{\Gamma}$ be the proper transforms of $E_{\psi}$ and $\tilde{\Gamma}$ on $\hat{S}$, respectively.
Then, $\hat{E}_{\psi}$ intersects $F$ at a nonsingular point and $\hat{\Gamma}$ intersects $F$ at $\hat{\msq}$.
Thus, by considering the blowup of $\hat{S}$ at $A_1$ singuar points $\hat{\msq}$ and $\msq_1$, we see that $G (S,\msp_4,\Gamma)$ is of type $A_{4,1}$.

Suppose that $\Gamma$ is as in (2) of Lemma \ref{lem:classifCG10}.
Then, $S$ is cut out by the section
\[
s := x_0 (\lambda_0 x_0 + \lambda_1 x_1 + \lambda w) + \mu (y_0 - \beta x_1^2) + \nu y_1,
\]
where $\lambda_i,\lambda,\mu,\nu \in \mbC$ are general.
Note that $K_{\tilde{S}} = \psi^*K_S$.
We have an isomorphism
\[
E_{\psi} \cong (x_0 = y_0 y_1 + \bar{f}_5 = 0) \subset \mbP (1_{x_0},1_{x_1},2_{y_0},3_{y_1})
\]
and
\[
J_{\psi} =
\begin{pmatrix}
\frac{\prt \bar{f}_5}{\prt x_0} & \frac{\prt \bar{f}_5}{\prt x_1} & y_1 + \frac{\prt \bar{f}_5}{\prt y_0} & y_0 & H \\
\lambda & 0 & 0 & 0 & x_0 (\lambda_0 x_0 + \lambda_1 x_1) + \mu (y_0-\beta x_1^2)
\end{pmatrix}.
\]
Note that $\tilde{\Gamma}$ intersects $E_{\psi}$ at $\msq_0 := (0 \!:\! 1 \!:\! \beta \!:\! 0)$.
We define
\[
\Sigma :=
\left( \frac{\prt \bar{f}_5}{\prt x_1} = y_1 + \frac{\prt \bar{f}_5}{\prt y_0} = y_0 = 0 \right) \cap E_{\psi}.
\]
If $f_5 (x_0,x_1,0,0)$ is not divisible by $x_0$, then $\Sigma = \emptyset$ and thus singular points of $\tilde{S}$ consists of $\msq_1$ and $\msq_2$ that are of type $A_1$ and $A_2$, respectively.
In this case, by considering successive blowups of $\tilde{S}$ at its singular points, we see that $G (S,\msp_4)$ is of type $A_{4,2}$.
Suppose that $f_5 (x_0,x_1,0,0)$ is divisible by $x_0$.
Then, we can write $\bar{f}_5 = y_0 (y_0 a_1 + a_3)$, where $a_i \in \mbC [x_0,x_1]$ and thus $E_{\psi} = E_1 + E_2$, where $E_1 = (x_0 = y_0 = 0)$ and $E_2 = (x_0 = y_1 +  y_0 a_1 + a_3 = 0)$.
Note that $\Sigma = \{\msq_3\}$, where $\msq_3 = (0 \!:\! 1 \!:\! 0 \!:\! - \sigma_0)$ is the intersection point $E_1 \cap E_2$.
We have
\[
J_{\psi} (\msq_3) = 
\begin{pmatrix}
\rho & 0 & 0 & 0 & - \sigma_0 \sigma_1 + \tau \\
\lambda & 0 & 0 & 0 & - \mu \beta
\end{pmatrix},
\]
where $\rho$ is the coefficient of $x_0 x_1^4$ in $f_5 (x_0,x_1,0,0)$ and $\sigma_i,\tau$ are as in Lemma \ref{lem:classifCG10}.
We claim that $\rank J_{\psi} (\msq_3) = 2$.
Since $\beta \ne 0$ and $\lambda,\mu$ are general, $\rank J_{\psi} (\msq_3) < 2$ if and only if $\rho = \tau - \sigma_0 \sigma_1 = 0$.
Note that $\rho = 0$ if and only if $f_5 (x_0,x_1,0,0)$ is divisible by $x_0^2$.
Hence, the case $\rho = \tau - \sigma_0 \sigma_1 = 0$ does not happen and we have $\rank J_{\psi} (\msq_3) = 2$.
It follows that $\tilde{S}$ is nonsingular at $\msq_3$ and thus singular points of $\tilde{S}$ consist of two points $\msq_1$ and $\msq_2$ that are of type $A_1$ and $A_2$.
Note that $\tilde{\Gamma}$ intersects $E_1$ at a point other than $\msq_1$ and does not intersect $E_2$.
This shows that, $G (S,\msp_4)$ is of type $A_{5,2}$.

By the above argument, the type of $G (S,\msp_4)$ is one of $A_{4,1}$, $A_{4,2}$ and $A_{5,2}$.
By Lemma \ref{lem:compselfint}, we have
\[
(\Gamma^2)_S \le -2 - (K_S \cdot \Gamma) + \frac{4}{3} = - \frac{5}{3}.
\]
Finally, we have
\[
(A^2 \cdot S) - 2 (A \cdot \Gamma) + (\Gamma^2)_S \le 3 -2 - \frac{5}{3} < 0.
\]
This shows that $\Gamma$ is not a maximal center.
\end{proof}

\subsection{Curves of degree $1$ on $X' \in \mcG'_{16}$.}

Let $X' = X'_7 \subset \mbP (1,1,2,3,1)$ be a member of $\mcG'_{16}$ and $\Gamma \subset X'$ an irreducible and reduced curve of degree $1$ that passes through $\msp_4$ but does not pass through the other singular points.
We see that $\Gamma$ is a WCI curve of type either $(1,1,6)$ or $(1,2,3)$.
Let $F' = w^2 y z + w f_6 + g_7$ be the defining polynomial of $X'$.
Since $X' \in \mcG'_{16}$, we have $z^2, y^3 \in f_6$ and $z y^2 \in g_7$.
After rescaling $y,z,w$, we assume that the coefficient of $z^2$ and $y^3$ in $f_6$ are both $1$.
Moreover, we assume that there is no monomial divisible by $y^3$ or $z^2$ in $g_7$ after replacing $y$ and $z$.

We claim that $\Gamma$ cannot be of type $(1,1,6)$.
Indeed, if $\Gamma$ of type $(1,1,6)$ passing through $\msp_4$, then $\Gamma = (x_0 = x_1 = h_6 = 0)$, where $h_6$ is a component of
\[
F' (0,0,y,z,w) = w^2 y z + w (z^2 + y^3) + \alpha y^2 z.
\]
Note that $\alpha \ne 0$.
This shows that $F' (0,0,y,z,w)$ is irreducible so that $h_6$ cannot be its component.
This is a contradiction.

In the following, we treat curves of type $(1,2,3)$.

\begin{Lem}
No curve of type $(1,2,3)$ on $X'$ is a maximal center. 
\end{Lem}

\begin{proof}
We can write
\[
\Gamma = (x_0 = y - \alpha w x_1 - \beta x_1^2 = z + \gamma w^2 x_1 + \delta w x_1^2 + \varepsilon x_1^3 = 0)
\]
for some $\alpha,\dots,\varepsilon \in \mbC$.
Let $S \in |\mcI_{\Gamma} (2A)|$ be a general member, which is nonsingular along $\Gamma \setminus \{\msp_4\}$ by \cite[Lemma 2.5]{Okada3}, and set $T := (x_0 = 0)_{X'}$.
We have
\[
\bar{F}'_{\alpha,\beta} := F' (0,x_1,\alpha w x_1 + \beta x_1^2,z,w) 
= \alpha^3 w^4 x_1^3 + \alpha w^3 z x_1 + w z^2 + \cdots,
\]
where the omitted part is a linear combination of monomials $w^3 x_1^4$ and $\{w^i z^j x_1^k \mid i+3j+k=7, i \le 2, j \le 1\}$.
Since $\Gamma \subset X'$, we have
\[
\bar{F}'_{\alpha,\beta} = (z + \gamma w^2 x_1 + \delta w x_1^2 + \varepsilon x_1^3) h_4 (x_1,z,w),
\]
for some $h_4$.
We can write down $h_4$ as $h_4 = w z + \lambda_1 w^3 x_1 + \lambda_2 w^2 x_1^2 + \lambda_3 w x_1^3 + \lambda_4 x_1^4$ for some $\lambda_1,\dots,\lambda_4 \in \mbC$.
We set $\Delta := (x_0 = y - \alpha w x_1 - \beta x_1^2 = h_4 = 0) \subset X'$.
Then $T|_S = \Gamma + \Delta$ and $\deg \Delta = 4/3$.
By comparing the coefficients of $w^5 x_1^2$, $w^4 x_1^3$ and $w^3 z x_1$, we have
\begin{equation} \label{eq:Cdge1G16}
\gamma \lambda_1 = 0, \gamma \lambda_2 + \delta \lambda_1 = \alpha^2, \gamma + \lambda_1 = \alpha.
\end{equation}

Suppose $(\gamma,\lambda_1) \ne (0,0)$.
Note that $\Delta$ intersects $\Gamma$ at $2$ points if $\lambda_1 - \gamma \ne 0$.
Since $\gamma \lambda_1 = 0$ and $(\gamma,\lambda_1) \ne (0,0)$, we have $\lambda_1 - \gamma \ne 0$ and hence $(\Gamma \cdot \Delta)_S \ge 2 > \deg \Delta$.
Thus, $\Gamma$ is not a maximal center if $\Delta$ is irreducible.
Suppose that $\Delta$ is reducible, that is, $\lambda_4 = 0$.
Then $\Delta = \Delta_1 + \Delta_2$, where $\Delta_1 = (x_0 = y - \alpha w x_1 - \beta x_1^2 = z + \lambda_1 w^2 x_1 + \lambda_2 w x_1^2 + \lambda_3 x_1^3 = 0)$ and $\Delta_2 = (x_0 = y - \alpha w x_1 - \beta x_1^2 = w = 0)$ are irreducible and reduced curves of degree $1$ and $1/3$, respectively. 
We see that both $\Delta_1$ and $\Delta_2$ intersect $\Gamma$ at a point other than $\msp_4$, which implies $(\Gamma \cdot \Delta_i) \ge 1 \ge \deg \Delta_i$ for $i = 1,2$.
Thus, $\Gamma$ is not a maximal center.

Suppose $\gamma = \lambda_1 = 0$.
Then, by the equations \eqref{eq:Cdge1G16}, we have $\alpha = 0$.
In this case, we have
\[
\bar{F}'_{0,\beta} = (z-\delta w x_1^2 - \varepsilon x_1^3)(w z + \lambda_2 w^2 x_1^2 + \lambda_3 w x_1^3 + \lambda_4 x_1^4).
\]
By comparing the coefficients of $w^3 x_1^4$, $w^2 z x_1^2$ and $w^2 x_1^5$, we have
\[
\delta \lambda_2 = 0, \lambda_2 = \beta, \varepsilon \lambda_2 + \delta \lambda_3 = 0.
\]
If $\lambda_2 \ne 0$ or $\lambda_2 = 0$ and $\delta \ne 0$, then, by the same argument as above, either $\Delta$ is irreducible and $(\Gamma \cdot \Delta)_S \ge 2 > \deg \Delta$ or $\Delta = \Delta_1 + \Delta_2$ splits as a sum of irreducible and reduced curves of degree $1$ and $1/3$ such that $(\Gamma \cdot \Delta)_S \ge 1 \ge \Delta_i$ for $i = 1,2$.
Thus, $\Gamma$ is not a maximal center.

In the following, we assume $\lambda_2 = \delta = 0$.
Note that $\beta = 0$.
Let $S' \in |\mcI_{\Gamma} (3 A)|$ be a general member, which is nonsingular along $\Gamma \setminus \{\msp_4\}$ by \cite[Lemma 2.5]{Okada3}.
We compute $(\Gamma^2)_{S'}$.
Let $\varphi \colon Y' \to X'$ be the weighted blowup at $\msp_4$ with $\wt (x_0,x_1,y,z) = (1,1,3,3)$ and with exceptional divisor $E$.
We set $\psi = \varphi|_{\tilde{S}'} \colon \tilde{S}' \to S'$ and $E_{\psi} = E|_{\tilde{S}'}$. 
Let $G$ and $H$ be the $\varphi$-weight $= 6$ and $7$ parts of $F (x_0,x_1,y,z,1)$, respectively.
We write $f_6 (x_0,x_1,0,z) = z^2 + z a_3 + a_6$, where $a_i \in \mbC [x_0,x_1]$.
Note that $G = y z + z^2 + z a_3 + a_6$.
The surface $S'$ is cut out by the section
\[
s := d_2 x_0 + e_1 y + \mu (z - \varepsilon x_1^3),
\]
where $d_2 = d_2 (x_0,x_1,y,w)$, $e_1 = e_1 (x_0,x_1,w)$ and $\mu \in \mbC$. 
Note that $K_{\tilde{S}'} = \psi^* K_{S'}$.
Let $\lambda \in \mbC$ be the coefficient of $w^2$ in $d_2$ so that $\lambda x_0$ is the $\varphi$-weight $=1$ part of $s (x_0,x_1,y,z,1)$.
Note that $a_6 = f_6 (x_0,x_1,0,0)$ is not divisible by $x_0$ because otherwise the system of equations $y = f_6 = g_7 = f_6 (x_0,x_1,0,0) = 0$ has a solution $(x_0,x_1,y,z) = (0,1,0,\varepsilon)$ and this is impossible by Condition \ref{addcond}. 
It follows that 
\[
E_{\psi} \cong (y z + z^2 + z a_3 + a_6 = x_0 = 0) \subset \mbP (1,1,3,3)
\]
is irreducible and reduced.
We have
\[
J_{\psi} =
\begin{pmatrix}
\frac{\prt G}{\prt x_0} & \frac{\prt G}{\prt x_1} & \frac{\prt G}{\prt y} & \frac{\prt G}{\prt z} & H \\
\lambda & 0 & 0 & 0 & t
\end{pmatrix},
\]
where $t = t (x_0,x_1,y,z)$ is the $\varphi$-weight $= 2$ part of $s (x_0,x_1,y,z,1)$. 
Since $x_0 \nmid a_6$, we have
\[
\left(\frac{\prt G}{\prt y} = \frac{\prt G}{\prt z} = 0 \right) \cap E_{\psi} = (z = y + a_3 = a_6 = x_0 = 0) = \emptyset.
\]
It follows that the singular points of $\tilde{S}'$ along $E_{\psi}$ consists of two $A_2$ points $\msq_1 := (0 \!:\! 0 \!:\! 1 \!:\! 0)$ and $\msq_2 := (0 \!:\! 0 \!:\! 1 \!:\! -1)$.
Moreover, $\tilde{\Gamma}$ intersects $E_{\psi}$ at $(0 \!:\! 1 \!:\! 0 \!:\! \varepsilon) \ne \msq_1, \msq_2$.
By considering successive blowups at $A_2$ points $\msq_1$ and $\msq_2$, we see that $G (S,\msp_4,\Gamma)$ is of type $A_{5,3}$.
By Lemma \ref{lem:compselfint}, $(\Gamma^2)_S = -2 - 2 \deg \Gamma + 3/2 = -5/2$.
It follows that 
\[
(A^2 \cdot S) -  2(A \cdot \Gamma) + (\Gamma^2)_S = \frac{7}{2} - 2 - \frac{5}{2} < 0,
\]
which implies that $\Gamma$ is not a maximal center. 
\end{proof}

\subsection{Curves of degree $1/2$ on $X' \in \mcG'_{18}$.}

Let $X' = X'_8 \subset \mbP (1,1,2,3,2)$ be a member of $\mcG'_{18}$ with defining polynomial $w^2 x_0 z + w f_6 + g_8 = 0$ and $\Gamma \subset X'$ an irreducible and reduced curve of degree $1/2$ that passes through $\msp_4$ but does not pass through the other singular points.
We see that $\Gamma$ is a WCI curve of type either $(1,1,6)$ or $(1,2,3)$.

\begin{Lem}
No curve of type $(1,1,6)$ on $X'$ is a maximal center.
\end{Lem}

\begin{proof}
Suppose that $\Gamma$ is of type $(1,1,6)$.
We have $\Gamma = (x_0 = x_1 = h_6 = 0)$ for some $h_6 \in \mbC [y,z,w]$.
Since $z^2 \in f_6$, we may write $f_6 (0,0,y,z) = z^2 + \alpha y^3$ and $h_8 (0,0,y,z) = \beta y^4 = 0$ for some $\alpha,\beta \in \mbC$ after replacing $w$.
Hence,
\[
(x_0 = x_1 = 0)_{X'} = (x_0 = x_1 = w (z^2 + \alpha y^3) + \gamma y^4 = 0).
\]
It follows that $\gamma = 0$ and $\Gamma = (x_0 = x_1 = z^2 + \alpha y^3 = 0)$.
Note that $\alpha \ne 0$ since $\Gamma$ is reduced.
Let $S$ and $T$ be general members of the pencil $|\mcI_{\Gamma} (A)|$.
We have $T|_S = \Gamma + \Delta$, where $\Delta = (x_0 = x_1 = w = 0)$ is of degree $1$.
We see that $S$ is nonsingular along $\Gamma \setminus \{\msp_4\}$ by Lemma \ref{lem:qsmcri} and $\Gamma$ intersects $\Delta$ in a nonsingular point.
Thus $(\Gamma \cdot \Delta) \ge 1 > (A \cdot \Delta)$, which shows that $\Gamma$ is not a maximal center.
\end{proof}

\begin{Lem}
No curve of type $(1,2,3)$ on $X'$ is a maximal center.
\end{Lem}

\begin{proof}
Let $\Gamma \subset X'$ be a curve of type $(1,2,3)$ passing through $\msp_4$.
Let $S \in |\mcI_{\Gamma} (2 A)|$ and $T \in |\mcI_{\Gamma} (A)|$ be general members.
We will show that $T|_S = \Gamma + \Delta$, where $\Delta$ is an effective divisor such that $(\Gamma \cdot \Delta_i)_S \ge \deg \Delta_i$ for every irreducible component $\Delta_i$ of $\Delta$.
This shows that $\Gamma$ is not a maximal center.
Note that $S$ is nonsingular along $\Gamma \setminus (\Gamma \cap (x_0 = x_1 = 0)) \cup \{\msp_4\}$ by \cite[Lemma 2.5]{Okada3}.
Since $\Gamma \cap (x_0 = x_1 = 0) = \{\msp_4\}$ (see the descriptions of $\Gamma$ below), $S$ is nonsingular along $\Gamma \setminus \{\msp_4\}$.

Suppose $\Gamma \not\subset (x_0 = 0)$.
Then, after replacing $x_1$ and $y$, we have $\Gamma = (x_1 = y = z + \lambda w x_0 + \mu x_0^3 = 0)$ for some $\lambda,\mu \in \mbC$.
We have
\[
\bar{F}' := F' (x_0,0,0,z) = w^2 x_0 z + w (z^2 + \alpha_6 x_0^6) + \beta_5 z x_0^5 + \beta_8 x_0^8,
\]
where $\alpha_6, \beta_5$ and $\beta_8$ are the coefficients of $w x_0^6$, $z x_0^5$ and $x_0^8$ in $F$. 
Since $\Gamma \subset X'$, we have
\[
\bar{F}' = (z-\lambda w x_0 + \lambda w x_0 + \mu x_0^3)(\gamma w^2 x_0 + w z + \varepsilon w x_0^2 + \zeta z x_0^2 + \eta x_0^5)
\]
for some $\gamma,\dots,\eta \in \mbC$. 
By comparing the coefficients of $w^3 x_0^2$, $w^2 z x_0$, $w^2 x_0^4$ and $z^2 x_0$, 
\[
\lambda \gamma = 0, \gamma + \lambda = 1, \lambda \varepsilon + \mu \gamma = 0, \zeta = 0.
\]
Solving these equations, we have either $\lambda = \mu = \zeta = 0$ and $\gamma = 1$ or $\gamma = \varepsilon = \zeta = 0$ and $\lambda = 1$.
Note that $T|_S = \Gamma + \Delta$, where 
\[
\Delta = (x_1 = y = \gamma w^2 x_0 + w z + \varepsilon w x_0^3 + \eta x_0^5 = 0).
\]
If $\Delta$ is irreducible, then it is reduced and it intersects $\Gamma$ at two points other than $\msp_4$, which implies $(\Gamma \cdot \Delta)_S \ge 2 > \deg \Delta = 5/6$.
Suppose that $\Delta$ is reducible.
Then, $\Delta = \Delta_1 + \Delta_2$, where $\Delta_1 = (x_1 = y = w + \zeta x_0^2 = 0)$ and $\Delta_2 = (x_1 = y = z + \gamma w x_0 + \theta x_0^3 = 0)$ for some $\theta \in \mbC$.
Note that $\Delta_i \ne \Gamma$ for $i = 1,2$ and hence $\Delta$ is reduced.
We see that both $\Delta_1$ and $\Delta_2$ intersect $\Gamma$ at a point other than $\msp_4$ so that $(\Gamma \cdot \Delta_i)_S \ge 1 > \deg \Delta_i$ for $i = 1,2$.

Suppose $\Gamma \subset (x_0 = 0)$.
Then, after replacing $y$, we have $\Gamma = (x_0 = y = z - \lambda w x_1 - \mu x_1^3 = 0)$ for some $\lambda,\mu \in \mbC$.
We have
\[
\bar{F} := F (0,x_1,0,z,w) = w (z^2 + \alpha_3 z x_1^3 + \alpha_6 x_1^6) + \beta_5 z x_1^5 + \beta_8 x_1^8,
\]
where $\alpha_5,\alpha_8,\beta_5$ and $\beta_8$ are the coefficients of $w z x_1^3$, $w x_1^6$, $z x_1^5$ and $x_1^8$, respectively.
Since $\Gamma \subset X'$, we have
\[
\bar{F} = (z + \lambda w x_1 + \mu x_1^3)(w z + \gamma w x_1^3 + \delta z x_1^2 + \varepsilon x_1^5),
\]
for some $\gamma,\dots,\varepsilon \in \mbC$.
By comparing the coefficients of $w^2 z x_1$ and $z^2 x_1$, we have $\lambda = \delta = 0$.
Moreover, by comparing the other terms, we have $\alpha_3 = \mu + \gamma$, $\alpha_6 = \mu \gamma$, $\beta_5 = \varepsilon$ and $\beta_8 = \mu \varepsilon$.
If $\gamma = \mu$, then $x_0 = f_6 = g_8 = \prt f_6/\prt z = 0$ has a solution $(x_0,x_1,y,z) = (0,1,0,-\mu)$.
If $\varepsilon = 0$, then $\beta_5 = \beta_8 = 0$ and hence $x_0 = b_5 = b_8 = 0$ has a solution $(x_0,x_1,y) = (0,1,0)$.
These are impossible by Condition \ref{addcond} and we have $\mu \ne \gamma$ and $\varepsilon \ne 0$. 
We see that $\Delta = (x_0 = y = w z + \gamma w x_1^3 + \varepsilon x_1^5 = 0)$ is irreducible and reduced since $\varepsilon \ne 0$, and $\Delta$ intersects $\Gamma$ at a point other than $\msp_4$ since $\mu \ne \gamma$.
It follows that $(\Gamma \cdot \Delta)_S \ge 1 = \deg \Delta = 5/6$.
This completes the proof.
\end{proof}

\subsection{Curves of degree $1/3$ on $X' \in \mcG'_{21}$.} \label{sec:CG21}

Let $X' = X'_9 \subset \mbP (1,1,2,3,3)$ be a member of $\mcG'_{21}$ and $\Gamma \subset X'$ an irreducible and reduced curve of degree $1/3$ that passes through $\msp_4$ but does not pass through the other singular points. 
Let $F' = w^2 x_0 y + w f_6 + g_9$ be the defining polynomial of $X'$.
We have $z^2, y^3 \in f_6$.
After replacing $w$, we assume that $z^3 \notin g_9$.
Then, we have $z y^3 \in g_9$.

We see that $\Gamma$ is a WCI curve of type either $(1,1,6)$ or $(1,2,3)$.
We claim that $\Gamma$ cannot be of type $(1,1,6)$.
Indeed, if $\Gamma$ is of type $(1,1,6)$, then $\Gamma = (x_0 = x_1 = h_6 = 0)$ for some $h_6 \in \mbC [y,z,w]$ of degree $6$.
On the other hand, we have
\[
F (0,0,y,z,w) = w (\alpha z^2 + \beta y^3) + \gamma z y^3,
\]
where $\alpha,\beta,\gamma \in \mbC$ are non-zro.
Hence, $F (0,0,y,z,w)$ is irreducible and $X'$ cannot contain $\Gamma$.

\begin{Lem}
No curve of type $(1,2,3)$ on $X'$ that passes through $\msp_4$ is a maximal center.
\end{Lem}

\begin{proof}
Let $\Gamma$ be a curve of type $(1,2,3)$ on $X'$ passing through $\msp_4$.
Let $S \in |\mcI_{\Gamma} (2 A)|$ be a general member and $T \in |\mcI_{\Gamma} (A)|$.
We will show that $T|_S = \Gamma + \Delta$, where $\Delta$ is an effective divisor on $S$ such that, for each component $\Delta_i$ of $\Delta$, there exists an effective divisor $\Xi_i$ on $S$ such that $(\Gamma \cdot \Xi_i)_S \ge \deg \Xi_i$ and $(\Xi_i \cdot \Delta_j)_S \ge 0$ for $j \ne i$.
Note that $S$ is nonsingular along $\Gamma \setminus \{\msp_4\}$ by \cite[Lemma 2.5]{Okada3} since $\Gamma \cap (x_0 = x_1 = 0) = \{\msp_4\}$.

Suppose $\Gamma \subset (x_0 = 0)$.
Then, after replacing $z$, we may assume $\Gamma = (x_0 = y - \lambda x_1^2 = z = 0)$ for some $\lambda \in \mbC$.
We have
\[
F' (0,x_1,\lambda x_1^2,z,w) = w (z^2 + \alpha_3 z x_1^3 + \alpha_6 x_1^6) + \beta_6 z x_1^6 + \beta_9 x_1^9.
\]
Since $\Gamma \subset X'$, we have $\alpha_6 = \beta_9 = 0$ and we have 
\[
\Delta = (x_0 = y - \lambda x_1^2 = w (z + \alpha_3 x_1^3) + \beta_6 x_1^6 = 0).
\]
If $\alpha_3 = \lambda$, then $x_0 = f_6 = g_9 = \prt f_6/\prt z = 0$ has a solution $(x_0,x_1,y,z) = (0,1,\lambda,0)$.
This is impossible by Condition \ref{addcond} and thus $\alpha_3 \ne \lambda$.
Suppose $\beta_6 \ne 0$.
Then $\Delta$ is irreducible and reduced and it intersects $\Gamma$ at a nonsingular point other than $\msp_4$ since $\alpha_3 \ne \lambda$.
Thus $(\Gamma \cdot \Delta)_S \ge 1 > \deg \Delta = 2/3$.
Suppose $\beta_6 = 0$.
Then $\Delta = \Delta_1 + \Delta_2$, where $\Delta_1 = (x_0 = y -\lambda x_1^2 = z + \alpha_3 x_1^3 = 0)$ and $\Delta_2 = (x_0 = y = w = 0)$.
We have $(\Gamma \cdot \Delta_2)_S = (\Delta_1 \cdot \Delta_2)_S = 1$.
Set $\Xi_2 = \Delta_2$.
Then we have $(\Gamma \cdot \Xi_2)_S \ge 1 > \deg \Xi_2 = 1/3$ and $(\Xi_2 \cdot \Delta_1) \ge 0$. 
By taking the intersection number of $T|_S = \Gamma + \Delta_1 + \Delta_2$ with $\Delta_2$, we have $(\Delta_2^2)_S = - 5/3$ since $(T|_S \cdot \Delta_2)_S = \deg \Delta_2 = 1/3$.
Let $\varepsilon$ be a rational number such that $1/2 \le \varepsilon \le 3/5$ and set $\Xi_1 = \Delta_1 + \varepsilon \Delta_2$.
Then,
\[
(\Gamma \cdot \Xi_1)_S - \deg \Xi_1 = (\Gamma \cdot \Delta_1) + \frac{2}{3} \varepsilon - \frac{1}{3} \ge 0,
\]
and
\[
(\Xi_1 \cdot \Delta_2)_S = 1 - \frac{5}{3} \varepsilon \ge 0,
\]
as desired.

Suppose $\Gamma \not\subset (x_0 = 0)$.
Then, after replacing $x_1$ and $z$, we may assume $\Gamma = (x_1 = y + \lambda x_0^2 = z = 0)$ for some $\lambda \in \mbC$.
We have
\[
F' (x_0,0,\lambda x_1^2,z,w) = \lambda w^2 x_0^3 + w (z^2 + \alpha_3 z x_0^3 + \alpha_6 x_0^6) + \beta_6 z x_0^6 + \beta_9 x_0^9.
\]
Since $\Gamma \subset X'$, we have $\lambda = \alpha_6 = \beta_9 = 0$ and
\[
\Delta = (x_1 = y = w (z + \alpha_3) + \beta_6 x_0^6 = 0).
\]
If $\alpha_3 = 0$, then $y = f_6 = g_9 = \prt f_6/\prt z = 0$ has a solution $(x_0,x_1,y,z) = (1,0,0,0)$.
This is impossible by Condition \ref{addcond} and thus $\alpha_3 \ne 0$.
Since $\Gamma \subset X'$, $F' (x_0,0,y,0,w)$ is divisible by $y - \eta x_0^2$, which happens if and only if $\eta = 0$ and both $f_6 (x_0,0,y,0)$ and $g_9 (x_0,0,y,w)$ are divisible by $y$.
Thus, $\Gamma = (x_1 = y = z = 0)$.
Suppose $\beta_6 \ne 0$.
Then, $\Delta$ is irreducible and reduced and $(\Gamma \cdot \Delta)_S \ge 1$ since $\alpha_3 \ne 0$.
Suppose $\beta_6 = 0$.
Then, $\Delta = \Delta_1 + \Delta_2$, where $\Delta_1 = (x_1 = y = z + \alpha_3 x_0^3 = 0)$ and $\Delta_2 = (x_1 = y = w = 0)$.
We have $(\Gamma \cdot \Delta_2)_S = (\Gamma \cdot \Delta_2)_S = 1$.
As in the avobe argument, $\Xi_1 = \Delta_2 + \varepsilon \Delta_2$, where $1/2 \le \varepsilon \le 3/5$, and $\Xi_2 = \Delta_2$ are the desired effective divisors. 
This completes the proof.
\end{proof}

The following is the conclusion of this section.

\begin{Thm} \label{thm:curve}
Let $X'$ be a member of $\mcG'_i$ with $i \in I^*_{cA/n} \cup I_{cD/3}$.
Then, no curve on $X'$ is a maximal center.
\end{Thm}

\begin{proof}
This follows from Lemma \ref{exclmostcurves} and the conclusions of Sections \ref{sec:CG6-1}--\ref{sec:CG21}.
\end{proof}

\section{The big table} \label{sec:bigtable}

\newlength{\myheight}
\setlength{\myheight}{0.65cm}

\begin{table}[h]
\begin{center}
\caption{Families $\mcG_i$}
\label{table:familyG}
\begin{tabular}{cccc}
\hline
\parbox[c][\myheight][c]{0cm}{} No. & $X_{d_1,d_2} \subset \mbP (a_0,\dots,a_5)$ & No. & $X_{d_1,d_2} \subset \mbP (a_0,\dots,a_5)$ \\
\hline
\parbox[c][\myheight][c]{0cm}{} 6 & $X_{4,5} \subset \mbP (1,1,1,2,2,3)$ & 33 & $X_{9,10} \subset \mbP (1,1,3,4,5,6)$ \\
\parbox[c][\myheight][c]{0cm}{} 7 & $X_{4,6} \subset \mbP (1,1,1,2,3,3)$ & 36 & $X_{8,12} \subset \mbP (1,1,3,4,5,7)$  \\ 
\parbox[c][\myheight][c]{0cm}{} 9 & $X_{5,6} \subset \mbP (1,1,1,2,3,4)$ & 38 & $X_{9,12} \subset \mbP (1,2,3,4,5,7)$ \\
\parbox[c][\myheight][c]{0cm}{} 10 & $X_{5,6} \subset \mbP (1,1,2,2,3,3)$ & 44 & $X_{10,12} \subset \mbP (1,2,3,5,5,7)$ \\
\parbox[c][\myheight][c]{0cm}{} 16 & $X_{6,7} \subset \mbP (1,1,2,3,3,4)$ & 48 & $X_{11,12} \subset \mbP (1,1,4,5,6,7)$ \\
\parbox[c][\myheight][c]{0cm}{} 18 & $X_{6,8} \subset \mbP (1,1,2,3,3,5)$ & 52 & $X_{10,15} \subset \mbP (1,2,3,5,7,8)$ \\
\parbox[c][\myheight][c]{0cm}{} 21 & $X_{6,9} \subset \mbP (1,1,2,3,4,5)$ & 57 & $X_{12,14} \subset \mbP (1,2,3,5,7,9)$ \\
\parbox[c][\myheight][c]{0cm}{} 22 & $X_{7,8} \subset \mbP (1,1,2,3,4,5)$ & 61 & $X_{12,15} \subset \mbP (1,1,4,5,6,11)$ \\
\parbox[c][\myheight][c]{0cm}{} 26 & $X_{8,9} \subset \mbP (1,1,3,4,4,5)$ & 62 & $X_{12,15} \subset \mbP (1,3,4,5,6,9)$ \\
\parbox[c][\myheight][c]{0cm}{} 28 & $X_{8,10} \subset \mbP (1,1,2,3,5,7)$ & 63 & $X_{12,15} \subset \mbP (1,3,4,5,7,8)$
\end{tabular}
\end{center}
\end{table}

The list of the families $\mcG_i$ with $i \in I^*_{cA/n} \cup I_{cD/3}$ is given in Table \ref{table:familyG} and we list the families $\mcG'_i$ with $i \in I^*_{cA/n} \cup I_{cD/3}$ below.
In each family $\mcG'_i$, a standard defining equation is described.
The monomials right after the equation is a condition imposed on the family (see Remark \ref{remcond}).
The table of each family is divided into two parts: terminal quotient parts and $\msp_4$ parts.

We first explain terminal quotient parts.
The first column indicates the number and type of the singular points.
The second column indicates how to exclude them if it is non-empty.
If a set of polynomials and a divisor of the form $b B + e E$ are given, then it is excluded in Proposition \ref{excltqsing}.
If the inequality on $B^3$ is given and no other information is given in the second column, then the point is excluded in one of Propositions \ref{excltqsingNo21_2}--\ref{excltqsingNo33_3}.
The third column indicates the existence of birational involution that is a Sarkisov link centered at the corresponding point (see Section \ref{sec:birinvtqpt}).

We next explain $\msp_4$ parts of the table.
We include in this table the equation of the singularity, blowup weights of divisorial extraction.
For family No.~$i$ with $i \in I^*_{cA/n}$, we also indicates what happens for each divisorial extraction.
The mark ``none" indicates that the corresponding divisorial extraction is not a maximal extraction (see Section \ref{sec:exclcA})
The mark ``$X' \ratmap X \ni \frac{1}{r} (\alpha,\beta,\gamma)$" (resp.\ ``B.I.") indicates that there is a Sarkisov link starting with the corresponding divisorial extraction that is a link to $X$ ending with the Kawamata blowup centered at a $\frac{1}{r} (\alpha,\beta,\gamma)$ point of $X$ (resp.\ that is a birational involution) (see Section \ref{sec:Slink} and \ref{sec:birinvcA}, respectively).

\begin{center}
\begin{flushleft}
No. 6: $X_5 \subset \mbP (1,1,1,2,1)$, $(A^3) = 5/2$. \nopagebreak \\
Eq: $w^2 x_0 y + w f_4 + g_5$, $y^2 \in f_4$.
\end{flushleft} \nopagebreak
\begin{tabular}{|p{105pt}|p{140pt}|p{110pt}|}
\hline
\parbox[c][\myheight][c]{0cm}{} $\msp_3 = \frac{1}{2} (1,1,1)$ & & Q.I. \\
\hline
\parbox[c][\myheight][c]{0cm}{} $\msp_4 = cA$ & $x_0 y + h_4 (x_1,x_2)$ & $\wt = (a,b,1,1)$ \\
\hline
\end{tabular}
\begin{tabular}{|p{67pt}|p{104.5pt}||p{67pt}|p{104.5pt}|}
\hline
\parbox[c][\myheight][c]{0cm}{} $(1,3), (3,1)$ & $X' \ratmap X \ni \frac{1}{2} (1,1,1)$ & $(2,2)$ & $X' \ratmap X \ni \frac{1}{3} (1,1,2)$ \\
\hline
\end{tabular}
\end{center}

\begin{center}
\begin{flushleft}
No. 7: $X_6 \subset \mbP (1,1,1,2,2)$, $(A^3) = 3/2$. \nopagebreak \\
Eq: $w^2 x_0 x_1 + w f_4 + g_6$.
\end{flushleft} \nopagebreak
\begin{tabular}{|p{105pt}|p{140pt}|p{110pt}|}
\hline
\parbox[c][\myheight][c]{0cm}{} $\msp_3 \msp_4 = \frac{1}{2} (1,1,1)$ & & Q.I. \\
\hline
\parbox[c][\myheight][c]{0cm}{} $\msp_4 = cA/2$ & $x_0 x_1 + h_4 (x_1,y)/\mbZ_2 (1,1,1,0)$ & $\wt = \frac{1}{2} (a,b,1,2)$ \\
\hline
\end{tabular}
\begin{tabular}{|p{67pt}|p{104.5pt}||p{67pt}|p{104.5pt}|}
\hline
\parbox[c][\myheight][c]{0cm}{} $(1,3), (3,1)$ & $X' \ratmap X \ni \frac{1}{3} (1,1,2)$ & & \\
\hline
\end{tabular}
\end{center}

\begin{center}
\begin{flushleft}
No. 9: $X_6 \subset \mbP (1,1,1,3,1)$, $(A^3) = 2$. \nopagebreak \\ 
Eq: $w^2 x_0 y + w f_5 + g_6$.
\end{flushleft} \nopagebreak
\begin{tabular}{|p{105pt}|p{140pt}|p{110pt}|}
\hline
\parbox[c][\myheight][c]{0cm}{} $\msp_4 = cA$ & $x_0 y + h_5 (x_1,x_2)$ & $\wt = (a,b,1,1)$ \\
\hline
\end{tabular}
\begin{tabular}{|p{67pt}|p{104.5pt}||p{67pt}|p{104.5pt}|}
\hline
\parbox[c][\myheight][c]{0cm}{} $(1,4), (4,1)$ & $X' \ratmap X \ni \frac{1}{2} (1,1,1)$ & $(2,3)$, $(3,2)$ & $X' \ratmap X \ni \frac{1}{4} (1,1,3)$ \\
\hline
\end{tabular}
\end{center}

\begin{center}
\begin{flushleft}
No. 10: $X_6 \subset \mbP (1,1,2,2,1)$, $(A^3) = 3/2$. \nopagebreak \\
Eq: $w^2 y_0 y_1 + w f_5 + g_6$, $y_0^3 \in g_6$ and $y_1^3 \in g_6$.
\end{flushleft} \nopagebreak
\begin{tabular}{|p{105pt}|p{140pt}|p{110pt}|}
\hline
\parbox[c][\myheight][c]{0cm}{} $\msp_2 \msp_3 = 3 \times \frac{1}{2} (1,1,1)$ & & Q.I. \\
\hline
\parbox[c][\myheight][c]{0cm}{} $\msp_4 = cA$ & $y_0 y_1 + h_5 (x_0,x_1)$ & $\wt = (a,b,1,1)$ \\
\hline
\end{tabular}
\begin{tabular}{|p{67pt}|p{104.5pt}||p{67pt}|p{104.5pt}|}
\hline
\parbox[c][\myheight][c]{0cm}{} $(1,4), (4,1)$ & B.I. & $(2,3), (3,2)$ & $X' \ratmap X \ni \frac{1}{3} (1,1,2)$ \\
\hline
\end{tabular}
\end{center}

\begin{center}
\begin{flushleft}
No. 16: $X_7 \subset \mbP (1,1,2,3,1)$, $(A^3) = 7/6$. \nopagebreak \\
Eq: $w^2 y z + w f_6 + g_7$, $z^2, y^3 \in f_6, y^2 z \in g_7$.
\end{flushleft} \nopagebreak
\begin{tabular}{|p{105pt}|p{140pt}|p{110pt}|}
\hline
\parbox[c][\myheight][c]{0cm}{} $\msp_2 = \frac{1}{2} (1,1,1)$ &  & Q.I. \\
\hline
\parbox[c][\myheight][c]{0cm}{} $\msp_3 = \frac{1}{3} (1,1,2)$ &  & Q.I. \\
\hline
\parbox[c][\myheight][c]{0cm}{} $\msp_4 = cA$ & $y z + h_6 (x_0,x_1)$ & $\wt = (a,b,1,1)$ \\
\hline
\end{tabular}
\begin{tabular}{|p{67pt}|p{104.5pt}||p{67pt}|p{104.5pt}|}
\hline
\parbox[c][\myheight][c]{0cm}{} $(1,5), (5,1)$ & none & $(2,4), (4,2)$ & $X' \ratmap X \ni \frac{1}{3} (1,1,2)$ \\
\hline
\parbox[c][\myheight][c]{0cm}{} $(3,3)$ & $X' \ratmap X \ni \frac{1}{4} (1,1,3)$  & &\\
\hline
\end{tabular}
\end{center}

\begin{center}
\begin{flushleft}
No. 18: $X_8 \subset \mbP (1,1,2,3,2)$, $(A^3) = 2/3$. \nopagebreak \\
Eq: $w^2 x_0 z + w f_6 + g_8$, $z^2 \in f_6$.
\end{flushleft} \nopagebreak
\begin{tabular}{|p{105pt}|p{140pt}|p{110pt}|}
\hline
\parbox[c][\myheight][c]{0cm}{} $\msp_2 \msp_4 = \frac{1}{2} (1,1,1)$ & & E.I. $*z^2 y$ \\
\hline
\parbox[c][\myheight][c]{0cm}{} $\msp_3 = \frac{1}{3} (1,1,2)$ & & Q.I. \\
\hline
\parbox[c][\myheight][c]{0cm}{} $\msp_4 = cA/2$ & $x_0 z + h_6 (x_1,y)/ \mbZ_2 (1,1,1,0)$ & $\wt = \frac{1}{2} (a,b,1,2)$ \\
\hline
\end{tabular}
\begin{tabular}{|p{67pt}|p{104.5pt}||p{67pt}|p{104.5pt}|}
\hline
\parbox[c][\myheight][c]{0cm}{} $(1,5), (5,1)$ & $X' \ratmap X \ni \frac{1}{3} (1,1,1)$ & $(3,3)$ & $X' \ratmap X \ni \frac{1}{5} (1,2,3)$ \\
\hline
\end{tabular}
\end{center}

\begin{center}
\begin{flushleft}
No. 21: $X_9 \subset \mbP (1,1,2,3,3)$, $(A^3) = 1/2$. \nopagebreak \\
Eq: $w^2 x_0 y + w f_6 + g_9$, $z^2, y^3 \in f_6$, and $z y^3 \in g_9$ if $z^3 \notin g_9$.
\end{flushleft} \nopagebreak
\begin{tabular}{|p{105pt}|p{140pt}|p{110pt}|}
\hline
\parbox[c][\myheight][c]{0cm}{} $\msp_2 = \frac{1}{2} (1,1,1)$ & $B^3 = 0$  & \\
\hline
\parbox[c][\myheight][c]{0cm}{} $\msp_3 \msp_4 = \frac{1}{3} (1,1,2)$ & & Q.I. \\
\hline
\parbox[c][\myheight][c]{0cm}{} $\msp_4 = cA/3$ & $x_0 y + h_6 (x_1,z)/\mbZ_3 (1,2,1,0)$ & $\wt = \frac{1}{3} (a,b,1,3)$ \\
\hline
\end{tabular}
\begin{tabular}{|p{67pt}|p{104.5pt}||p{67pt}|p{104.5pt}|}
\hline
\parbox[c][\myheight][c]{0cm}{} $(1,5)$ & $X' \ratmap X \ni \frac{1}{4} (1,1,3)$ & $(4,2)$ & $X' \ratmap X \ni \frac{1}{5} (1,2,3)$ \\
\hline
\end{tabular}
\end{center}

\begin{center}
\begin{flushleft}
No. 22: $X_8 \subset \mbP (1,1,2,4,1)$, $(A^3) = 1$. \nopagebreak \\
Eq: $w^2 y z + w f_7 + g_8$, $y^4 \in g_8$.
\end{flushleft} \nopagebreak
\begin{tabular}{|p{105pt}|p{140pt}|p{110pt}|}
\hline
\parbox[c][\myheight][c]{0cm}{} $\msp_2 \msp_3 = 2 \times \frac{1}{2} (1,1,1)$ & & Q.I. \\
\hline
\parbox[c][\myheight][c]{0cm}{} $\msp_4 = cA$ & $y z + h_7 (x_0,x_1)$ & $\wt = (a,b,1,1)$ \\
\hline
\end{tabular}
\nopagebreak
\begin{tabular}{|p{67pt}|p{104.5pt}||p{67pt}|p{104.5pt}|}
\hline
\parbox[c][\myheight][c]{0cm}{} $(1,6), (6,1)$ & none & $(2,5), (5,2)$ & $X' \ratmap X \ni \frac{1}{3} (1,1,2)$ \\
\hline
\parbox[c][\myheight][c]{0cm}{} $(3,4), (4,3)$ & $X' \ratmap X \ni \frac{1}{5} (1,1,4)$ & & \\
\hline
\end{tabular}
\end{center}

\begin{center}
\begin{flushleft}
No. 26: $X_9 \subset \mbP (1,1,3,4,1)$, $(A^3) = 3/4$. \nopagebreak \\
Eq: $w^2 y z + w f_8 + g_9$, $z^2 \in f_8$.
\end{flushleft} \nopagebreak
\begin{tabular}{|p{105pt}|p{140pt}|p{110pt}|}
\hline
\parbox[c][\myheight][c]{0cm}{} $\msp_3 = \frac{1}{4} (1,1,3)$ & & Q.I. \\
\hline
\parbox[c][\myheight][c]{0cm}{} $\msp_4 = cA$ & $y z + h_8 (x_0,x_1)$ & $\wt = (a,b,1,1)$ \\
\hline
\end{tabular}
\nopagebreak
\begin{tabular}{|p{67pt}|p{104.5pt}||p{67pt}|p{104.5pt}|}
\hline
\parbox[c][\myheight][c]{0cm}{} $(1,7), (7,1)$ & none & $(2,6), (6,2)$ & B.I. \\
\hline
\parbox[c][\myheight][c]{0cm}{} $(3,5), (5,3)$ & $X' \ratmap X \ni \frac{1}{4} (1,1,3)$ & $(4,4)$ & $X' \ratmap X \ni \frac{1}{5} (1,1,4)$ \\
\hline
\end{tabular}
\end{center}

\begin{center}
\begin{flushleft}
No. 28: $X_{10} \subset \mbP (1,1,2,5,2)$, $(A^3) = 1/2$.\nopagebreak \\
Eq: $w^2 x_0 z + w f_8 + g_{10}$.
\end{flushleft} \nopagebreak
\begin{tabular}{|p{105pt}|p{140pt}|p{110pt}|}
\hline
\parbox[c][\myheight][c]{0cm}{} $\msp_2 \msp_4 = \frac{1}{2} (1,1,1)$ & $B^3 = 0$, $\{x_0,x_1,w'\}$, $B$ & \\
\hline
\parbox[c][\myheight][c]{0cm}{} $\msp_4 = cA/2$ & $x_0 z + h_8 (x_1,y)/ \mbZ_2 (1,1,1,0)$ & $\wt = \frac{1}{2} (a,b,1,2)$ \\
\hline
\end{tabular}
\nopagebreak
\begin{tabular}{|p{67pt}|p{104.5pt}||p{67pt}|p{104.5pt}|}
\hline
\parbox[c][\myheight][c]{0cm}{} $(1,7), (7,1)$ & $X' \ratmap X \ni \frac{1}{3} (1,1,2)$ & $(3,5), (5,3)$ & $X' \ratmap X \ni \frac{1}{7} (1,2,5)$ \\
\hline
\end{tabular}
\end{center}

\begin{center}
\begin{flushleft}
No. 33: $X_{10} \subset \mbP (1,1,3,5,1)$, $(A^3) = 2/3$. \nopagebreak \\
Eq: $w^2 y z + w f_9 + g_{10}, y^3 \in f_9$.
\end{flushleft} \nopagebreak
\begin{tabular}{|p{105pt}|p{140pt}|p{110pt}|}
\hline
\parbox[c][\myheight][c]{0cm}{} $\msp_2 = \frac{1}{3} (1,1,2)$ & $B^3 > 0$ & \\
\hline
\parbox[c][\myheight][c]{0cm}{} $\msp_4 = cA$ & $y z + h_9 (x_0,x_1)$ & $\wt = (a,b,1,1)$ \\
\hline
\end{tabular}
\begin{tabular}{|p{67pt}|p{104.5pt}||p{67pt}|p{104.5pt}|}
\hline
\parbox[c][\myheight][c]{0cm}{} $(1,8), (8,1)$ & none & $(2,7)$, $(7,2)$ & B.I. \\
\hline
\parbox[c][\myheight][c]{0cm}{} $(3,6), (6,3)$ & $X' \ratmap X \ni \frac{1}{4} (1,1,3)$ & $(4,5), (5,4)$ & $X' \ratmap X \ni \frac{1}{6} (1,1,5)$ \\
\hline
\end{tabular}
\end{center}

\begin{center}
\begin{flushleft}
No. 36: $X_{12} \subset \mbP (1,1,3,4,4)$, $(A^3) = 1/4$. \nopagebreak \\
Eq: $w^2 x_0 y + w f_8 + g_{12}$.
\end{flushleft} \nopagebreak
\begin{tabular}{|p{105pt}|p{140pt}|p{110pt}|}
\hline
\parbox[c][\myheight][c]{0cm}{} $\msp_3 \msp_4 = \frac{1}{4} (1,1,3)$ & & Q.I. \\
\hline
\parbox[c][\myheight][c]{0cm}{} $\msp_4 = cA/4$ & $x_0 y + h_8 (x_1,z)/\mbZ_4 (1,3,1,0)$ & $\wt = \frac{1}{4} (a,b,1,4)$ \\
\hline
\end{tabular}
\nopagebreak
\begin{tabular}{|p{67pt}|p{104.5pt}||p{67pt}|p{104.5pt}|}
\hline
\parbox[c][\myheight][c]{0cm}{} $(1,7)$ & $X' \ratmap X \ni \frac{1}{5} (1,1,4)$ & $(5,3)$ & $X' \ratmap X \ni \frac{1}{7} (1,3,4)$ \\
\hline
\end{tabular}
\end{center}

\begin{center}
\begin{flushleft}
No. 38: $X_{12} \subset \mbP (1,2,3,4,3)$, $(A^3) = 1/6$. \nopagebreak \\
Eq: $w^2 y t + w f_9 + g_{12}$, $y^6 \in g_{12}$.
\end{flushleft} \nopagebreak
\begin{tabular}{|p{105pt}|p{140pt}|p{110pt}|}
\hline
\parbox[c][\myheight][c]{0cm}{} $\msp_1 \msp_3 = 3 \times \frac{1}{2} (1,1,1)$ & $B^3 < 0$, $\{x,z,w\}$, $3B+E$ & \\
\hline
\parbox[c][\myheight][c]{0cm}{} $\msp_2 \msp_4 = \frac{1}{3} (1,1,2)$ & $B^3 = 0$, $\{x,y,w'\}$, $B$ & \\
\hline
\parbox[c][\myheight][c]{0cm}{} $\msp_4 = cA/3$ & $y t + h_9 (x,z)/\mbZ_3 (2,1,1,0)$ & $\wt = \frac{1}{3} (a,b,1,3)$ \\
\hline
\end{tabular}
\nopagebreak
\begin{tabular}{|p{67pt}|p{104.5pt}||p{67pt}|p{104.5pt}|}
\hline
\parbox[c][\myheight][c]{0cm}{} $(2,7)$ & $X' \ratmap X \ni \frac{1}{5} (1,2,3)$ & $(5,4)$ & $X' \ratmap X \ni \frac{1}{7} (1,3,4)$ \\
\hline
\parbox[c][\myheight][c]{0cm}{} $(8,1)$ & B.I. & & \\
\hline
\end{tabular}
\end{center}

\begin{center}
\begin{flushleft}
No. 44: $X_{12} \subset \mbP (1,2,3,5,2)$, $(A^3) = 1/5$. \nopagebreak \\
Eq: $w^2 z t + w f_{10} + g_{12}$, $t^2 \in f_{10}$.
\end{flushleft} \nopagebreak
\begin{tabular}{|p{105pt}|p{140pt}|p{110pt}|}
\hline
\parbox[c][\myheight][c]{0cm}{} $\msp_1 \msp_4 = \frac{1}{2} (1,1,1)$ & $B^3 < 0$ & \\
\hline
\parbox[c][\myheight][c]{0cm}{} $\msp_3 = \frac{1}{5} (1,2,3)$ & & Q.I. \\
\hline
\parbox[c][\myheight][c]{0cm}{} $\msp_4 = cA/2$ & $z t + h_{10} (x,y)/ \mbZ_2 (1,1,1,0)$ & $\wt = \frac{1}{2} (a,b,1,2)$ \\
\hline
\end{tabular}
\nopagebreak
\begin{tabular}{|p{67pt}|p{104.5pt}||p{67pt}|p{104.5pt}|}
\hline
\parbox[c][\myheight][c]{0cm}{} $(1,9), (9,1)$ & none & $(3,7), (7,3)$ & $X' \ratmap X \ni \frac{1}{5} (1,2,3)$ \\
\hline
\parbox[c][\myheight][c]{0cm}{} $(5,5)$ & $X' \ratmap X \ni \frac{1}{7} (1,2,5)$ & & \\
\hline
\end{tabular}
\end{center}

\begin{center}
\begin{flushleft}
No. 48: $X_{12} \subset \mbP (1,1,4,6,1)$, $(A^3) = 1/2$. \nopagebreak \\
Eq: $w^2 y z + w f_{11} + g_{12}$.
\end{flushleft} \nopagebreak
\begin{tabular}{|p{105pt}|p{140pt}|p{110pt}|}
\hline
\parbox[c][\myheight][c]{0cm}{} $\msp_2 \msp_3 = \frac{1}{2} (1,1,1)$ & $B^3 = 0$, $\{x_0,x_1,w\}$, $B$ & \\
\hline
\parbox[c][\myheight][c]{0cm}{} $\msp_4 = cA$ & $y z + h_{11} (x_0,x_1)$ & $\wt = (a,b,1,1)$ \\
\hline
\end{tabular}
\nopagebreak
\begin{tabular}{|p{67pt}|p{104.5pt}||p{67pt}|p{104.5pt}|}
\hline
\parbox[c][\myheight][c]{0cm}{} $(1,10), (10,1)$ & none & $(2,9), (9,2)$ & none \\
\hline
\parbox[c][\myheight][c]{0cm}{} $(3,8), (8,3)$ & B.I. & $(4,7), (7,4)$ & $X' \ratmap X \ni \frac{1}{3} (1,1,2)$ \\
\hline
\parbox[c][\myheight][c]{0cm}{} $(5,6), (6,5)$ & $X' \ratmap X \ni \frac{1}{7} (1,1,6) $ & & \\
\hline
\end{tabular}
\end{center}

\begin{center}
\begin{flushleft}
No. 52: $X_{15} \subset \mbP (1,2,3,5,5)$, $(A^3) = 1/10$. \nopagebreak \\
Eq: $w^2 y z + w f_{10} + g_{15}$, $t^2 \in f_{10}$.
\end{flushleft} \nopagebreak
\begin{tabular}{|p{105pt}|p{140pt}|p{110pt}|}
\hline
\parbox[c][\myheight][c]{0cm}{} $\msp_1 = \frac{1}{2} (1,1,1)$ & $B^3 < 0$, $\{x,z,t,w\}$, $5B+2E$ & \\
\hline
\parbox[c][\myheight][c]{0cm}{} $\msp_3 \msp_4 = \frac{1}{5} (1,2,3)$ & & Q.I. \\
\hline
\parbox[c][\myheight][c]{0cm}{} $\msp_4 = cA/5$ & $y z + h_{10} (x,t)/ \mbZ_5 (2,3,1,0)$ & $\wt = \frac{1}{5} (a,b,1,5)$ \\
\hline
\end{tabular}
\nopagebreak
\begin{tabular}{|p{67pt}|p{104.5pt}||p{67pt}|p{104.5pt}|}
\hline
\parbox[c][\myheight][c]{0cm}{} $(2,8)$ & $X' \ratmap X \ni \frac{1}{7} (1,2,5)$ & $(7,3)$ & $X' \ratmap X \ni \frac{1}{8} (1,3,5)$ \\
\hline
\end{tabular}
\end{center}

\begin{center}
\begin{flushleft}
No. 57: $X_{14} \subset \mbP (1,2,3,7,2)$, $(A^3) = 1/6$. \nopagebreak \\
Eq: $w^2 z t + w f_{12} + g_{14}, z^4 \in f_{12}$.
\end{flushleft} \nopagebreak
\begin{tabular}{|p{105pt}|p{140pt}|p{110pt}|}
\hline
\parbox[c][\myheight][c]{0cm}{} $\msp_1 \msp_4 = \frac{1}{2} (1,1,1)$ & $B^3 < 0$, $\{x,z,w'\}$, $3B+E$ & \\
\hline
\parbox[c][\myheight][c]{0cm}{} $\msp_2 = \frac{1}{3} (1,1,2)$ & $B^3 = 0$, $\{x,y,w\}$, $B$ & \\
\hline
\parbox[c][\myheight][c]{0cm}{} $\msp_4 = cA/2$ & $z t + h_{12} (x,y)/ \mbZ_2 (1,1,1,0)$ & $\wt = \frac{1}{2} (a,b,1,2)$ \\
\hline
\end{tabular}
\nopagebreak
\begin{tabular}{|p{67pt}|p{104.5pt}||p{67pt}|p{104.5pt}|}
\hline
\parbox[c][\myheight][c]{0cm}{} $(1,11), (11,1)$ & none & $(3,9), (9,3)$ & $X' \ratmap X \ni \frac{1}{5} (1,2,3)$ \\
\hline
\parbox[c][\myheight][c]{0cm}{} $(5,7), (7,5)$ & $X' \ratmap X \ni \frac{1}{9} (1,2,7)$ & & \\
\hline
\end{tabular}
\end{center}

\begin{center}
\begin{flushleft}
No. 61: $X_{15} \subset \mbP (1,1,5,6,3)$, $(A^3) = 1/6$. \nopagebreak \\
Eq: $w^4 x_0^3 + w^3 x_0^2 f_4 + w^2 x_0 f_8 + w f_{12} + g_{15}$.
\end{flushleft} \nopagebreak
\begin{tabular}{|p{105pt}|p{140pt}|p{110pt}|}
\hline
\parbox[c][\myheight][c]{0cm}{} $\msp_3 = \frac{1}{6} (1,1,5)$ & & Q.I. \\
\hline
\parbox[c][\myheight][c]{0cm}{} $\msp_4 = cD/3$ & & $X' \ratmap X \ni \frac{1}{11} (1,5,6)$ \\
\hline
\end{tabular}
\begin{flushleft}
\end{flushleft}
\end{center}

\begin{center}
\begin{flushleft}
No. 62: $X_{15} \subset \mbP (1,3,4,5,3)$, $(A^3) = 1/12$. \nopagebreak \\
Eq: $w^3 y^2 + w^2 y f_6 + w f_{12} + g_{15}$.
\end{flushleft} \nopagebreak
\begin{tabular}{|p{105pt}|p{140pt}|p{110pt}|}
\hline
\parbox[c][\myheight][c]{0cm}{} $\msp_1 \msp_4 = 3 \times \frac{1}{3} (1,1,2)$ & $B^3 < 0$ & \\
\hline
\parbox[c][\myheight][c]{0cm}{} $\msp_2 = \frac{1}{4} (1,1,3)$ & $B^3 = 0$, $\{x,y,w\}$, $B$ & \\
\hline
\parbox[c][\myheight][c]{0cm}{} $\msp_4 = cD/3$ & & $X' \ratmap X \ni \frac{1}{9} (1,4,5)$ \\
\hline
\end{tabular}
\begin{flushleft}
\end{flushleft}
\end{center}

\begin{center}
\begin{flushleft}
No. 63: $X_{15} \subset \mbP (1,3,4,5,3)$, $(A^3) = 1/12$. \nopagebreak \\
Eq: $w^2 z t + w f_{12} + g_{15}, z^3 \in f_{12}$.
\end{flushleft} \nopagebreak
\begin{tabular}{|p{105pt}|p{140pt}|p{110pt}|}
\hline
\parbox[c][\myheight][c]{0cm}{} $\msp_1 \msp_4 = \frac{1}{3} (1,1,2)$ & $B^3 < 0$ & \\
\hline
\parbox[c][\myheight][c]{0cm}{} $\msp_2 = \frac{1}{4} (1,1,3)$ & $B^3 = 0$, $\{x,y,w\}$, $B$ & \\
\hline
\parbox[c][\myheight][c]{0cm}{} $\msp_4 = cA/3$ & $z t + h_{12} (x,y) / \mbZ_3 (1,2,1,0)$ & $\wt =  \frac{1}{3} (a,b,1,3)$ \\
\hline
\end{tabular}
\nopagebreak
\begin{tabular}{|p{67pt}|p{104.5pt}||p{67pt}|p{104.5pt}|}
\hline
\parbox[c][\myheight][c]{0cm}{} $(1,11)$ & none & $(4,8)$ & $X' \ratmap X \ni \frac{1}{7} (1,1,2)$ \\
\hline
\parbox[c][\myheight][c]{0cm}{} $(7,5)$ & $X' \ratmap X \ni \frac{1}{5} (1,1,3)$ & $(10,2)$ & B.I. \\
\hline
\end{tabular}
\end{center}


\end{document}